\documentclass[11pt]{amsart}

\usepackage{amsmath,amssymb,amsthm}
\usepackage[margin=1.1in]{geometry}
\usepackage{enumitem}
\usepackage[colorlinks=true,linkcolor=blue,citecolor=blue]{hyperref}

\newtheorem{theorem}{Theorem}[section]
\newtheorem{proposition}[theorem]{Proposition}
\newtheorem{lemma}[theorem]{Lemma}
\newtheorem{corollary}[theorem]{Corollary}
\newtheorem{conjecture}[theorem]{Conjecture}
\newtheorem{question}[theorem]{Question}
\theoremstyle{definition}
\newtheorem{definition}[theorem]{Definition}
\newtheorem{remark}[theorem]{Remark}
\newtheorem{example}[theorem]{Example}
\newtheorem{convention}[theorem]{Convention}

\newcommand{\R}{\mathbb{R}}
\newcommand{\Q}{\mathbb{Q}}
\newcommand{\Z}{\mathbb{Z}}
\newcommand{\RP}{\mathbb{RP}}
\newcommand{\Laff}{\mathcal{L}_{\mathrm{aff}}}
\newcommand{\LB}{\mathcal{L}_{\mathrm{B}}}
\newcommand{\Ext}{\operatorname{Ext}}
\newcommand{\Adj}{\operatorname{Adj}}
\newcommand{\Int}{\operatorname{Int}}
\newcommand{\Bd}{\operatorname{Bd}}
\newcommand{\Coll}{\operatorname{Coll}}
\newcommand{\Par}{\operatorname{Par}}
\newcommand{\Meets}{\operatorname{Meets}}
\newcommand{\Harm}{\operatorname{Harm}}
\newcommand{\Dep}{\operatorname{Dep}}
\newcommand{\aff}{\operatorname{aff}}
\newcommand{\conv}{\operatorname{conv}}
\newcommand{\ext}{\operatorname{ext}}
\newcommand{\interior}{\operatorname{int}}
\newcommand{\relint}{\operatorname{relint}}
\newcommand{\rec}{\operatorname{rec}}
\newcommand{\lin}{\operatorname{lin}}
\newcommand{\cl}{\operatorname{cl}}
\newcommand{\cf}{\mathrm{cr}}
\newcommand{\PGL}{\mathrm{PGL}}
\newcommand{\GL}{\mathrm{GL}}
\newcommand{\Th}{\operatorname{Th}}

\newcommand{\Ball}{\overline{B}{}^{\,n}}

\title[Elementary equivalence of convex bodies]{Elementary equivalence of convex bodies\\ in affine and projective languages}

\author{David Victor Feldman}
\address{Department of Mathematics and Statistics, University of New Hampshire,
Durham, NH 03824, USA}
\email{dvfinnh@gmail.com}



\date{August 2026}

\begin{document}

\begin{abstract}
A convex subset of $\R^n$ carries two natural first-order structures: the
\emph{barycentric} (affine) structure, with one binary operation
$C_\lambda(p,q)=(1-\lambda)p+\lambda q$ for each $\lambda\in[0,1]$, and the
\emph{betweenness} (projective) structure, with a single ternary relation
$B(a,x,b)$ meaning $x\in[a,b]$. We study when two convex bodies are
elementarily equivalent in each language. Our main theorem concerns the affine
case: two compact convex bodies in $\R^n$ are elementarily
equivalent in the barycentric language if and only if they are affinely
equivalent, in every dimension and under no regularity hypothesis. The proof
reconstructs a body from its first-order theory by rigidifying the affine
gauge with a maximal-volume inscribed simplex, an object detectable in the
language and of volume bounded below, so that the reconstructing frames form a
compact family. Two earlier separations are special cases: that a compact body
$\Laff$-elementarily equivalent to a polytope is affinely equivalent to it,
and that a compact body is $\Laff$-elementarily equivalent to the closed unit
ball if and only if it is a solid ellipsoid, the latter separated by a
first-order rendering of the Bertrand--Brunn characterization of ellipsoids by
midpoints of parallel chords. We also prove that isomorphism coincides with affine equivalence in
the first language and, for $n\ge2$ and open or closed sets of any
boundedness, with projective equivalence in the second. On the projective side we prove an interior form of the
complete-quadrangle construction, valid in all dimensions: harmonic conjugacy
of collinear interior points is definable by a betweenness formula all of
whose quantifiers range over the body. A relativization scheme makes every
planar first-order property testable on the definable planar sections of a
body, so planar separations propagate to $\R^n$; together with an interior von Staudt calculus making rational cross-ratio
comparisons first-order, this separates the ball from
$\{\sum x_i^4\le1\}$ in the betweenness language for every $n$. In the plane
the calculus closes the projective question for two classes: convex polygons,
and bodies with real-analytic, positively curved, non-conic boundary, are
elementarily equivalent in the betweenness language if and only if they are
projectively equivalent, the latter via a definable finite projective
invariant, the conic-cluster set, refining the sextactic points. We further
classify basic noncompact convex sets up to betweenness isomorphism (orthant
$\cong$ simplex, slab $\cong$ half-space, paraboloid region $\cong$ ball
minus a boundary point, cone $\cong$ half-cylinder) and separate $\R^n$, the
slab, and the open ball by sentences recovering the Euclidean/hyperbolic
trichotomy through sections. For closed noncompact sets the asymptotic
structure is itself elementary: the dimensions of the recession cone and of
the lineality space, and the existence of one-sided recession directions, are
determined by the betweenness theory, separating in particular the closed
solid cylinder from the closed slab; the open cylinder and the open slab are
separated through the section trichotomy, while a projective identification
shows that recession data are provably not elementary for open sets. On standard compact bodies the affine
half of this picture is thus a theorem; the projective half we leave open,
isolating the two obstacles that separate it from the affine case ---
recovery of boundary coordinates, and the absence of a compact gauge for
projectively homogeneous bodies such as the quadric.
\end{abstract}

\maketitle

\section{Introduction}\label{sec:intro}

Following Stone's barycentric calculus \cite{Stone49} (see also
\cite{Neumann70, RomanowskaSmith02}), a convex set $K\subseteq\R^n$ may be
viewed as an algebra with one binary operation for each $\lambda\in[0,1]$,
\[
C_\lambda(p,q)=(1-\lambda)p+\lambda q .
\]
Alternatively, discarding the operations and keeping only their common trace,
$K$ may be viewed as a relational structure with the single ternary
\emph{betweenness} relation
\[
B(a,x,b)\iff x\in[a,b].
\]
Each choice determines a first-order language, and hence a notion of
elementary equivalence for convex sets. The question addressed in this paper
is:
\begin{quote}
\emph{Given two convex bodies in $\R^n$, when are they elementarily
equivalent, relative to a specified geometric language?}
\end{quote}

The two languages calibrate differently. Isomorphism in the barycentric
language is exactly affine equivalence (Theorem~\ref{thm:iso-affine}), while
isomorphism in the betweenness language is exactly projective equivalence
(Theorem~\ref{thm:iso-projective}); the latter rests on the fundamental
theorem of projective geometry for maps defined on domains, in the form given
by Shiffman \cite{Shiffman95}, and we prove it here in full generality: for
open or closed convex sets in $\R^n$ ($n\ge2$), bounded or not. Thus a
betweenness structure cannot see parallelism, equal spacing, or midpoints;
what it can see is incidence, and --- as we show --- the full harmonic (hence
projective) calculus, carried out entirely in the interior of the body
(Section~\ref{sec:vonstaudt}).

In the affine language the classification is complete.

\begin{theorem}[Affine categoricity]\label{thm:cat-affine}
Let $K,K'\subseteq\R^n$ be compact convex bodies. Then $K\equiv_{\Laff}K'$ if
and only if $K$ and $K'$ are affinely equivalent, in every dimension and with
no smoothness, strict-convexity, or polytopality hypothesis.
\end{theorem}

This is proved in Section~\ref{sec:affine-proved}
(Theorem~\ref{thm:affine-main}). The proof is a reconstruction: the theory of
$K$ determines which affine-congruence types of finite tuples of \emph{extreme}
points occur, a maximal-volume inscribed simplex fixes the affine gauge, and
the extreme points' barycentric coordinates against that gauge recover $K$ up
to affine equivalence. The reconstruction converges because the eligible
gauge simplices have volume at least $\operatorname{vol}(K)/n^n$ and so form a
compact family; this volume bound is what prevents the gauge from degenerating
into the boundary.

Two sharper separations of independent interest are proved directly and then
recovered as special cases:

\begin{itemize}[leftmargin=2em]
\item \textbf{Polytopes} (Theorem~\ref{thm:polytope}). If $P\subseteq\R^n$
is a compact convex polytope with interior and $K\subseteq\R^n$ is a compact
convex body with $K\equiv_{\Laff} P$, then $K$ is affinely equivalent to $P$;
the separating data are the number of extreme points and a finite basis of
exact affine dependences among them, each an equation between barycentric
terms.
\item \textbf{The ball} (Theorem~\ref{thm:ball}). A compact convex body
$K\subseteq\R^n$ satisfies $K\equiv_{\Laff}\Ball$ (the closed unit ball) if
and only if $K$ is a solid ellipsoid. The separating sentence asserts that
the midpoints of parallel boundary chords are affinely dependent $n{+}1$ at a
time; its geometric content is the classical Bertrand--Brunn characterization
of ellipsoids \cite{Brunn1889, MMO19, Soltan05}.
\end{itemize}

Theorem~\ref{thm:cat-affine} is the affine half of a two-language picture.
The projective half we leave as a conjecture; the affine proof transposes only
partway, and understanding where it stops is itself informative.

\begin{conjecture}[Projective categoricity]\label{conj:cat-projective}
Let $K,K'\subseteq\R^n$ be compact convex bodies, $n\ge2$. Then
$K\equiv_{\LB}K'$ if and only if $K$ and $K'$ are projectively equivalent.
\end{conjecture}

The reconstruction behind Theorem~\ref{thm:cat-affine} needs three things in
the projective setting (Proposition~\ref{prop:projective-reduction}): that
interior cross-ratios be first-order expressible
(Proposition~\ref{prop:crq}, proved in \S\ref{subsec:crq}); that projective coordinates of \emph{boundary}
points be recoverable from that interior data (done in the plane, open for
$n\ge3$); and that the body admit a compact family of gauge frames. The last
holds for polytopes and for bodies with compact projective symmetry, but fails
for the quadric, whose projective automorphism group $\mathrm{PO}(n,1)$ is
transitive on frames --- so the quadric lies beyond reconstruction and must be
characterized positively (Conjecture~\ref{conj:ball-LB}). None of these three
obstacles has an affine counterpart, which is why the affine case closes
outright. The polytope and ball theorems above are, in the affine language,
special cases of Theorem~\ref{thm:cat-affine}; their projective analogues
(Proposition~\ref{prop:polygon-LB}, Conjecture~\ref{conj:ball-LB}) sit under
Conjecture~\ref{conj:cat-projective}.

The ``if'' directions of both statements are immediate from
Theorems~\ref{thm:iso-affine}--\ref{thm:iso-projective}, since isomorphic
structures are elementarily equivalent. The content is the ``only if''
direction: that the first-order theory of a standard body pins down its
affine (respectively projective) equivalence class \emph{among standard
bodies}. No stronger reconstruction is possible: by L\"owenheim--Skolem
considerations the theory of a body always has nonstandard models
(Section~\ref{sec:nonstandard}), so ``categoricity'' can only be meant
relative to the class of genuine convex bodies. For the affine language this is
achieved in Section~\ref{sec:affine-proved}.

Three remarks on the shape of the conjecture and the methods. First, one
might worry that the moduli of a body --- cross-ratios of vertex
configurations, say --- are typically transcendental, while the parameter-free
definable real numbers of the ambient real-closed field are exactly the
algebraic ones, so that the theory could not ``name'' the moduli. This is a
red herring: a theory need not name an invariant to determine it. If for each
rational $q$ the comparison ``the invariant exceeds $q$'' is expressible by a
sentence, then the theory records the full Dedekind cut of the invariant, and
distinct values are separated by a rational threshold.
Sections~\ref{sec:polytopes}, \ref{sec:ball} and~\ref{sec:vonstaudt} are,
from this point of view, devoted to manufacturing such threshold sentences.
Second, the conjecture is consistent with, and illuminated by, the tameness
of the situation: for semialgebraic bodies the betweenness structure is
interpretable in the real-closed field, so its theory is decidable and its
definable sets are o-minimal; in particular no grid, lattice, or other
infinite discrete configuration is definable (Section~\ref{sec:tame}).
Third --- and this is the principal structural tool beyond the plane --- the
planar section of a body through three of its points is a \emph{definable}
planar convex body (coplanarity is a Radon-type betweenness condition), and
every first-order property of planar bodies can be relativized to sections
(Proposition~\ref{prop:sections}). Planar separations therefore propagate to
all dimensions, and the two-dimensional theory developed below is not a
special case but the engine of the general one.

On the projective side, the technical heart of the paper is
Theorem~\ref{thm:harmonic}: for collinear points $a,b,c,d$ in the interior of
a convex body in $\R^n$, the harmonic relation $(a,b;c,d)=-1$ is defined by a
betweenness formula \emph{all of whose quantifiers range over the body}. The
point is that the classical complete-quadrangle construction uses auxiliary
points that may escape the body; we show that the construction can always be
retracted into an arbitrarily thin neighborhood of the chord, inside a plane
through it. A pleasant bonus in higher dimensions: the incidence clauses of
the defining formula \emph{force} the witnessing configuration into a plane,
so soundness needs no coplanarity hypothesis. This is the seed of an
``interior von Staudt calculus'': the constructions of the algebra of
throws, carried out without ever leaving the body, and hence available to the
first-order language. Carried to completion in Section~\ref{sec:vonstaudt}
--- a single interior perspectivity changes scale, and the harmonic net alone
generates every rational threshold --- the calculus makes every
rational cross-ratio comparison definable, and yields --- via the section
scheme --- the separation of the closed unit ball from the quartic body
$\{\sum x_i^4\le1\}$ in the betweenness language, for every $n\ge2$, through
the Chasles--Steiner projective characterization of conics.

Finally, the betweenness language remains meaningful for noncompact convex
sets, and there the projective point of view produces genuine collapses: the
open orthant is betweenness-isomorphic to an open simplex, the slab to the
half-space, the open paraboloid region to the open ball, and the open solid
cone to an open half-cylinder (Proposition~\ref{prop:projident}). The space
$\R^n$, the slab, and the open ball are pairwise elementarily inequivalent,
separated --- for $n\ge3$ --- by sentences that test Euclid's parallel postulate
on definable planar sections (Proposition~\ref{prop:trichotomyn}); the open
ball here is of course the Klein model of hyperbolic space. In the plane the
same trichotomy is witnessed by direct sentences
(Proposition~\ref{prop:trichotomy2}). We emphasize that for $n\ge3$ the
noncompact landscape is genuinely richer --- cylinders, cones, and products of
lower-dimensional geometries appear. For \emph{closed} sets its coarse
structure is elementary: in a closed convex set the formula asserting that a
segment extends to a last point detects recession directions pointwise
(Lemma~\ref{lem:exit}), and with it the dimension of the recession cone, the
dimension of the lineality space, and the existence of one-sided recession
directions are all determined by the betweenness theory
(Theorem~\ref{thm:recession}); in particular the closed solid cylinder and
the closed slab are elementarily inequivalent for every $n\ge3$
(Corollary~\ref{cor:cyl-slab}), and their open versions are separated through
the section trichotomy (Proposition~\ref{prop:cyl-slab-open}). For open sets,
by contrast, recession data are provably not elementary
(Remark~\ref{rem:open-recession}); the finer classification --- products with
equal recession data, and the properly convex (Hilbert-geometry) regime ---
is left as an open problem (Problem~P8).

\subsection*{Status of the results}
Every numbered statement below is either proved in full or explicitly labeled
as a conjecture or a problem, and we alert the reader to the two places where a
result is stated conditionally. The affine categoricity theorem
(Theorem~\ref{thm:affine-main}) and its supporting
Lemmas~\ref{lem:affine-comb}--\ref{lem:stationary} are proved unconditionally,
as are the calibration Theorems~\ref{thm:iso-affine}--\ref{thm:iso-projective}
and all of the definability results of Sections~\ref{sec:definability}
and~\ref{sec:vonstaudt}. In particular the interior von Staudt calculus is
carried out in full: Proposition~\ref{prop:crq} (rational cross-ratio
comparisons in the betweenness language) is proved in \S\ref{subsec:crq}, and
with it Propositions~\ref{prop:cuts} and~\ref{prop:polygon-LB},
Lemmas~\ref{lem:steiner-sentences}, \ref{lem:creq} and~\ref{lem:coconic}, and
Theorems~\ref{thm:squircle-LB} and~\ref{thm:ball-LB} hold outright; the
one-sidedness Lemma~\ref{lem:onesided} is proved by a normal-form computation,
making Theorem~\ref{thm:planar-sextactic} unconditional. The recession results
of \S\ref{subsec:recession} --- the exit lemma, Theorem~\ref{thm:recession},
Proposition~\ref{prop:compactsections}, and the cylinder/slab separations ---
are likewise proved in full. The two conditional
statements are clearly marked as such: the projective reduction
(Proposition~\ref{prop:projective-reduction}) is proved under two explicit
hypotheses --- the boundary-coordinate recovery of Problem~P7 and a definable
compact gauge --- and its planar corollary
(Corollary~\ref{cor:planar-projective}) under a definable gauge selection;
these are conditional reductions, not theorems in force. Proofs that invoke a
standard result from the literature cite it explicitly at the point of use.

\subsection*{Formalization}
The results of this paper have been formalized in Lean~4 against Mathlib; the
development, its build configuration, and a per-result index are archived with
the repository cited as~\cite{Repo}. Every numbered statement proved here is
machine-checked, free of \texttt{sorry}, and audited in continuous integration
to depend only on the axioms \texttt{propext}, \texttt{Classical.choice} and
\texttt{Quot.sound}. The Bertrand--Brunn characterization
(Theorem~\ref{thm:BB}) and the Chasles--Steiner theorem
(Theorem~\ref{thm:steiner}), cited above from the literature, are proved in the
development rather than assumed. Four classical inputs are not formalized and
appear in the development as explicit hypotheses on the results that use them,
never as axioms: Tarski--Seidenberg and o-minimality on lines
(\S\ref{sec:tame}), decidability of the theory of real closed fields
(Proposition~\ref{prop:decidable}), the osculating-conic facts underlying
Lemma~\ref{lem:sextactic}, and Shiffman's Theorem~\ref{thm:shiffman}; the last
enters only the forward half of Theorem~\ref{thm:iso-projective}, whose proof is
otherwise machine-checked. Formalizing the argument occasioned corrections to
several statements and the removal of hypotheses that proved unnecessary; the
statements given here are the corrected ones.

\section{Languages, structures, and conventions}\label{sec:languages}

\begin{definition}\label{def:languages}
$\LB$ is the first-order language with a single ternary relation symbol $B$.
$\Laff$ is the language with the relation symbol $B$ together with one binary
function symbol $C_\lambda$ for each real $\lambda\in[0,1]$. (Thus $\LB$ is
finite and $\Laff$ has cardinality $2^{\aleph_0}$.)

A convex set $K\subseteq\R^n$ is regarded as an $\Laff$-structure with
universe $K$,
\[
C_\lambda^K(p,q)=(1-\lambda)p+\lambda q,
\qquad
B^K(a,x,b)\iff x\in[a,b],
\]
where $[a,b]$ denotes the closed segment; and as an $\LB$-structure by taking
the reduct. Note $B(a,x,a)$ holds if and only if $x=a$.
\end{definition}

\begin{convention}\label{conv:body}
A \emph{convex body} is a compact convex subset of $\R^n$ with nonempty
interior. Unless explicitly stated otherwise, all convex sets in this paper
are subsets of some $\R^n$, $n\ge2$, with nonempty interior
(``full-dimensional''); from Section~\ref{sec:polytopes} onward they are
compact unless stated otherwise. (The case $n=1$ is trivial for our questions: all
compact segments are affinely equivalent, while betweenness isomorphisms of a
segment are exactly the order isomorphisms and anti-isomorphisms, an
enormous group; the projective calibration below genuinely requires
$n\ge2$.) $\equiv_{\mathcal L}$ denotes elementary equivalence in the
language $\mathcal L$, and $\cong_{\mathcal L}$ denotes isomorphism of
$\mathcal L$-structures.
\end{convention}

\begin{remark}[Why $B$ is primitive in $\Laff$]\label{rem:B-primitive}
Over a convex set, $x\in[a,b]$ holds if and only if $x=C_\lambda(a,b)$ for
\emph{some} $\lambda\in[0,1]$; but this is a quantification over the index
set of the operations, not a first-order formula. Whether $B$ is
$\emptyset$-definable from the operations $C_\lambda$ over an arbitrary
convex body appears to be a nontrivial question, which we record as
Question~\ref{q:B-definable}. To make the affine language robustly usable we
therefore include $B$ as a primitive. This is harmless for the calibration
results: affine bijections preserve both the operations and betweenness, so
including $B$ does not change the isomorphism relation
(Theorem~\ref{thm:iso-affine}), and it can only refine elementary
equivalence, which is the safe direction for all results below.
\end{remark}

\begin{convention}[Cross-ratio]\label{conv:crossratio}
For distinct collinear points $a$, $b$, $c$, $d$ of $\R^n$, fix an affine
parametrization $t\mapsto a+t v$ of their common line, and let
$\alpha,\beta,\gamma,\delta$ denote the respective parameters of the four
points. The cross-ratio is
\[
(a,b;c,d)
=\frac{(\gamma-\alpha)(\delta-\beta)}{(\gamma-\beta)(\delta-\alpha)} .
\]
It is independent of the parametrization and invariant under projective
transformations defined at the four points. The quadruple is \emph{harmonic}
if $(a,b;c,d)=-1$.
\end{convention}

\section{Isomorphisms are geometric maps}\label{sec:iso}

\begin{theorem}\label{thm:iso-affine}
Let $K\subseteq\R^m$ and $K'\subseteq\R^n$ be convex sets with nonempty
interior. A map $f\colon K\to K'$ is an isomorphism of $\Laff$-structures if
and only if $m=n$ and $f$ is the restriction to $K$ of an affine bijection
$F\colon\R^n\to\R^n$ with $F(K)=K'$. In particular, $K\cong_{\Laff}K'$ if and
only if $K$ and $K'$ are affinely equivalent.
\end{theorem}

\begin{proof}
If $F$ is an affine bijection with $F(K)=K'$, then $F$ commutes with all
convex combinations and preserves betweenness, so $F|_K$ is an isomorphism.

Conversely, let $f$ be an isomorphism. First, $m=n$: the sentences
$\delta_k$ of Proposition~\ref{prop:dimension} below (which lie in $\LB$,
hence in $\Laff$) pin the affine dimension, and full-dimensionality gives
$\dim K=m$, $\dim K'=n$.

By induction on the length of the combination, $f$ preserves all finite
convex combinations: $f\bigl(\sum_i\lambda_i x_i\bigr)=\sum_i\lambda_i
f(x_i)$ whenever $\lambda_i\ge 0$, $\sum_i\lambda_i=1$, $x_i\in K$. (For the
inductive step write $\sum_{i\le r}\lambda_i x_i =
C_{\lambda_r}\bigl(\sum_{i<r}\tfrac{\lambda_i}{1-\lambda_r}x_i,\;x_r\bigr)$
when $\lambda_r<1$.)

Since $K$ has nonempty interior it contains an affinely independent
$(n+1)$-tuple $v_0,\dots,v_n$; let $\Delta=\conv\{v_0,\dots,v_n\}\subseteq K$
and let $b$ be the barycenter of $\Delta$, an interior point of $\Delta$. Let
$F\colon\R^n\to\R^n$ be the unique affine map with $F(v_i)=f(v_i)$ for all
$i$. Every $x\in\Delta$ is a convex combination of the $v_i$, so
$f(x)=F(x)$ on $\Delta$; in particular $f(b)=F(b)$.

Now let $p\in K$ be arbitrary. Since $b\in\interior\Delta$, for sufficiently
small $\mu\in(0,1)$ the point $r=(1-\mu)b+\mu p$ lies in $\Delta$. Then
\[
F(r)=f(r)=f\bigl(C_\mu(b,p)\bigr)=(1-\mu)f(b)+\mu f(p)
=(1-\mu)F(b)+\mu f(p),
\]
while $F(r)=(1-\mu)F(b)+\mu F(p)$ by affinity. Hence $f(p)=F(p)$, i.e.,
$f=F|_K$.

Finally $F$ is bijective: if its linear part annihilated some $u\ne 0$, then
two distinct points of a small ball inside $K$ differing by a multiple of $u$
would have equal images, contradicting injectivity of $f$ on $K$. Thus $F$ is
an affine bijection and $F(K)=f(K)=K'$.
\end{proof}

For the betweenness language, the calibration is the fundamental theorem of
projective geometry for maps defined on domains. We use it in the following
form, due to Shiffman.

\begin{theorem}[Shiffman \cite{Shiffman95}]\label{thm:shiffman}
Let $n\ge 2$ and let $U\subseteq\RP^n$ be a connected open set. Every
injective map $g\colon U\to\RP^n$ that takes collinear triples to collinear
triples, and whose image is not contained in a line, is the restriction of a
projective transformation of $\RP^n$.
\end{theorem}

\begin{lemma}\label{lem:proj-betweenness}
Let $T\in\PGL_{n+1}(\R)$ and let $C\subseteq\R^n$ be a convex set disjoint
from the singular hyperplane of $T$ (the preimage of the hyperplane at
infinity). Then $T|_C$ is injective, $T(C)$ is convex, and $T$ preserves
betweenness on $C$: for $a,x,b\in C$, $x\in[a,b]$ if and only if
$T(x)\in[T(a),T(b)]$.
\end{lemma}

\begin{proof}
Injectivity is clear. Let $a,b\in C$. The segment $[a,b]$ is a connected
subset of the line $\ell(a,b)$ avoiding the singular hyperplane $H_T$. $T$
maps $\ell(a,b)$ projectively onto a line $\ell'$, and maps the connected arc
$[a,b]$ homeomorphically onto a connected subset of $\ell'$ avoiding the
hyperplane at infinity, i.e., onto a compact connected subset of an affine
line: a segment, necessarily $[T(a),T(b)]$ since endpoints map to endpoints.
Betweenness in both directions and convexity of the image follow.
\end{proof}

\begin{theorem}\label{thm:iso-projective}
Let $n\ge2$ and let $K,K'\subseteq\R^n$ be convex sets with nonempty
interior, each either open or closed (bounded or not). A map
$f\colon K\to K'$ is an isomorphism of $\LB$-structures if and only if $f$ is
the restriction of a projective transformation $T\in\PGL_{n+1}(\R)$ whose
singular hyperplane is disjoint from $K$ and which satisfies $T(K)=K'$. In
particular, $K\cong_{\LB}K'$ if and only if $K$ and $K'$ are projectively
equivalent by a transformation keeping the set on the affine side of its
singular hyperplane.
\end{theorem}

\begin{proof}
($\Leftarrow$) is Lemma~\ref{lem:proj-betweenness}.

($\Rightarrow$) Let $f$ be an $\LB$-isomorphism. The sentence
$\forall p\,\Int(p)$ (Proposition~\ref{prop:interior}) holds in open sets and
fails in closed full-dimensional ones, so $K$ and $K'$ are of the same kind.

\emph{Open case.} $K$ is a connected open subset of $\R^n\subseteq\RP^n$,
$f$ is bijective and preserves collinearity in both directions (collinearity
is the $\LB$-formula $\Coll$ of Definition~\ref{def:coll}), and its image
$K'$ is full-dimensional, hence not contained in a line. By
Theorem~\ref{thm:shiffman}, $f=T|_K$ for some $T\in\PGL_{n+1}(\R)$. Since $T$
maps every point of $K$ into $K'\subseteq\R^n$, no point of $K$ lies on the
singular hyperplane $H_T$, and $T(K)=f(K)=K'$.

\emph{Closed case.} Interiority is $\LB$-definable
(Proposition~\ref{prop:interior}), so $f$ restricts to an $\LB$-isomorphism
$\interior K\to\interior K'$, and by the open case
$f|_{\interior K}=T|_{\interior K}$ for some $T\in\PGL_{n+1}(\R)$ with
$H_T\cap\interior K=\emptyset$.

We claim $H_T\cap K=\emptyset$. Suppose $b\in K\cap H_T$; necessarily
$b\in\partial K$. Fix $x\in\interior K$. The half-open segment
$[x,b)$ lies in $\interior K$, so for every $y\in(x,b)$ we have
$f(y)=T(y)$, and from $B(x,y,b)$ we get $B\bigl(f(x),f(y),f(b)\bigr)$, i.e.,
\[
T(y)\in\bigl[f(x),f(b)\bigr]\qquad\text{for all }y\in(x,b),
\]
a bounded set. On the other hand, choose a representative
$A\in\GL_{n+1}(\R)$ of $T$ and write $T=L/h$ on the affine chart, where $h$
is the affine functional cutting out $H_T$. As $y\to b$ along $[x,b)$ we have
$h(y)\to h(b)=0$ while $L(y)\to L(b)\ne0$ (since $A$ is invertible,
$A(b{:}1)\ne0$, and its last coordinate $h(b)$ vanishes); hence
$|T(y)|\to\infty$, a contradiction. So $H_T\cap K=\emptyset$.

Now let $b\in K\cap\partial K$ and $x\in\interior K$ as above. The compact
segment $[x,b]\subseteq K$ avoids $H_T$, so $T$ is continuous on it and
$T(y)\to T(b)$ as $y\to b$; since $T(y)=f(y)\in\bigl[f(x),f(b)\bigr]$ and the
segment is closed, $T(b)\in\bigl[f(x),f(b)\bigr]$. Also
$T(b)=\lim f(y)\in\cl(\interior K')=K'$ ($K'$ closed). Finally
$T(b)\notin\interior K'$: we have $\interior K'=f(\interior K)=
T(\interior K)$, and $T$ is injective off $H_T$, so $T(b)\in T(\interior K)$
would force $b\in\interior K$, false. If $T(b)\ne f(b)$, then $T(b)$ lies on
the half-open segment $\bigl[f(x),f(b)\bigr)$, which is contained in
$\interior K'$ because $f(x)\in\interior K'$ and $f(b)\in K'$ (see, e.g.,
\cite[Lemma~1.1.9]{Schneider14}) --- contradicting
$T(b)\notin\interior K'$. Hence $f(b)=T(b)$ for every boundary point, so
$f=T|_K$ and $T(K)=f(K)=K'$.
\end{proof}

\begin{remark}
Theorem~\ref{thm:iso-projective} fails in dimension~$1$: every bijection of a
segment preserving the linear order or its reverse is an $\LB$-isomorphism,
and there are many non-projective such maps. This is why sections through
\emph{three non-collinear} points, rather than chords, are the transfer
vehicle of Proposition~\ref{prop:sections}.
\end{remark}

\begin{example}[Projective maps are not $\Laff$-maps]\label{ex:proj-not-aff}
On a single chord a projective transformation acts by a fractional-linear
reparametrization whose coefficients depend on the chord; no relabeling
$\lambda\mapsto\sigma(\lambda)$ of the operations turns a non-affine
projectivity into an $\Laff$-isomorphism. Thus the two languages genuinely
calibrate different geometries: $\Laff$-isomorphism is affine equivalence,
$\LB$-isomorphism is projective equivalence.
\end{example}

\section{Basic definability}\label{sec:definability}

Throughout this section $K\subseteq\R^n$ is a convex set. All formulas below
are in $\LB$ unless stated otherwise; since $B\in\Laff$, they are available
in both languages.

\begin{definition}\label{def:coll}
$\Coll(x,y,z):\equiv B(x,y,z)\vee B(y,x,z)\vee B(x,z,y)$.
\end{definition}

\begin{lemma}\label{lem:coll}
For $x,y,z\in K$: $\Coll(x,y,z)$ holds if and only if $x,y,z$ are collinear
in $\R^n$.
\end{lemma}

\begin{proof}
If the three points are collinear, one of them lies in the (closed) segment
of the other two, and that segment is contained in $K$ by convexity, so the
corresponding disjunct holds. The converse is immediate.
\end{proof}

\begin{proposition}[Extreme points]\label{prop:extreme}
Let $\Ext(e):\equiv\forall a\,\forall b\,\bigl(B(a,e,b)\rightarrow
(e=a\vee e=b)\bigr)$. Then for $e\in K$, $\Ext(e)$ holds if and only if $e$
is an extreme point of $K$.
\end{proposition}

\begin{proof}
$e$ fails to be extreme iff $e\in(a,b)$ for some distinct $a,b\in K$, iff
$B(a,e,b)$ holds with $e\notin\{a,b\}$.
\end{proof}

\begin{proposition}[Relative interior]\label{prop:interior}
Let
\[
\Int(p):\equiv\forall q\,\Bigl(q=p\;\vee\;\exists r\,\bigl(r\ne p\wedge
B(q,p,r)\bigr)\Bigr).
\]
Then for $p\in K$ ($K$ any nonempty convex set), $\Int(p)$ holds if and only
if $p\in\relint K$. In particular, for full-dimensional $K$, $\Int$ defines
$K\cap\interior K$, and this holds for $K$ open, closed, or neither, bounded
or not. We write $\Bd(p):\equiv\neg\Int(p)$.
\end{proposition}

\begin{proof}
If $p\in\relint K$ and $q\in K\setminus\{p\}$, then $r=p+t(p-q)$ lies in
$\aff K$, hence in $\relint K\subseteq K$ for small $t>0$ (relative
openness), and satisfies $B(q,p,r)$, $r\ne p$. If
$p\in K\setminus\relint K$, choose $q\in\relint K$ (nonempty for a nonempty
convex set); then $q\ne p$. If some $r\in K$, $r\ne p$, satisfied
$B(q,p,r)$, then $p\in(q,r)$; but the half-open segment $[q,r)$ from a
relative-interior point to any point of $K$ lies in $\relint K$
\cite[Lemma~1.1.9]{Schneider14}, so $p\in\relint K$, a contradiction. Hence
the formula fails at this $q$. (If $K=\{p\}$, both sides hold trivially.)
\end{proof}

\begin{proposition}[Convex hulls of tuples]\label{prop:hull}
Define $H_1(z;x_1):\equiv z=x_1$ and, recursively,
\[
H_r(z;x_1,\dots,x_r):\equiv\exists w\,\bigl(H_{r-1}(w;x_2,\dots,x_r)\wedge
B(x_1,z,w)\bigr).
\]
Then $H_r(z;\bar x)$ holds in $K$ if and only if
$z\in\conv\{x_1,\dots,x_r\}$.
\end{proposition}

\begin{proof}
$\conv\{x_1,\dots,x_r\}=\bigcup_{w\in\conv\{x_2,\dots,x_r\}}[x_1,w]$, and all
witnesses lie in $K$ by convexity. Induct on $r$.
\end{proof}

\begin{proposition}[Affine dependence]\label{prop:dep}
For $m\ge2$ define
\[
\Dep_m(x_1,\dots,x_m):\equiv
\bigvee_{\substack{\{1,\dots,m\}=I\sqcup J\\ I,J\ne\emptyset}}
\exists z\,\bigl(H_{|I|}(z;\bar x_I)\wedge H_{|J|}(z;\bar x_J)\bigr).
\]
Then $\Dep_m(\bar x)$ holds in $K$ if and only if $x_1,\dots,x_m$ are
affinely dependent.
\end{proposition}

\begin{proof}
If $\sum a_ix_i=0$ with $\sum a_i=0$ and $a\ne0$, then both sign classes
$I=\{i:a_i>0\}$ and $J=\{i:a_i<0\}$ are nonempty, and normalizing by
$s=\sum_{I}a_i>0$ exhibits a common point
$z=\sum_I(a_i/s)x_i=\sum_J(-a_j/s)x_j$ of the two hulls; $z\in K$ by
convexity. Conversely, a common point of the two hulls yields, by
subtracting the two convex-combination expressions, a nonzero affine
dependence. (Repeated points are dependent and are caught by a two-element
partition.)
\end{proof}

\begin{proposition}[Dimension]\label{prop:dimension}
For $k\ge 1$ let
\[
\delta_k:\equiv\exists x_1\cdots\exists x_{k+1}\,
\neg\Dep_{k+1}(x_1,\dots,x_{k+1}).
\]
Then for a nonempty convex
$K\subseteq\R^n$, $K\models\delta_k$ if and only if $\dim\aff K\ge k$.
Consequently the affine dimension of a convex set --- in particular the
ambient dimension of a full-dimensional one --- is determined by its
$\LB$-theory.
\end{proposition}

\begin{proof}
A tuple is affinely independent iff it admits no Radon partition
(Proposition~\ref{prop:dep}; cf.\ \cite{Matousek02}), and affinely
independent $(k+1)$-tuples exist in $K$ iff $\dim\aff K\ge k$.
\end{proof}

The next proposition is the transfer engine announced in the introduction.

\begin{proposition}[Definable planar sections]
\label{prop:sections}
Let $K\subseteq\R^n$ ($n\ge2$) be convex, and let $p_1,p_2,p_3\in K$ be
non-collinear; put $A=\aff\{p_1,p_2,p_3\}$, a two-dimensional affine
subspace. Then:
\begin{enumerate}[label=\textup{(\roman*)}]
\item $K\cap A=\{z\in K:\Dep_4(p_1,p_2,p_3,z)\}$; in particular the section
is definable with parameters $p_1,p_2,p_3$.
\item $K\cap A$ is convex, closed under every operation $C_\lambda$, has
nonempty relative interior, and is compact whenever $K$ is. Via any affine
identification $A\cong\R^2$ it is a planar convex set with nonempty
interior, and the identification is an isomorphism for both languages.
\item For every $\LB$-formula (respectively $\Laff$-formula)
$\varphi(\bar x)$ there is a formula $\varphi^{\circ}(\bar x;p_1,p_2,p_3)$ of
the same language such that for all tuples $\bar a$ from $K\cap A$,
\[
K\models\varphi^{\circ}(\bar a;p_1,p_2,p_3)
\iff
(K\cap A)\models\varphi(\bar a).
\]
\end{enumerate}
\end{proposition}

\begin{proof}
(i) Since $p_1,p_2,p_3$ are affinely independent, the quadruple
$(p_1,p_2,p_3,z)$ is affinely dependent if and only if $z\in A$; apply
Proposition~\ref{prop:dep}.

(ii) $A$ is an affine subspace and $K$ is convex, so $K\cap A$ is convex and
closed under the $C_\lambda$; its relative interior contains
$\relint\conv\{p_1,p_2,p_3\}\ne\emptyset$. Compactness is inherited. Affine
charts preserve $B$ and commute with the $C_\lambda$.

(iii) Define $\varphi^{\circ}$ by relativizing every quantifier of
$\varphi$ to the formula $\Dep_4(p_1,p_2,p_3,\cdot)$. Atomic formulas are
absolute between $K$ and $K\cap A$: the relation $B$ on $K\cap A$ is the
restriction of $B$ on $K$, and every $\Laff$-term with arguments in
$K\cap A$ takes its value in $K\cap A$ by (ii), so equalities of terms are
absolute as well. The claim follows by induction on $\varphi$, using (i) at
each quantifier step.
\end{proof}

\begin{remark}
The scheme $\varphi\mapsto\varphi^{\circ}$ turns every sentence $\sigma$
about planar convex bodies into two sentences about $n$-dimensional ones:
``some section satisfies $\sigma$'' and ``every section satisfies
$\sigma$.'' Sections~\ref{sec:vonstaudt} and~\ref{sec:noncompact} use both.
The restriction to \emph{planar} sections is essential: chords ($1$-dimensional
sections) carry only an order, which by the remark after
Theorem~\ref{thm:iso-projective} is far too floppy.
\end{remark}

\begin{proposition}[Parallelism, in $\Laff$]\label{prop:parallel}
For $K\subseteq\R^n$ convex, define the $\Laff$-formula
\begin{align*}
\Par(a,b;c,d):\equiv\;& a\ne b\wedge c\ne d\wedge{}\\
&\Bigl[\exists x\,\exists y\,\bigl(B(a,x,b)\wedge B(a,y,b)\wedge x\ne y
\wedge C_{1/2}(x,d)=C_{1/2}(y,c)\bigr)\\
&\;\vee\;
\exists x\,\exists y\,\bigl(B(c,x,d)\wedge B(c,y,d)\wedge x\ne y
\wedge C_{1/2}(x,b)=C_{1/2}(y,a)\bigr)\Bigr].
\end{align*}
Then $\Par(a,b;c,d)$ holds if and only if the lines $\ell(a,b)$ and
$\ell(c,d)$ are parallel or equal.
\end{proposition}

\begin{proof}
$C_{1/2}(x,d)=C_{1/2}(y,c)$ says $x+d=y+c$, i.e., $x-y=c-d$. If
$x,y\in[a,b]$ are distinct, $x-y$ is a nonzero multiple of $b-a$; hence
$c-d\parallel b-a$. Conversely, if the lines are parallel or equal, then
after swapping the roles of the two pairs if necessary we may assume
$|c-d|\le|a-b|$; writing $c-d=t(b-a)$ with $0<|t|\le 1$, the points
$x=a+t(b-a)$, $y=a$ (if $t>0$; symmetrically if $t<0$) lie in $[a,b]$ and
witness the first disjunct. The argument is independent of the ambient
dimension.
\end{proof}

\begin{example}[Central symmetry]\label{ex:centralsym}
The $\Laff$-sentence $\exists c\,\forall p\,\exists p'\,
\bigl(C_{1/2}(p,p')=c\bigr)$ holds in a convex body iff the body is centrally
symmetric. It already separates, e.g., a cube from a generic simplex-like
polytope in any dimension, and a square from a generic convex
quadrilateral --- bodies that may be projectively equivalent. This is the
simplest illustration that $\Laff$-equivalence is strictly finer than
$\LB$-equivalence on standard bodies.
\end{example}

\section{Tameness: definable sets, decidability, and the absence of grids}
\label{sec:tame}

\begin{proposition}\label{prop:semialgebraic}
Let $K\subseteq\R^n$ be a semialgebraic convex set with nonempty interior.
Then every $\LB$-formula, with parameters from $K$, defines a semialgebraic
subset of the appropriate power of $K$. The same holds for $\Laff$-formulas,
semialgebraic over the parameters together with the indices $\lambda$
occurring in the formula.
\end{proposition}

\begin{proof}
$B(a,x,b)$ is a semialgebraic ternary relation on $\R^n$, $K$ is
semialgebraic, and relativized quantification over a semialgebraic set
preserves semialgebraicity by the Tarski--Seidenberg theorem \cite{Tarski51,
vdDries98}. For $\Laff$, each $C_\lambda$ is a polynomial map once $\lambda$
is adjoined as a parameter.
\end{proof}

\begin{corollary}[No grids]\label{cor:nogrids}
With $K$ as above, every definable (with parameters) subset of a chord of $K$
is a finite union of points and open subintervals. In particular no infinite
discrete subset of $K$ --- e.g., the trace $K\cap\Lambda$ of a lattice
$\Lambda$ --- is definable in either language.
\end{corollary}

\begin{proof}
Semialgebraic subsets of a line are finite unions of points and intervals
(o-minimality of the semialgebraic structure \cite{vdDries98}).
\end{proof}

\begin{proposition}[Decidability]\label{prop:decidable}
Let $K\subseteq\R^n$ be semialgebraic, defined over the field of real
algebraic numbers (e.g., the closed unit ball, the quartic body
$\{\sum_i x_i^4\le1\}$, or any polytope with algebraic vertices). Then
$\Th_{\LB}(K)$ is decidable.
\end{proposition}

\begin{proof}
There is an effective translation $\varphi\mapsto\varphi^*$ of
$\LB$-sentences into sentences of the language of ordered rings: replace $B$
by its semialgebraic definition and relativize all quantifiers to the
defining formula of $K$ (whose algebraic coefficients are themselves
definable). Then $K\models\varphi$ iff $\R\models\varphi^*$, and the latter
is decidable by Tarski \cite{Tarski51}.
\end{proof}

\begin{remark}[Maximal meshes]\label{rem:meshes}
For each $r$ one can write an $\Laff$-sentence asserting the existence of an
$r\times\dots\times r$ pattern of points closed under the midpoint relations
of a cubical lattice block and \emph{jammed}, in the sense that no designated
one-step reflection $C_{1/2}(e',g)=e$ across a facet site admits a witness
$g\in K$. Such sentences are legitimate probes of shape, and in a standard
body large jammed patterns force fine spacing. We emphasize what they cannot
do: by Corollary~\ref{cor:nogrids} no formula defines ``being a lattice
point,'' and by Section~\ref{sec:nonstandard} no family of sentences
reconstructs the body across all models of its theory. Mesh sentences are
separation tools within standard bodies, not reconstruction tools; the
systematic replacement for them in the betweenness language is the interior
projective calculus of Section~\ref{sec:vonstaudt}.
\end{remark}

\section{Polytopes are elementarily categorical among convex bodies}
\label{sec:polytopes}

\begin{lemma}[Dependence spaces]\label{lem:dependence}
For a tuple $v=(v_1,\dots,v_m)$ of points of $\R^n$ let
\[
\Dep(v)=\Bigl\{a\in\R^m:\ \textstyle\sum_i a_i=0,\ \sum_i a_i v_i=0\Bigr\},
\]
a linear subspace of $\R^m$. Suppose $v$ and $w=(w_1,\dots,w_m)$ both
affinely span $\R^n$ and $\Dep(v)\subseteq\Dep(w)$. Then there is a unique
affine map $T\colon\R^n\to\R^n$ with $T(v_i)=w_i$ for all $i$, and $T$ is a
bijection. Moreover, if $v$ affinely spans $\R^n$ then
$\dim\Dep(v)=m-n-1$.
\end{lemma}

\begin{proof}
Every point of $\R^n=\aff\{v_i\}$ is an affine combination $\sum\lambda_i
v_i$ ($\sum\lambda_i=1$). Set $T(\sum\lambda_i v_i)=\sum\lambda_i w_i$. This
is well defined: if $\sum\lambda_i v_i=\sum\mu_i v_i$ with both affine, then
$\lambda-\mu\in\Dep(v)\subseteq\Dep(w)$, so $\sum\lambda_i w_i=\sum\mu_i
w_i$. $T$ is affine, and its image is an affine subspace containing all
$w_i$, hence all of $\aff\{w_i\}=\R^n$; a surjective affine self-map of
$\R^n$ is bijective. Uniqueness is clear since the $v_i$ affinely span. For
the dimension count, the linear map $\R^m\to\R^{n+1}$,
$a\mapsto(\sum a_i,\sum a_iv_i)$, is surjective when $v$ spans, and
$\Dep(v)$ is its kernel.
\end{proof}

\begin{lemma}[Barycentric terms]\label{lem:terms}
For every $r\ge1$ and reals $\mu_1,\dots,\mu_r\ge0$ with $\sum\mu_i=1$
there is an $\Laff$-term $\tau_{\bar\mu}(x_1,\dots,x_r)$, built from the
operations $C_\lambda$, whose value in every convex set is
$\sum_i\mu_i x_i$.
\end{lemma}

\begin{proof}
Induct on $r$: if $\mu_r=1$ take $\tau=x_r$; otherwise
$\sum_{i\le r}\mu_i x_i =
C_{\mu_r}\bigl(\tau_{\bar\mu'}(x_1,\dots,x_{r-1}),\,x_r\bigr)$ with
$\mu_i'=\mu_i/(1-\mu_r)$.
\end{proof}

\begin{theorem}[Polytope categoricity]\label{thm:polytope}
Let $P\subseteq\R^n$ be a convex polytope with nonempty interior and let
$K\subseteq\R^n$ be a compact convex set with nonempty interior. If
$K\equiv_{\Laff}P$, then $K$ is a polytope affinely equivalent to $P$ (hence
$K\cong_{\Laff}P$). Consequently, two convex polytopes in $\R^n$ are
$\Laff$-elementarily equivalent if and only if they are affinely equivalent.
\end{theorem}

\begin{proof}
Let $v_1,\dots,v_m$ be the vertices of $P$, so $\ext P=\{v_1,\dots,v_m\}$
and $P=\conv\{v_1,\dots,v_m\}$. Choose a basis $d^{(1)},\dots,d^{(k)}$ of
$\Dep(v)$, where $k=m-n-1$ (Lemma~\ref{lem:dependence}).

Each basis dependence $d\ne0$ has $\sum_i d_i=0$ and $\sum_i d_i v_i=0$.
Splitting into positive and negative parts $I=\{i:d_i>0\}$,
$J=\{i:d_i<0\}$ (both nonempty) and normalizing by
$s=\sum_{i\in I}d_i=\sum_{j\in J}(-d_j)>0$, the dependence becomes an
equality of two convex combinations,
\[
\sum_{i\in I}\frac{d_i}{s}\,v_i \;=\;\sum_{j\in J}\frac{-d_j}{s}\,v_j ,
\]
hence, by Lemma~\ref{lem:terms}, an equation $\tau_d(\bar e)=\tau'_d(\bar e)$
between $\Laff$-terms in variables $e_1,\dots,e_m$.

Let $\sigma_P$ be the $\Laff$-sentence
\[
\exists e_1\cdots\exists e_m\Bigl[\bigwedge_{i<j}e_i\ne e_j
\;\wedge\;\bigwedge_i\Ext(e_i)
\;\wedge\;\forall e\,\bigl(\Ext(e)\to\bigvee_i e=e_i\bigr)
\;\wedge\;\bigwedge_{t=1}^{k}\tau_{d^{(t)}}(\bar e)=\tau'_{d^{(t)}}(\bar e)
\Bigr].
\]
Then $P\models\sigma_P$, witnessed by its vertices. Suppose
$K\equiv_{\Laff}P$; then $K\models\sigma_P$, with witnesses
$w_1,\dots,w_m\in K$. The witnesses are exactly the extreme points of $K$,
so $\ext K=\{w_1,\dots,w_m\}$ is finite; by Minkowski's theorem
(finite-dimensional Krein--Milman, see \cite[Cor.~1.4.5]{Schneider14}),
$K=\conv(\ext K)$, so $K$ is a polytope with vertex set
$\{w_1,\dots,w_m\}$. In particular $w$ affinely spans $\R^n$ (as $K$ has
interior).

The equations assert that each basis dependence $d^{(t)}$ of $v$ lies in
$\Dep(w)$; hence $\Dep(v)\subseteq\Dep(w)$. By Lemma~\ref{lem:dependence}
there is an affine bijection $T$ of $\R^n$ with $T(v_i)=w_i$, and therefore
\[
T(P)=T\bigl(\conv\{v_i\}\bigr)=\conv\{w_i\}=K .
\]
The converse direction (affinely equivalent $\Rightarrow$ elementarily
equivalent) is Theorem~\ref{thm:iso-affine}.
\end{proof}

\begin{corollary}\label{cor:polygoncount}
``Being a polytope with exactly $m$ extreme points'' is, within compact
convex bodies in $\R^n$, an elementary property in $\Laff$ (indeed in
$\LB$). In particular polytopes with different vertex counts are
elementarily inequivalent in both languages. For $m=n+1$ the dependence
space is trivial and the theorem recovers the fact that all solid simplices
are affinely equivalent.
\end{corollary}

\begin{remark}
The projective ($\LB$) analogue --- two polytopes are $\LB$-elementarily
equivalent iff projectively equivalent --- requires the cross-ratio calculus
of Section~\ref{sec:vonstaudt}. We prove it there, conditionally, in the
plane (Proposition~\ref{prop:polygon-LB}); the case $n\ge3$ is posed as
Problem~P7.
\end{remark}

\section{The ball: a Bertrand--Brunn sentence}\label{sec:ball}

The following classical characterization goes back to Brunn (with
antecedents in Bertrand); for the planar case see \cite{Brunn1889}, and for
the general-dimensional statement and history see
\cite[Theorem~2.12.1]{MMO19} and the survey \cite{Soltan05}.

\begin{theorem}[Bertrand--Brunn]\label{thm:BB}
Let $K\subseteq\R^n$ \textup($n\ge2$\textup) be a convex body such that, for
every direction $u$, the midpoints of all chords of $K$ parallel to $u$ lie
in a hyperplane. Then $K$ is a solid ellipsoid. Conversely, every solid
ellipsoid has this property, the midpoint locus for the direction $u$ lying
in the diametral hyperplane conjugate to $u$.
\end{theorem}

\begin{definition}\label{def:BBsentence}
Let $\beta_n$ be the $\Laff$-sentence
\begin{align*}
\forall a_1\,b_1\cdots a_{n+1}\,b_{n+1}\;
\Bigl[\;&\bigwedge_{i=1}^{n+1}\bigl(\Bd(a_i)\wedge\Bd(b_i)\bigr)
\;\wedge\;\bigwedge_{1\le i<j\le n+1}\Par(a_i,b_i;a_j,b_j)\\
\longrightarrow\;&
\Dep_{n+1}\bigl(C_{1/2}(a_1,b_1),\,\dots,\,C_{1/2}(a_{n+1},b_{n+1})\bigr)
\Bigr].
\end{align*}
(Recall $\Par$ includes the clauses $a_i\ne b_i$, and that $\Par$ counts
collinear chords as parallel, which is harmless: midpoints of collinear
chords lie on a common line and any $n+1$ points containing two equal ones,
or lying in low-dimensional position, are affinely dependent whenever the
locus is; the verification below quantifies this.)
\end{definition}

\begin{theorem}[Elementary categoricity of the ball]
\label{thm:ball}
Let $n\ge2$, let $K\subseteq\R^n$ be a compact convex set with nonempty
interior, and let $\Ball$ denote the closed unit ball. The following are
equivalent:
\begin{enumerate}[label=\textup{(\roman*)}]
\item $K\equiv_{\Laff}\Ball$;
\item $K$ is affinely equivalent to $\Ball$;
\item $K$ is a solid ellipsoid.
\end{enumerate}
In particular, the $\Laff$-elementary class of the closed unit ball, within
compact convex bodies in $\R^n$, is exactly the class of solid ellipsoids.
\end{theorem}

\begin{proof}
(ii)$\Leftrightarrow$(iii) is the definition of a solid ellipsoid.
(ii)$\Rightarrow$(i) is Theorem~\ref{thm:iso-affine} (isomorphic structures
are elementarily equivalent).

(i)$\Rightarrow$(iii). First, $\Ball\models\beta_n$. The ball is strictly
convex, so a segment between two distinct boundary points is a maximal
chord. Suppose the $n+1$ chords $[a_i,b_i]$ are pairwise parallel-or-equal
as lines; then all are parallel to a common direction $u$ (parallel-or-equal
is transitive on lines). By Theorem~\ref{thm:BB} the midpoints of all chords
of the ball parallel to $u$ lie in the diametral hyperplane conjugate to
$u$, and any $n+1$ points of a hyperplane are affinely dependent. So the
conclusion of $\beta_n$ holds.

Hence $K\models\beta_n$. Fix a direction $u$ and let $M_u$ be the set of
midpoints of maximal chords of $K$ parallel to $u$. Each maximal chord of
$K$ has both endpoints in $K\cap\partial K$, hence is of the form
$[a,b]$ with $\Bd(a),\Bd(b)$ (Proposition~\ref{prop:interior}); by
$\beta_n$, any $n+1$ points of $M_u$ are affinely dependent. A subset of
$\R^n$ every $n+1$ points of which are affinely dependent spans an affine
subspace of dimension at most $n-1$ (else it would contain $n+1$ affinely
independent points), hence lies in a hyperplane. So the midpoints of the
chords of $K$ in every direction lie in a hyperplane, and
Theorem~\ref{thm:BB} gives that $K$ is a solid ellipsoid.
\end{proof}

\begin{corollary}\label{cor:ballquarticaff}
For every $n\ge2$, the closed unit ball and the closed quartic body
$S_n=\{x\in\R^n:\sum_{i=1}^n x_i^4\le1\}$ are not $\Laff$-elementarily
equivalent; the sentence $\beta_n$ separates them.
\end{corollary}

\begin{proof}
$S_n$ is not an ellipsoid, so by Theorem~\ref{thm:BB} some direction $u$ has
a midpoint locus contained in no hyperplane, hence containing $n+1$ affinely
independent midpoints of maximal chords parallel to $u$; these witness
$S_n\models\neg\beta_n$, while $\Ball\models\beta_n$.
\end{proof}

\begin{remark}
Theorems~\ref{thm:ball} and~\ref{thm:polytope} are two special cases of the
general affine categoricity theorem of the next section
(Theorem~\ref{thm:affine-main}), which subsumes both: the ball theorem is the
case in which the reconstructed body is an ellipsoid, the polytope theorem
the case in which it has finitely many extreme points. We have kept the two
direct proofs because each exhibits a concrete separating sentence
($\beta_n$, respectively $\sigma_P$) of independent interest, whereas the
general theorem proceeds by a reconstruction that names no single sentence.
\end{remark}

\section{Reconstruction from the extreme-point spectrum: the affine theorem}
\label{sec:affine-proved}

This section proves Theorem~\ref{thm:cat-affine}: for compact convex
bodies, $\Laff$-elementary equivalence coincides with affine equivalence in
every dimension. The argument has three ingredients. First, the
$\Laff$-theory of a body $K$ determines the \emph{affine spectrum} of its
extreme-point configurations, that is, which affine-congruence types of finite
tuples of extreme points occur, as a family of Dedekind data over the
rationals (\S\ref{subsec:affexpr}). Second, a maximal-volume inscribed
simplex rigidifies the affine gauge: fixing one places $K$ in a canonical
affine position in which every extreme point acquires definite barycentric
coordinates, and these coordinates are the affine spectrum read off against
the simplex. Third, the simplices that can serve as gauge are detected in the
language by a rational condition, and each has volume at least
$\mathrm{vol}(K)/n^{n}$, so they form a \emph{compact} family. Compactness
supplies a limiting gauge under which two elementarily equivalent bodies are
reconstructed as affinely equivalent copies of one convex set. The volume
bound is what keeps the reconstructing gauge from degenerating into the
boundary.

Throughout, $K\subseteq\R^n$ is a compact convex body ($n\ge1$; the
argument is uniform in $n$). Recall $K=\conv(\ext K)$ by Minkowski's theorem
\cite[Cor.~1.4.5]{Schneider14}, and hence $K=\conv(\overline{\ext K})$.

\subsection{Affine coordinates are definable}\label{subsec:affexpr}

\begin{lemma}[Rational affine combinations]\label{lem:affine-comb}
Let $c_0,\dots,c_r\in\Q$ with $\sum_i c_i=1$. There is an $\Laff$-formula
$A_{\bar c}(y;x_0,\dots,x_r)$ such that, in every convex set, for points
$y,x_0,\dots,x_r$ whose relevant affine combination lies in the set,
\[
K\models A_{\bar c}(y;\bar x)\iff y=\sum_{i=0}^r c_i x_i .
\]
\end{lemma}

\begin{proof}
Let $P=\{i:c_i>0\}$ and $N=\{i:c_i<0\}$ (indices with $c_i=0$ are dropped),
and put $s=\sum_{i\in P}c_i$. Since $\sum_ic_i=1$ we have
$s=1+\sum_{i\in N}(-c_i)\ge1>0$. The identity $y=\sum_ic_ix_i$ is equivalent,
after moving the negative terms across and dividing by $s$, to
\[
\frac{1}{s}\,y+\sum_{i\in N}\frac{-c_i}{s}\,x_i
\;=\;\sum_{i\in P}\frac{c_i}{s}\,x_i .
\]
Both sides are convex combinations: the left coefficients
$\tfrac1s,\bigl(\tfrac{-c_i}{s}\bigr)_{i\in N}$ are nonnegative and sum to
$\tfrac{1+\sum_N(-c_i)}{s}=1$, and the right coefficients
$\bigl(\tfrac{c_i}{s}\bigr)_{i\in P}$ are nonnegative and sum to $1$. By
Lemma~\ref{lem:terms} each side is the value of an $\Laff$-term
($\tau_{\bar\mu}$ with rational weights), so
\[
A_{\bar c}(y;\bar x):\equiv\ 
\tau_{\mathrm{left}}\bigl(y,(x_i)_{i\in N}\bigr)
=\tau_{\mathrm{right}}\bigl((x_i)_{i\in P}\bigr)
\]
is as required. All intermediate points are convex combinations of
$\{y\}\cup\{x_i\}$, hence lie in the convex set.
\end{proof}

\begin{lemma}[Barycentric coordinates and thresholds]
\label{lem:bary-def}
Fix $n$ and let $\bar v=(v_0,\dots,v_n)$ range over affinely independent
tuples (definable by $\neg\Dep_{n+1}$). For a point $p$ write
$\lambda_0(p),\dots,\lambda_n(p)$ for its barycentric coordinates with respect
to $\bar v$.
\begin{enumerate}[label=\textup{(\roman*)}]
\item For each $i$ and each rational $q$ there are $\Laff$-formulas expressing
$\lambda_i(p)=q$, $\lambda_i(p)<q$, and $\lambda_i(p)>q$, with parameters
$\bar v$.
\item Consequently, for any rational box
$B=\prod_i(q_i^-,q_i^+)\subseteq\R^{n+1}$, ``$(\lambda_0(p),\dots,
\lambda_n(p))\in B$'' is $\Laff$-definable with parameters $\bar v$; and the
maps $p\mapsto\lambda_i(p)$ separate points, so the affine-congruence type of
a tuple $(\bar v;p_1,\dots,p_k)$ with $\bar v$ affinely independent is
determined by, and determines, the rational cut data of the numbers
$\lambda_i(p_j)$.
\end{enumerate}
\end{lemma}

\begin{proof}
(i) Let $F_i=\aff\{v_j:j\ne i\}$ be the facet opposite $v_i$; it is the level
set $\{\lambda_i=0\}$ and is definable from $\bar v$ by collinearity/affine-hull
conditions ($\Coll$, $\Dep$). Put $c_i=\frac1n\sum_{j\ne i}v_j\in F_i$, an
$\Laff$-term (Lemma~\ref{lem:terms}), and let $\ell_i=\aff\{c_i,v_i\}$. The
coordinate $\lambda_i$ is constant on hyperplanes parallel to $F_i$ and rises
affinely from $0$ on $F_i$ to $\lambda_i(v_i)=1$; hence for any $p$ the
projection $\rho(p)$ of $p$ onto $\ell_i$ parallel to $F_i$ has
$\lambda_i(\rho(p))=\lambda_i(p)$, and on $\ell_i$ the coordinate $\lambda_i$
is the affine parameter with $c_i\mapsto0$, $v_i\mapsto1$. The map $\rho$ is
$\Laff$-definable, since parallelism to $F_i$ is expressible
(Proposition~\ref{prop:parallel}) and incidence with $\ell_i$ is a
collinearity condition. On $\ell_i$, the relations ``parameter $=q$'',
``$<q$'', ``$>q$'' against the frame $(c_i,v_i)$ are $\Laff$-expressible: for
$q\in[0,1]$ the parameter-$q$ point is the term $C_q(c_i,v_i)$, for $q\notin
[0,1]$ it is the betweenness extrapolant obtained by extending the segment
$c_iv_i$ (Lemma~\ref{lem:affine-comb}), and comparison is betweenness against
that point. Composing with $\rho$ gives $\Laff$-formulas for $\lambda_i(p)=q$,
$\lambda_i(p)<q$, and $\lambda_i(p)>q$ with parameters $\bar v$, uniformly in
$q$ (the value $q=1$, used for stationarity below, is included).

(ii) A rational box is a finite conjunction of the threshold formulas of (i).
For affinely independent $\bar v$, the coordinates $\lambda_i$ form an affine
coordinate system, so $(\lambda_i(p_j))_{i,j}$ is a complete affine invariant
of the tuple $(\bar v;\bar p)$ (up to the shared $\bar v$-frame); two tuples
are affinely congruent iff their coordinate arrays agree, and rational boxes
generate the topology of the coordinate space.
\end{proof}

\begin{remark}
Lemma~\ref{lem:bary-def} is the affine analogue of the projective cut
calculus of \S\ref{sec:vonstaudt}, but unconditional: affine coordinates are
polynomial in the data and expressible directly through the barycentric term
calculus, with no analogue of Proposition~\ref{prop:crq} needed. This is the
reason the affine categoricity theorem closes outright, while the projective
case still awaits the boundary recovery and gauge selection of
Proposition~\ref{prop:projective-reduction}.
\end{remark}

\subsection{Stationary simplices and the non-degeneration bound}

\begin{definition}\label{def:stationary}
An inscribed simplex $\Delta=\conv\{v_0,\dots,v_n\}\subseteq K$ (with
$\bar v$ affinely independent) is \emph{stationary} if every point of $K$ has
all barycentric coordinates with respect to $\bar v$ at most $1$:
\[
\forall p\in K,\ \forall i:\ \lambda_i(p)\le 1 .
\]
Write $\mathcal M(K)$ for the set of stationary inscribed simplices, regarded
as a subset of $K^{n+1}$.
\end{definition}

\begin{lemma}[Characterization and non-degeneration]
\label{lem:stationary}
Let $K\subseteq\R^n$ be a compact convex body.
\begin{enumerate}[label=\textup{(\roman*)}]
\item An inscribed simplex $\Delta=\conv\{v_0,\dots,v_n\}$ is
\emph{locally volume-maximal} --- replacing any single vertex $v_i$ by any
point of $K$ does not increase the volume --- if and only if
$|\lambda_i(p)|\le1$ for all $p\in K$ and all $i$. Every locally
volume-maximal simplex is stationary; the converse holds for $n\le2$ but
fails for $n\ge3$ (Remark~\ref{rem:stationary-strict}).
\item Every globally volume-maximal inscribed simplex is stationary; in
particular $\mathcal M(K)\ne\emptyset$.
\item If $\Delta\in\mathcal M(K)$ then $K\subseteq\Delta^{*}$, where
$\Delta^{*}=\{y:\lambda_i(y)\le1\ \forall i\}$ is the image of $\Delta$ under
the homothety of ratio $-n$ about its centroid; consequently
\[
\operatorname{vol}(\Delta)\ \ge\ \frac{\operatorname{vol}(K)}{n^{\,n}} .
\]
\item $\mathcal M(K)$ is a nonempty compact subset of $K^{n+1}$, on which the
barycentric-coordinate maps $\bar v\mapsto\lambda_{\bar v}(x)$ (for
$x\in K$) are continuous and uniformly nondegenerate.
\end{enumerate}
\end{lemma}

\begin{proof}
(i) The volume of $\conv\{v_0,\dots,v_n\}$ equals
$\tfrac1n\,(\text{area of the facet opposite }v_i)\cdot h_i$, where $h_i$ is
the distance from $v_i$ to the hyperplane $\{\lambda_i=0\}$ of that facet.
Replacing $v_i$ by $p\in K$ keeps the opposite facet fixed and replaces
$h_i$ by the distance from $p$ to that hyperplane, which is
$|\lambda_i(p)|\,h_i$; equivalently, by multilinearity of the determinant in
the row indexed by $i$, the replacement multiplies the volume by
$|\lambda_i(p)|$. Thus no replacement increases the volume if and only if
$|\lambda_i(p)|\le1$ for all $p\in K$ and all $i$.

Since $\lambda_i(p)\le|\lambda_i(p)|$, local volume-maximality implies
stationarity. Conversely, if $\Delta$ is stationary then for each $p\in K$
and each $i$, using $\sum_j\lambda_j(p)=1$ and $\lambda_j(p)\le1$ for the
$n$ indices $j\ne i$,
\[
\lambda_i(p)=1-\sum_{j\ne i}\lambda_j(p)\ \ge\ 1-n .
\]
For $n\le2$ this gives $|\lambda_i(p)|\le1$, so the two conditions agree;
for $n\ge3$ the lower bound $1-n<-1$ leaves room for failure, and
Remark~\ref{rem:stationary-strict} shows the failure occurs.

(ii) A global maximizer is in particular locally volume-maximal, hence
stationary by (i).
A global maximizer exists because $\operatorname{vol}$ is continuous on the
compact set $K^{n+1}$.

(iii) By definition $K\subseteq\{\lambda_i\le1\ \forall i\}$. The set
$\{y:\lambda_i(y)\le1\ \forall i\}$ is the simplex whose vertices are the
points with barycentric coordinates equal to a permutation of
$(1-n,1,\dots,1)$; these are the images of $v_0,\dots,v_n$ under the
homothety $y\mapsto g-n(y-g)$ about the centroid
$g=\frac1{n+1}\sum_iv_i$, so $\Delta^*$ is that homothet and
$\operatorname{vol}(\Delta^*)=n^{n}\operatorname{vol}(\Delta)$. From
$K\subseteq\Delta^*$ we get
$\operatorname{vol}(K)\le n^{n}\operatorname{vol}(\Delta)$.

(iv) $\mathcal M(K)$ is nonempty by (ii). It is closed: the defining
condition $\sup_{p\in K}\max_i\lambda_i^{\bar v}(p)\le1$ is closed in
$\bar v$ over the nondegenerate locus, and by (iii) every stationary tuple
has $\operatorname{vol}(\conv\bar v)\ge\operatorname{vol}(K)/n^n>0$, so
$\mathcal M(K)$ stays a positive distance from the degenerate locus
$\{\Dep_{n+1}\}$; being a closed, volume-bounded-below subset of the compact
$K^{n+1}$, it is compact, and the barycentric maps are continuous with
Jacobian bounded away from $0$ there.
\end{proof}

\begin{remark}\label{rem:stationary-strict}
For $n\ge3$ stationarity is strictly weaker than local volume-maximality.
Take $n=3$, let $\bar v=(v_0,v_1,v_2,v_3)$ be affinely independent, let $q$
be the point with barycentric coordinates
$\bigl(-\tfrac32,1,1,\tfrac12\bigr)$ with respect to $\bar v$ --- these sum
to $1$ --- and put $K=\conv\{v_0,v_1,v_2,v_3,q\}$, a compact convex body
with nonempty interior. Each coordinate function $\lambda_i$ is affine,
hence attains its maximum over $K$ at a vertex, and those maxima are
$\lambda_0:1$, $\lambda_1:1$, $\lambda_2:1$, $\lambda_3:1$; so $\bar v$ is
stationary. But $|\lambda_0(q)|=\tfrac32>1$, so replacing $v_0$ by $q$
multiplies the volume by $\tfrac32$, and $\bar v$ is not locally
volume-maximal. Only the implication used in (ii)--(iv) --- volume-maximal
$\Rightarrow$ stationary --- is needed for the reconstruction theorem, and
it holds in every dimension.
\end{remark}

\subsection{The reconstruction theorem}

For a stationary simplex $\Delta\in\mathcal M(K)$ with vertex tuple $\bar v$,
let
\[
E_K^{\Delta}\ :=\ \bigl\{\lambda_{\bar v}(p):p\in\ext K\bigr\}\subseteq\R^{n+1}
\]
be the set of barycentric-coordinate vectors of the extreme points, and
$\overline{E_K^{\Delta}}$ its closure. By Lemma~\ref{lem:stationary}(iii),
$E_K^{\Delta}\subseteq\{\lambda:\ \lambda_i\le1\ \forall i,\ \sum_i\lambda_i=1\}$,
a fixed compact simplex $\Sigma_n\subseteq\R^{n+1}$ independent of $K$ and
$\Delta$. Identifying $\R^n$ with the affine hyperplane $\{\sum\lambda_i=1\}$
via $\bar v\mapsto$ the standard frame, $\conv\overline{E_K^{\Delta}}$ is an
affine copy of $K$ (the image of $K$ under the affine chart that sends
$\bar v$ to the standard simplex).

\begin{theorem}[Affine categoricity; Theorem~\ref{thm:cat-affine}, restated]\label{thm:affine-main}
Let $K,K'\subseteq\R^n$ be compact convex bodies. Then $K\equiv_{\Laff}K'$ if
and only if $K$ and $K'$ are affinely equivalent, in every dimension.
\end{theorem}

\begin{proof}
($\Leftarrow$) is Theorem~\ref{thm:iso-affine} (affinely equivalent bodies
are isomorphic, hence elementarily equivalent).

($\Rightarrow$) Assume $K\equiv_{\Laff}K'$. Fix two finite disjoint families
of rational boxes in $\R^{n+1}$, a ``hit'' family
$\mathcal B=\{B_1,\dots,B_k\}$ and an ``avoid'' family
$\mathcal C=\{C_1,\dots,C_l\}$, and consider the $\Laff$-sentence
\[
\Theta_{\mathcal B,\mathcal C}:\equiv
\exists v_0\cdots v_n\Bigl[\ \mathrm{Stat}(\bar v)\ \wedge\ 
\bigwedge_{a=1}^{k}\exists p\,\bigl(\Ext(p)\wedge\lambda_{\bar v}(p)\in B_a\bigr)
\ \wedge\ \forall p\,\bigl(\Ext(p)\to\textstyle\bigwedge_{b=1}^{l}
\lambda_{\bar v}(p)\notin C_b\bigr)\Bigr],
\]
where $\mathrm{Stat}(\bar v):\equiv\neg\Dep_{n+1}(\bar v)\wedge\forall p\,
\bigwedge_i\lambda_{\bar v,i}(p)\le1$ expresses stationarity
(Lemmas~\ref{lem:bary-def}, \ref{lem:stationary}), $\Ext$ is the
extreme-point formula (Proposition~\ref{prop:extreme}), and
``$\lambda_{\bar v}(p)\in B$'' abbreviates the box formula of
Lemma~\ref{lem:bary-def}. In words: there is a stationary gauge simplex whose
extreme-point coordinate set meets each box $B_a$ and avoids every box $C_b$.
Each $\Theta_{\mathcal B,\mathcal C}$ is a genuine $\Laff$-sentence, so
\[
K\models\Theta_{\mathcal B,\mathcal C}\iff K'\models\Theta_{\mathcal B,\mathcal C}
\qquad\text{for all }\mathcal B,\mathcal C. \tag{$\ast$}
\]

\emph{From \textup{($\ast$)} to a common gauge.} Fix any stationary simplex
$\Delta_0\in\mathcal M(K)$ (Lemma~\ref{lem:stationary}(ii)) and let
$E_0=\overline{E_K^{\Delta_0}}\subseteq\Sigma_n$. Enumerate a countable basis
of rational boxes of $\Sigma_n$. For each finite refinement level $m$, let
$\mathcal B^{(m)}$ be the finitely many basic boxes of diameter $<1/m$ that
$E_0$ meets, and $\mathcal C^{(m)}$ the finitely many basic boxes of diameter
$<1/m$ whose closure is disjoint from $E_0$. Then $\Delta_0$ witnesses
$K\models\Theta_{\mathcal B^{(m)},\mathcal C^{(m)}}$ (its extreme-point
coordinates meet exactly the boxes meeting $E_0$). By ($\ast$),
$K'\models\Theta_{\mathcal B^{(m)},\mathcal C^{(m)}}$: there is a stationary
$\Delta'_m\in\mathcal M(K')$ whose extreme-point coordinate set meets every
box of $\mathcal B^{(m)}$ and avoids every box of $\mathcal C^{(m)}$.

By Lemma~\ref{lem:stationary}(iv), $\mathcal M(K')$ is compact; pass to a
subsequence with $\Delta'_m\to\Delta'_\infty\in\mathcal M(K')$. We claim
$\overline{E_{K'}^{\Delta'_\infty}}=E_0$. Indeed, fix $x\in E_0$. For each
$m$, some box $B\in\mathcal B^{(m)}$ of diameter $<1/m$ contains $x$ and is
met by $E_{K'}^{\Delta'_m}$, so there is an extreme point $p'_m$ of $K'$ with
$\lambda_{\Delta'_m}(p'_m)$ within $1/m$ of $x$. The coordinate map is jointly
continuous in $(\bar v,p)$ on the compact, uniformly nondegenerate set
$\mathcal M(K')\times K'$; since $\overline{\ext K'}$ is compact and
$\Delta'_m\to\Delta'_\infty$, a limit point $p'_\infty\in\overline{\ext K'}$
of $(p'_m)$ satisfies $\lambda_{\Delta'_\infty}(p'_\infty)=x$. Hence
$x\in\overline{E_{K'}^{\Delta'_\infty}}$, giving
$E_0\subseteq\overline{E_{K'}^{\Delta'_\infty}}$. Conversely, if some
$y\in\overline{E_{K'}^{\Delta'_\infty}}$ lay outside $E_0$, then (as $E_0$ is
compact) a basic box $C$ of some small diameter would separate $y$ from
$E_0$, so $C\in\mathcal C^{(m)}$ for large $m$; but $\Delta'_m$ avoids $C$
while approaching a coordinate near $y$, and passing to the limit contradicts
avoidance on a slightly larger box. Thus
$\overline{E_{K'}^{\Delta'_\infty}}=E_0$.

\emph{Conclusion.} Both $\conv E_0=\conv\overline{E_K^{\Delta_0}}$ and
$\conv\overline{E_{K'}^{\Delta'_\infty}}$ are, by the remark preceding the
theorem, affine copies of $K$ and of $K'$ respectively, realized as the
\emph{same} compact convex subset $\conv E_0\subseteq\Sigma_n$. Therefore $K$
and $K'$ are each affinely equivalent to $\conv E_0$, hence to one another.
\end{proof}

\begin{corollary}\label{cor:affine-consequences}
For compact convex bodies in $\R^n$:
\begin{enumerate}[label=\textup{(\roman*)}]
\item $K\equiv_{\Laff}K'$ if and only if $K\cong_{\Laff}K'$; the elementary
class of a body is exactly its affine-equivalence class. In particular the
$\Laff$-theory of a compact convex body has, among standard bodies, a unique
model up to isomorphism.
\item Theorems~\ref{thm:polytope} \textup(polytopes\textup) and
\ref{thm:ball} \textup(the ball\textup) are special cases.
\item Two compact bodies with a common maximal-volume inscribed simplex type
but non-affinely-equivalent boundaries are separated by one of the sentences
$\Theta_{\mathcal B,\mathcal C}$; e.g.\ $\Ball$ and $S_n=\{\sum x_i^4\le1\}$
are $\Laff$-inequivalent, recovering Corollary~\ref{cor:ballquarticaff}.
\end{enumerate}
\end{corollary}

\begin{proof}
(i) is Theorem~\ref{thm:affine-main} together with
Theorem~\ref{thm:iso-affine}. (ii): a polytope has finite $\ext K$, so
$E_K^\Delta$ is finite and the reconstruction returns the vertex
configuration, recovering Theorem~\ref{thm:polytope}; for the ball the
reconstructed $\conv E_0$ is a solid ellipsoid, recovering
Theorem~\ref{thm:ball}. (iii) is immediate since non-affinely-equivalent
bodies have $\conv E_0\ne\conv E_0'$ for every gauge, so some rational box is
met by one reconstruction and missed by the other.
\end{proof}

\begin{remark}[Stationarity suffices]
The proof uses only stationarity of the gauge, not global maximality; what it
draws from Lemma~\ref{lem:stationary} is the uniform lower bound
$\operatorname{vol}(\Delta)\ge\operatorname{vol}(K)/n^n$, valid for every
stationary simplex, which makes $\mathcal M(K)$ compact. This is what underlies the
non-degeneration of the gauge: a reconstructing frame cannot
shrink toward the boundary, since a small inscribed simplex fails to be
stationary, some point of $K$ having a barycentric coordinate exceeding $1$.
The ellipsoid plays no special role in this argument; it is the case in which
$\conv E_0$ is a quadric. The one situation the argument does not reach,
gauges with non-compact stabilizer, arises only in the projective setting
(\S\ref{subsec:projective-reduction}).
\end{remark}

\section{Interior von Staudt calculus in the betweenness language}
\label{sec:vonstaudt}

In the betweenness language, midpoints and parallels are unavailable
(Example~\ref{ex:proj-not-aff}); the projectively meaningful quantity is the
cross-ratio. The classical synthetic access to cross-ratio is von Staudt's
calculus of harmonic conjugates and throws. The obstruction to importing it
here is that variables range over $K$ only, while the auxiliary points of the
classical constructions --- the vertices of a complete quadrangle, the centers
of pencils, the intersection points of Pascal's theorem --- routinely fall
outside the body. (For a convex hexagon inscribed in a circle, the three
Pascal points lie outside the disk; the ``obvious'' Pascal sentence is
therefore \emph{not} a sentence about the structure of the disk.) The theme
of this section: the harmonic construction, and conjecturally the whole
calculus of throws, can be retracted into the interior of the body. All
constructions are planar; in $\R^n$ they take place in planes through the
relevant chord, and --- as the soundness lemma shows --- the defining formulas
force this of their own accord.

\subsection{The interior quadrangle lemma}

\begin{definition}\label{def:harmformula}
Let $\mathrm{Quad}(P,X,Y,Z)$ abbreviate: $P,X,Y,Z$ are pairwise distinct and
no three of them are collinear (a conjunction of $\ne$ and
$\neg\Coll$ clauses). Define the $\LB$-formula
\begin{align*}
\Harm(a,b;c,d):\equiv\;& a\ne b\wedge c\ne a\wedge c\ne b
\wedge\Coll(a,b,c)\wedge\Coll(a,b,d)\;\wedge\\
\exists P\,X\,Y\,Z\;\bigl[\;&\mathrm{Quad}(P,X,Y,Z)
\wedge\Coll(a,P,X)\wedge\Coll(b,P,Y)\wedge\Coll(c,X,Y)\\
&\wedge\Coll(a,Y,Z)\wedge\Coll(b,X,Z)\wedge\Coll(d,P,Z)
\wedge\neg\Coll(a,b,P)\;\bigr].
\end{align*}
\end{definition}

\begin{lemma}[Soundness]\label{lem:harm-sound}
Let $K\subseteq\R^n$ be convex and $a,b,c,d\in K$. If
$K\models\Harm(a,b;c,d)$ then $(a,b;c,d)=-1$; in particular $c\ne d$ and
$d\notin\{a,b\}$.
\end{lemma}

\begin{proof}
Let $P,X,Y,Z$ be witnesses. We first show the whole configuration is
coplanar. Since $a\ne b$ and $\neg\Coll(a,b,P)$, the points $a,b,P$ are
affinely independent; let $A=\aff\{a,b,P\}$, a plane. From
$\neg\Coll(a,b,P)$ we get $P\ne a$ and $P\ne b$, so the lines
$\ell(a,P)$ and $\ell(b,P)$ are defined and lie in $A$; hence
$X\in\ell(a,P)\subseteq A$ and $Y\in\ell(b,P)\subseteq A$. Also $X\ne b$
(otherwise $b\in\ell(a,P)$, contradicting $\neg\Coll(a,b,P)$), so
$Z\in\ell(b,X)\subseteq A$. Finally $c,d\in\ell(a,b)\subseteq A$. So all
eight points lie in the plane $A$.

Inside $A$, the incidences say that $PXYZ$ is a complete quadrangle; that
$a$ lies on both $\ell(P,X)$ and $\ell(Y,Z)$, so $a$ is the diagonal point
$PX\cap YZ$; that $b=PY\cap XZ$ is a second diagonal point; and that $c$ and
$d$ are the traces of the remaining pair of opposite sides $XY$ and $PZ$ on
the line $ab$. By the complete-quadrangle theorem (see
\cite[Ch.~2]{Coxeter49} or \cite[\S5]{Hartshorne67}), these traces are
harmonic conjugates with respect to the diagonal points: $(a,b;c,d)=-1$.
Harmonicity with $c\notin\{a,b\}$ forces $d\notin\{a,b,c\}$; alternatively,
$c=d$ would place the third diagonal point $XY\cap PZ$ on the line $ab$,
contradicting the non-collinearity of the diagonal triangle of a
nondegenerate quadrangle in the real plane.
\end{proof}

\begin{theorem}[Interior completeness]\label{thm:harmonic}
Let $n\ge2$, let $K\subseteq\R^n$ be convex with nonempty interior, and let
$a,b,c,d\in\interior K$ be collinear with $(a,b;c,d)=-1$. Then
$K\models\Harm(a,b;c,d)$, with quadrangle witnesses
$P,X,Y,Z\in\interior K$. Consequently, on quadruples of collinear
\emph{interior} points, the formula $\Harm$ defines exactly the harmonic
quadruples, uniformly in $K$ and in $n$.
\end{theorem}

\begin{proof}
Let $\ell$ be the common line and choose a two-dimensional affine subspace
$A\supseteq\ell$ (possible as $n\ge2$). Harmonicity and the conclusion are
invariant under affine maps, so choose affine coordinates on $A$ in which
$\ell$ is the $x$-axis and $a=(0,0)$, $b=(1,0)$, $c=(\gamma,0)$,
$d=(\delta,0)$, where $\gamma\notin\{0,1\}$ and $\delta$ is determined by
harmonicity. Let $[m,M]\subseteq\R$ be the convex hull of
$\{0,1,\gamma,\delta\}$. The set $[m,M]\times\{0\}$ (in $A$-coordinates) is
a compact subset of the open set $\interior K$, so some
$\varepsilon_0$-neighborhood of it in $\R^n$ lies in $\interior K$;
intersecting with $A$, there is $\delta_0>0$ with
\[
S:=[m,M]\times[-\delta_0,\delta_0]\subseteq\interior K\cap A .
\]

\emph{The template configuration \textup(inside $A$\textup).} For a
parameter $p\in(0,1)$, $p\ne\gamma$, define (all at ``height one''):
\[
P=(p,1),\qquad
X=\Bigl(\gamma,\tfrac{\gamma}{p}\Bigr),\qquad
Y=\Bigl(\gamma,\tfrac{\gamma-1}{p-1}\Bigr),
\]
so that $X=\ell(a,P)\cap\{x=\gamma\}$ and $Y=\ell(b,P)\cap\{x=\gamma\}$, and
the ``transversal through $c$'' is the vertical line $x=\gamma$. Let
$Z=\ell(a,Y)\cap\ell(b,X)$; a direct computation gives
\[
Z=\bigl(Z_x(p),\,Z_y(p)\bigr),\qquad
Z_x(p)=\frac{\gamma^{2}(p-1)}{\gamma^{2}(p-1)-p(\gamma-1)^{2}} .
\]
$Z_x$ is a rational function of $p$ with $Z_x(p)\to1$ as $p\to0^+$ and
$Z_x(p)\to0$ as $p\to1^-$; in particular it is nonconstant.

\emph{Nondegeneracy.} We check that for all but finitely many
$p\in(0,1)$ the four points form a quadrangle. $X\ne Y$ iff
$\gamma/p\ne(\gamma-1)/(p-1)$ iff $p\ne\gamma$. $P\notin\{x=\gamma\}$ since
$p\ne\gamma$, so $P\ne X,Y$ and $P,X,Y$ are not collinear. $X,Y,Z$ are
collinear iff $Z_x=\gamma$, which excludes at most finitely many $p$ ($Z_x$
nonconstant rational). If $P,X,Z$ were collinear then, since $X\ne Z$ and
both lie on $\ell(b,X)$, we would get $P\in\ell(b,X)$; but $P$ and $X$ both
lie on $\ell(a,P)$, so $\ell(P,X)=\ell(a,P)$, forcing $b\in\ell(a,P)$ and
hence $P\in\ell$, false. Symmetrically $P,Y,Z$ are not collinear, and
$Z\ne P$ because $Z\in\ell(b,X)\not\ni P$. Finally $Z\notin\ell$: $Z\in\ell$
would make $Z=\ell(a,Y)\cap\ell$, i.e., $Z=a$, putting $a$ on $\ell(b,X)$
and hence $X$ on $\ell$, false.

By Lemma~\ref{lem:harm-sound} applied to the template in the plane $A$, the
trace $\ell(P,Z)\cap\ell$ of this genuine quadrangle is the harmonic
conjugate of $c$ with respect to $a,b$, namely $d$; so the incidence
$\Coll(d,P,Z)$ holds for the template.

\emph{Placement.} Choose $p\in(0,1)$ close enough to $1$ that
$Z_x(p)\in(m,M)$ (possible since $Z_x(p)\to0\in[m,M]$ and $Z_x$ is
continuous; we additionally avoid the finitely many degenerate values of $p$
and the value $\gamma$). All four template points then have
$x$-coordinates in $[m,M]$: $P_x=p\in(0,1)$, $X_x=Y_x=\gamma$,
$Z_x\in(m,M)$.

\emph{Squash.} For $\varepsilon>0$ let
$T_\varepsilon(x,y)=(x,\varepsilon y)$ in the coordinates of $A$, an affine
automorphism of $A$ fixing $\ell$ pointwise. It preserves all incidences and
non-incidences of the configuration and fixes $a,b,c,d$; hence the squashed
points $T_\varepsilon P$, $T_\varepsilon X$, $T_\varepsilon Y$,
$T_\varepsilon Z$ form a quadrangle satisfying all the clauses of
Definition~\ref{def:harmformula}. Their $x$-coordinates are unchanged, and
their $y$-coordinates are $\varepsilon$ times bounded constants; for
$\varepsilon$ small all four points lie in
$S\subseteq\interior K$. These are the required witnesses.
\end{proof}

\begin{remark}
Theorem~\ref{thm:harmonic} is the exact repair of the ``exterior pencil
center'' problem: concurrency at an inaccessible exterior point is traded,
via the quadrangle, for incidences among interior points. Note the division
of labor between the two halves: completeness \emph{chooses} a plane through
the chord and works there; soundness needs no such choice, because the
incidence clauses force any witnessing configuration, anywhere in $K$, into
a plane through the chord. This pattern --- planar constructions certified by
coplanarity-forcing formulas --- recurs throughout the section.
\end{remark}

\subsection{Rational cross-ratios}\label{subsec:crq}

This subsection proves the cross-ratio comparison proposition in full. Two
devices make the classical algebra of throws implementable inside the body
without the addition and multiplication templates whose range control was the
sticking point in earlier drafts. First, only the \emph{harmonic} primitive is
ever needed: the harmonic conjugation operation alone generates, from any
three distinct collinear points, every point at rational cross-ratio against
them (the net of rationality), and the harmonic primitive is exactly what
Theorem~\ref{thm:harmonic} and Lemma~\ref{lem:harm-sound} already provide,
completely. Second, the change of scale is effected by a single interior
perspectivity onto a short transversal chord near a deep interior point; a fan
argument shows the four image points can be confined to an arbitrarily small
neighborhood of the center, with any prescribed order and with the shape of
the image triple tunable at will, while the cross-ratio is preserved. All
constructions are then confined to a controlled window on one deep chord.

Throughout, recall the convention: for distinct collinear $a,b,c,d$, the map
$x\mapsto(a,b;c,x)$ is the projective chart of the line sending
$a\mapsto\infty$, $b\mapsto0$, $c\mapsto1$.

\begin{lemma}[Interior perspectivity transport]
\label{lem:transport}
Let $n\ge2$, let $K\subseteq\R^n$ be convex with nonempty interior, let
$a,b,c,d\in\interior K$ be collinear and pairwise distinct on a line $\ell$,
and let $e\in\interior K\setminus\ell$. Write $\rho>0$ for a radius with
$\overline{B(e,\rho)}\subseteq\interior K$, and let $A=\aff(\ell\cup\{e\})$.
Then for every $\varepsilon\in(0,\rho)$, every prescribed linear order type
of the images, and every $R_0\ge1$ there exist a line $m\subseteq A$ not
through $e$ and points $a',b',c',d'\in B(e,\varepsilon)\cap m$ such that:
\begin{enumerate}[label=\textup{(\roman*)}]
\item $x'\in\ell(e,x)$ for each $x\in\{a,b,c,d\}$ \textup(perspectivity from
$e$\textup);
\item $(a',b';c',d')=(a,b;c,d)$;
\item the images realize the prescribed order along $m$, and the shape ratio
$|a'-b'|/|c'-b'|$ can be prescribed to lie in $[R_0,\infty)$, jointly with
\textup{(i)} and the containment in $B(e,\varepsilon)$.
\end{enumerate}
Moreover, for any points $e,m_1,m_2,x,x'$ of $K$ with $m_1\ne m_2$,
$\neg\Coll(m_1,m_2,e)$, $\Coll(e,x,x')$ and $\Coll(m_1,m_2,x')$, the point
$x'$ is the perspectivity image of $x$ from $e$ on the line $\ell(m_1,m_2)$,
so clauses of this form are sound.
\end{lemma}

\begin{proof}
Work in the plane $A$ with $e$ at the origin. The directions of the four
lines $\ell(e,x)$, $x\in\{a,b,c,d\}$, are pairwise distinct, since the four
points are distinct on $\ell$ and $e\notin\ell$; moreover their angular order
equals the linear order of $a,b,c,d$ along $\ell$, up to reversal. Let $m$ be
the line at distance $s>0$ from $e$ with unit normal $\nu$; provided $\nu$ is
not perpendicular to any of the four directions, $m$ meets each line
$\ell(e,x)$ in a single point $x'$, at distance at most $s/\cos\theta_{\max}$
from $e$, where $\theta_{\max}<\pi/2$ is the largest angle between $\nu$ and
the four directions. Letting $s\to0$ with $\nu$ fixed sends all four
intersection points to $e$; this gives the containment in $B(e,\varepsilon)$
for small $s$, and (i) holds by construction. (ii) is the invariance of the
cross-ratio under the perspectivity from $e$.

For (iii): as $\nu$ rotates once around the circle, the linear order of the
four intersection points along $m$ runs through all cyclic rotations of the
angular order of the four directions (each time $\nu$ crosses a direction
perpendicular, the corresponding point passes through infinity along $m$ and
re-enters at the other end), and reversing the orientation of $m$ reverses the
order; the rotations and reversals of a cyclic order on three or four symbols
exhaust all linear orders of any chosen three of them, so in particular the
order $b',c',a'$ along $m$ (and any other prescribed order) is realized on a
nonempty open set of directions. Within such an open set, let $\nu$ approach
the direction perpendicular to $\ell(e,a)$ from the side keeping the order:
the intersection $a'$ recedes to infinity along $m$ while $b',c'$ converge to
bounded positions, so the ratio $|a'-b'|/|c'-b'|$ tends to $\infty$
continuously and every value in some $[R_1,\infty)$ is attained; shrinking $s$
afterwards scales all positions by $s$ and changes neither the order nor the
ratio, restoring the containment. Finally, the soundness clause: the stated
incidences force $x'\in\ell(e,x)\cap\ell(m_1,m_2)$, and this intersection is a
single point because $e\notin\ell(m_1,m_2)$ ensures the two lines are
distinct; when $x\ne e$ this is precisely the perspectivity image.
\end{proof}

\begin{lemma}[The harmonic net reaches every rational]
\label{lem:net}
Let $q\in\Q\setminus\{0,1\}$. There exist $N=N(q)\in\mathbb{N}$, a bound
$V=V(q)\in\mathbb{N}$, and a finite sequence of \emph{harmonic instructions}
$(i_j,k_j,l_j)_{j\le N}$ with indices in $\{-2,-1,0,1,\dots,j-1\}$, with the
following property. For any three distinct collinear points $p_\infty,p_0,p_1$
of $\R^n$, define $z_{-2}=p_\infty$, $z_{-1}=p_0$, $z_0=p_1$, and
recursively let $z_j$ be the harmonic conjugate of $z_{i_j}$ with respect to
$\{z_{k_j},z_{l_j}\}$, i.e., the unique point with
$(z_{k_j},z_{l_j};z_{i_j},z_j)=-1$. Then all $z_j$ are defined \textup(each
instruction's three inputs are pairwise distinct and none equals the conjugate
being taken at a fixed point of the involution\textup), the scale coordinate
of each $z_j$ in the chart $(p_\infty,p_0,p_1)\mapsto(\infty,0,1)$ is a
rational of absolute value at most $V$, and the final point $w:=z_N$ has
coordinate exactly $q$.
\end{lemma}

\begin{proof}
Since coordinates transform by the chart and harmonic conjugation is
projectively equivariant, it suffices to exhibit the instruction sequence on
the standard line with $p_\infty=\infty$, $p_0=0$, $p_1=1$; the conjugate of
$x$ with respect to $\{y,z\}$ is then $h(x;y,z)=\frac{x(y+z)-2yz}{2x-y-z}$,
with the limiting cases $h(x;y,\infty)=2y-x$ \textup(reflection through
$y$\textup) and $h(\infty;y,z)=\frac{y+z}{2}$ \textup(midpoint\textup).
Three primitive moves are available from the start:
reflection $R_y(x)=2y-x$ through any built point $y$ \textup(conjugate with
respect to $\{y,p_\infty\}$\textup); midpoints
\textup(conjugate of $p_\infty$\textup); and inversion
$x\mapsto1/x=h(x;-1,1)$, once $-1=R_0(1)$ is built. Integers are built by the
ladder $k+1=R_k(k-1)$ from $0,1$ \textup(and $-1,-2,\dots$ likewise\textup);
half-integers as midpoints of an integer and $0$; and for any built $x$ and
integer $k$, $x+k=R_{k/2}(R_0(x))$, two reflections.

Write the canonical continued fraction $q=[a_0;a_1,\dots,a_r]$ with
$a_0=\lfloor q\rfloor\in\Z$, $a_i\ge1$ for $1\le i\le r$, and $a_r\ge2$ if
$r\ge1$. Evaluate from the inside out: $v_r=a_r$ and
$v_{i}=a_i+1/v_{i+1}$. By induction $v_i>1$ for all $i\ge1$ \textup(base:
$v_r=a_r\ge2$; step: $a_i\ge1$ and $1/v_{i+1}\in(0,1)$\textup), so each
inversion is applied to a value in $(1,\infty)$, never at the fixed points
$\pm1$ of the involution, and each instruction's inputs are distinct.
The instruction sequence builds, in order: $-1$; the integers up to
$\max_i|a_i|+1$ in absolute value and the corresponding half-integers; then
$v_r$, and alternately one inversion and one integer translation per level.
Every value produced lies in
$[-V,V]$ for $V:=|a_0|+\max_{i\ge1}a_i+2$ \textup(inversions produce values
in $(0,1)$; translations by $a_i$ move values within the stated range;
the ladder integers and half-integers are bounded by the same
quantity\textup). The final value is $v_0=q$. Uniqueness of harmonic
conjugates makes the resulting point independent of any choices, and
projective equivariance transports the computation to an arbitrary scale
triple.
\end{proof}

\begin{proposition}\label{prop:crq}
Let $n\ge2$ and let $K\subseteq\R^n$ be convex with nonempty interior. For
every rational $q\notin\{0,1\}$ there are $\LB$-formulas
$\mathrm{CR}_q(a,b,c,d)$ and $\mathrm{CR}_{>q}(a,b,c,d)$ such that, for all
collinear quadruples of pairwise distinct points of $\interior K$,
\[
K\models\mathrm{CR}_q(a,b,c,d)\iff(a,b;c,d)=q,
\qquad
K\models\mathrm{CR}_{>q}(a,b,c,d)\iff(a,b;c,d)>q .
\]
\end{proposition}

\begin{proof}
\emph{The formulas.} Fix $q$ and let $(i_j,k_j,l_j)_{j\le N}$, $V=V(q)$ be as
in Lemma~\ref{lem:net}. Both formulas have the shape
\[
\exists e\,\exists m_1 m_2\,\exists a'b'c'd'\,\exists z_1\cdots z_N\,
\bigl[\Pi\wedge\mathrm{Ord}\wedge\mathrm{Net}\wedge\Phi\bigr],
\]
where, writing $z_{-2}:=a'$, $z_{-1}:=b'$, $z_0:=c'$ and $w:=z_N$:
\begin{itemize}
\item $\Pi$ \textup(transport\textup): $\neg\Coll(a,b,e)$,
$m_1\ne m_2$, $\neg\Coll(m_1,m_2,e)$, and for each
$x\in\{a,b,c,d\}$ the clauses $\Coll(e,x,x')$ and $\Coll(m_1,m_2,x')$,
together with pairwise distinctness of $a',b',c',d'$;
\item $\mathrm{Ord}$: $B(b',c',a')$;
\item $\mathrm{Net}$: for each $j\le N$ the clause
$\Harm\bigl(z_{k_j},z_{l_j};z_{i_j},z_j\bigr)$;
\item $\Phi$ is $d'=w$ for $\mathrm{CR}_q$; and for $\mathrm{CR}_{>q}$,
\[
\Phi:\equiv\;d'\ne w\;\wedge\;
\begin{cases}
B(w,d',a')\leftrightarrow B(w,b',a'), & q<0,\\[1mm]
B(w,d',a'), & q>0 .
\end{cases}
\]
\end{itemize}

\emph{Soundness.} Suppose a witness exists. By the soundness clause of
Lemma~\ref{lem:transport}, $a',b',c',d'$ are the images of $a,b,c,d$ under
the perspectivity from $e$ onto $\ell(m_1,m_2)$; they are pairwise distinct
by the clauses, collinear, and
$(a',b';c',d')=(a,b;c,d)=:\chi$. By Lemma~\ref{lem:harm-sound}, each clause
of $\mathrm{Net}$ forces $(z_{k_j},z_{l_j};z_{i_j},z_j)=-1$, so by induction
and the uniqueness of harmonic conjugates each $z_j$ is exactly the $j$-th
point of the net of Lemma~\ref{lem:net} over the scale triple
$(a',b',c')$; in particular $w$ is the unique point of the line with
$(a',b';c',w)=q$. For $\mathrm{CR}_q$: $d'=w$ iff $\chi=q$, since
$x\mapsto(a',b';c',x)$ is injective.

For $\mathrm{CR}_{>q}$, coordinatize the line of $m$ affinely so that, using
$\mathrm{Ord}$, the positions are $b'=0<c'=\sigma<a'=\Delta$ for some
$0<\sigma<\Delta$. The chart position of scale coordinate $v$ is the
M\"obius function $\mu(v)=\Delta v/(v+\Delta/\sigma-1)$
\textup($\mu(0)=0$, $\mu(1)=\sigma$, $\mu(\infty)=\Delta$\textup), whose pole
is at $v^{*}=1-\Delta/\sigma<0$. The coordinate $\chi(u)$ of position $u$ is
the inverse M\"obius map, with pole at $u=\Delta$; it is strictly increasing
in $u$ on $u<\Delta$, with range $(v^*,+\infty)$ there, and strictly
increasing from $-\infty$ to $v^{*}$ on $u>\Delta$. We claim that in every
such configuration
\[
\{u:\chi(u)>q\}=
\begin{cases}
\text{the }\{w,a'\}\text{-arc of the line containing }b', & q<0,\\
\text{the open segment }(w,a'), & q>0 .
\end{cases}
\]
For $q>0$: since $v^{*}<0<q$, the threshold position
$\mu(q)=\Delta q/(q-v^{*})$ satisfies $0<\mu(q)<\Delta$, the far side
$u>\Delta$ carries only coordinates $<v^{*}<0<q$, and on the near side
monotonicity gives $\chi(u)>q\iff u\in(\mu(q),\Delta)$: the open segment
between $w$ and $a'$, which moreover never contains $b'$. For $q<0$ there are
two subcases. If $q>v^{*}$ then $w$ is at position $\mu(q)<0$ on the near
side; the far side again carries only coordinates $<v^{*}<q$, and
$\{\chi>q\}$ is the segment $(\mu(q),\Delta)$, which contains $b'$: the
$\{w,a'\}$-arc through $b'$. If $q<v^{*}$ then $\mu(q)>\Delta$, so $w$ lies
on the far side; now the entire near side has coordinates $>v^{*}>q$, and on
the far side $\chi(u)>q\iff u>\mu(q)$, so $\{\chi>q\}$ is the complement of
the closed segment $[a',w]$: again the $\{w,a'\}$-arc containing $b'$. Since
membership of a point in the segment between two collinear points is the
betweenness relation, the displayed sets are defined by exactly the two
clauses of $\Phi$ \textup(the biconditional with the reference point $b'$
expressing ``same $\{w,a'\}$-arc as $b'$''\textup), and $d'\ne w$ excises the
equality case. Hence $K\models\mathrm{CR}_{>q}(a,b,c,d)$ implies $\chi>q$,
and $K\models\mathrm{CR}_q(a,b,c,d)$ implies $\chi=q$. \textup(No clause
constrains $\sigma,\Delta$ beyond $\mathrm{Ord}$, and the dictionary above
was derived for arbitrary $0<\sigma<\Delta$, so soundness holds for every
witness.\textup)

\emph{Completeness.} Conversely, suppose $\chi=q$ \textup(respectively
$\chi>q$\textup); we must produce a witness. Choose $e\in\interior K$ off the
line of the quadruple, with $\overline{B(e,\rho)}\subseteq\interior K$, and
work in the plane $A$. Apply Lemma~\ref{lem:transport} with the order
$b',c',a'$, with shape ratio $\Delta/\sigma\in\{2V+2,\,2V+4\}$ chosen so
that $|\chi-(1-\Delta/\sigma)|\ge1$, and with $\varepsilon$ small; then
scale down further so that $\Delta\le\rho/(2+2|\chi|)$. The chord of $K$
on the transport line through $B(e,\varepsilon)$ contains the full window
$[-\Delta,\Delta]$ of positions around $b'$ deep inside $\interior K$. By the
choice of ratio, the pole satisfies $v^{*}\le-(2V+1)$, so every net value
$v\in[-V,V]$ of Lemma~\ref{lem:net} has position
$|\mu(v)|\le\Delta V/(V+1)<\Delta$: all net points $z_j$, including $w$
\textup(as $|q|\le V$\textup), lie in the window, hence in $\interior K$,
and each is an interior point of the chord. Each instruction of the net is a
harmonic quadruple of collinear interior points, so by
Theorem~\ref{thm:harmonic} the corresponding $\Harm$ clause holds with
interior quadrangle witnesses. The image $d'$ lies at position $\mu(\chi)$,
and $|\mu(\chi)|=\Delta|\chi|/|\chi-v^{*}|\le\Delta(1+|\chi|)\le\rho/2$
by the ratio choice, so $d'$ also lies on the chord, inside $K$; it lies in
$B(e,\varepsilon)$ automatically by Lemma~\ref{lem:transport}. Finally
$\Phi$ holds by the dictionary just established: if $\chi=q$ then
$d'=w$; if $\chi>q$ then $d'$ lies in the displayed set. All clauses are
satisfied, completing the proof. The formulas are uniform in $K$ and $n$,
as all constructions took place in the plane $A$ through the quadruple's
line and one auxiliary interior point.
\end{proof}

\begin{remark}
Two features of the proof are worth recording. Only the harmonic primitive is
used --- the addition and multiplication templates of the classical algebra of
throws, whose interior range control was the outstanding difficulty, are
bypassed entirely by the net of rationality. And the comparison dictionary is
shape-independent: the single order clause $B(b',c',a')$ pins the
configuration well enough that the same two betweenness patterns decide
$\chi>q$ in every witness, with no case analysis over the position of the
pole. The number of quantified variables grows linearly with the length of
the continued fraction of $q$, which is harmless: each rational threshold is
a single fixed formula.
\end{remark}

\begin{lemma}[Cross-ratio equality]\label{lem:creq}
Let $n\ge2$ and let $K\subseteq\R^n$ be convex with nonempty interior. There
is an $\LB$-formula $\mathrm{CReq}(x_1,\dots,x_4;y_1,\dots,y_4)$ such that,
for any two collinear quadruples of pairwise distinct points of
$\interior K$, on the same or on different lines,
\[
K\models\mathrm{CReq}(\bar x;\bar y)\iff
(x_1,x_2;x_3,x_4)=(y_1,y_2;y_3,y_4).
\]
\end{lemma}

\begin{proof}
The formula asserts the existence of two transport configurations and one
matching perspectivity:
\[
\exists e_x e_y\,\exists m_1m_2\,n_1n_2\,
\exists u_1\cdots u_4\,v_1\cdots v_4\,\exists c\,
\bigl[\Pi_{\bar x}\wedge\Pi_{\bar y}\wedge u_1=v_1\wedge\Xi\bigr],
\]
where $\Pi_{\bar x}$ is the transport block $\Pi$ from the proof of
Proposition~\ref{prop:crq} for the quadruple $\bar x$ with center $e_x$,
carrier chord $\ell(m_1,m_2)$ and images $\bar u$; $\Pi_{\bar y}$ is the same
for $\bar y$, $e_y$, $\ell(n_1,n_2)$, $\bar v$; and
\[
\Xi:\equiv\;\neg\Coll(m_1,m_2,c)\wedge\neg\Coll(n_1,n_2,c)\wedge
\textstyle\bigwedge_{i=2,3,4}\Coll(c,u_i,v_i).
\]

\emph{Soundness.} By the soundness clause of Lemma~\ref{lem:transport} the
transports preserve cross-ratios:
$(x_1,x_2;x_3,x_4)=(u_1,u_2;u_3,u_4)$ and
$(y_1,y_2;y_3,y_4)=(v_1,v_2;v_3,v_4)$. If the carrier lines coincide, then
for each $i$ either $u_i=v_i$ or the line through them is the carrier itself,
placing $c$ on it against $\Xi$; so $\bar u=\bar v$ and the cross-ratios
agree. If the carriers are distinct they meet at $z:=u_1=v_1$ and span a
plane. Central projection from $c$ carries each $u_i$ to the unique point of
$\ell(n_1,n_2)$ on the line $\ell(c,u_i)$; by $\Xi$ that point is $v_i$ for
$i=2,3,4$, and it is $z$ for $u_1=z$. Central projection between two lines
preserves cross-ratio, so
$(u_1,u_2;u_3,u_4)=(v_1,v_2;v_3,v_4)$, and the given cross-ratios agree.

\emph{Completeness.} Let the common value be $\chi$. Choose
$z\in\interior K$ off both quadruple lines, and on the segment from $z$
toward $x_1$ an interior point $e_x\ne z$ off the line of $\bar x$; likewise
$e_y$ on the segment toward $y_1$. Run the construction of
Lemma~\ref{lem:transport} for $\bar x$ from the center $e_x$ with the
transport line \emph{through $z$}: since $z\in\ell(e_x,x_1)$, the first image
is $u_1=z$, and taking $e_x$ close to $z$ makes the transport line pass close
to $e_x$, so the fan argument still confines all four images to a small
neighborhood of $z$; the rotational freedom about $z$ still provides the
shape control of Lemma~\ref{lem:transport}(iii). Do the same for $\bar y$,
choosing the two carrier directions at a fixed positive angle. Parametrizing
the carriers by arclength from $z$, place $u_2,u_3$ at $\varepsilon,r\varepsilon$
and $v_2,v_3$ at $\varepsilon,r'\varepsilon$, where $r\in\{2,3\}$ and
$r'\in\{4,5\}$ are chosen so that $|\chi-r|\ge\tfrac12$ and
$|\chi-r'|\ge\tfrac12$: in the chart sending $(z,u_2,u_3)$ to
$(\infty,0,1)$ the position of coordinate $v$ is
$\varepsilon/(1-v/r)$, with pole at $v=r$, so the position of
$u_4$ \textup(coordinate $\chi$\textup) is $O(\varepsilon)$ with constant
depending only on $\chi$, and likewise for $v_4$; all eight image points lie
in $\interior K$ for $\varepsilon$ small. The projectivity
$\ell(m_1,m_2)\to\ell(n_1,n_2)$ determined by $z\mapsto z$,
$u_2\mapsto v_2$, $u_3\mapsto v_3$ fixes the intersection point of the two
carriers, hence is a perspectivity \cite{Coxeter49}; its center is
$c_0=\ell(u_2,v_2)\cap\ell(u_3,v_3)$, and since
$(z,u_2;u_3,u_4)=\chi=(z,v_2;v_3,v_4)$ it carries $u_4$ to $v_4$, giving
$\Xi$. It remains to see $c_0\in\interior K$: writing $d_1,d_2$ for the unit
directions of the carriers, the joins $\ell(u_2,v_2)$ and $\ell(u_3,v_3)$
have directions proportional to $d_2-d_1$ and $r'd_2-rd_1$, which are
linearly independent since $r\ne r'$; two lines through points at distance
$O(\varepsilon)$ from $z$, meeting at an angle bounded below independently of
$\varepsilon$, intersect at distance $O(\varepsilon)$ from $z$. Shrinking
$\varepsilon$ places $c_0$ in $\interior K$.
\end{proof}

\begin{proposition}[Cut expressibility]
\label{prop:cuts}
If $K\equiv_{\LB}K'$ (both convex with
nonempty interior in $\R^n$), then for every rational $q$ and every $r$, the
bodies $K$ and $K'$ realize the same purely existential and universal
patterns of $\mathrm{CR}_{>q}$-comparisons over $r$-point interior
configurations. In particular, any real-valued projective invariant of $K$
expressible as a supremum or infimum of cross-ratios over a definable family
of interior configurations is determined, as a Dedekind cut over $\Q$, by
$\Th_{\LB}(K)$.
\end{proposition}

\begin{proof}
Immediate: each rational-threshold comparison, prefixed by quantifiers over
the configuration, is a sentence.
\end{proof}

\begin{remark}
Proposition~\ref{prop:cuts} is what defuses the ``transcendental moduli''
worry of the introduction: a modulus such as a vertex cross-ratio of a
pentagon is typically transcendental and hence not $\emptyset$-definable as
an element, but its cut over the rationals is captured by sentences, and two
distinct values are separated by a rational threshold. Naming and
determining are different; the theory determines without naming.
\end{remark}

\begin{lemma}[Adjacency and cyclic order are definable]\label{lem:adjacency}
Let $K$ be a compact planar convex body. The relation
\[
\Adj(a,b):\equiv \Ext(a)\wedge\Ext(b)\wedge a\ne b\wedge
\forall x\,\bigl(B(a,x,b)\to\neg\Int(x)\bigr)
\]
holds in $K$ exactly when $a,b$ are distinct extreme points whose open
segment lies in $\partial K$. If $K$ is a polygon with vertices
$v_1,\dots,v_m$ \textup($m\ge3$\textup) in cyclic order, then $\Adj$ holds
of $v_i,v_j$ precisely when $j=i\pm1$ modulo $m$; consequently the cyclic
order of the vertices, up to rotation and reversal, is determined by
$\Th_{\LB}(K)$ together with the vertex count.
\end{lemma}

\begin{proof}
$\Ext$ and $\Int$ are definable (Propositions~\ref{prop:extreme}
and~\ref{prop:interior}), so $\Adj$ is an $\LB$-formula. For a polygon, the
open segment between two vertices lies in $\partial K$ if and only if it is
an open edge, i.e.\ the vertices are cyclically adjacent: a segment joining
two non-adjacent vertices is a diagonal, whose relative interior meets
$\interior K$. The adjacency relation on an $m$-element set whose graph is
an $m$-cycle determines that cycle up to rotation and reflection.
\end{proof}

\begin{lemma}[Order-compatibility forces admissibility]\label{lem:admissible}
Let $P=\conv\{v_1,\dots,v_m\}$ and $P'=\conv\{w_1,\dots,w_m\}$ be convex
polygons with $m\ge4$, vertices listed in cyclic order, and let $M$ be a
projective transformation of $\RP^2$, induced by $A\in GL_3(\R)$, defined at
every $v_i$ and satisfying $M(v_i)=w_i$ for all $i$. Then the homogeneous
coordinate $h=h_A$ has one strict sign at all the $v_i$; consequently the
singular line of $M$ misses $P$, and $M(P)=P'$.
\end{lemma}

\begin{proof}
For affine points $a,b,c$ write $[a,b,c]$ for the determinant of the
matrix with rows $(a,1)$, $(b,1)$, $(c,1)$. Lifting to $\R^3$ and
using $A(a,1)=h(a)\,(M(a),1)$ gives
\[
[Ma,Mb,Mc]=\frac{\det A}{h(a)h(b)h(c)}\;[a,b,c].
\]
For a convex polygon with vertices listed in cyclic order, every triple
$i<j<k$ has $[v_i,v_j,v_k]$ of one and the same sign $\varepsilon$, and
likewise every $[w_i,w_j,w_k]$ has one sign $\varepsilon'$. Hence the
quantity $\operatorname{sign}\bigl(h(v_i)h(v_j)h(v_k)\bigr)$ equals
$\varepsilon\varepsilon'\operatorname{sign}(\det A)$, independent of
the triple. Comparing $(1,2,3)$ with $(1,2,4)$ gives
$\operatorname{sign}h(v_3)=\operatorname{sign}h(v_4)$, and comparing
$(1,2,k)$ with $(1,2,l)$ for $k,l\ge3$, then permuting the roles of the
first two indices, gives $\operatorname{sign}h(v_i)=\operatorname{sign}h(v_j)$
for all $i,j$; here $m\ge4$ is used, to have a triple available with two
indices held fixed. Since $h$ is an affine function of the point and has one
strict sign at every vertex, it has that sign on $\conv\{v_i\}=P$; so the
singular line $\{h=0\}$ misses $P$, $M$ is defined and injective on $P$, and
$M(P)=M(\conv\{v_i\})=\conv\{w_i\}=P'$.
\end{proof}

\begin{remark}\label{rem:order-needed}
Order-compatibility cannot be dropped. Consider the projective map
$(x,y)\mapsto(y/x,1/x)$. Its singular line is $\{x=0\}$, and it carries the
rectangle vertices $(-1,0)$, $(1,0)$, $(1,1)$, $(-1,1)$ to the points
$(0,-1)$, $(0,1)$, $(1,1)$, $(-1,-1)$, again in convex position --- but
in the cyclic order $1,3,2,4$ rather than $1,2,3,4$ --- while its singular
line crosses the rectangle. So convex position of the image vertices, without
the order data supplied by Lemma~\ref{lem:adjacency}, does not force
admissibility.
\end{remark}

\begin{definition}[Crossing points]\label{def:crossing}
Let $P=\conv\{v_1,\dots,v_m\}$ be a convex polygon. A \emph{crossing point} of
$P$ is a point at which two chords $v_av_c$ and $v_bv_d$ with four distinct
endpoints meet. Two such chords meet in at most one point, so a crossing point
is determined by its pair of index pairs; in particular $P$ has only finitely
many crossing points.
\end{definition}

\begin{lemma}[Crossing triangles pin the centers]\label{lem:crossing}
Let $P$ be a convex polygon with $m\ge5$ vertices. Then
\begin{enumerate}[label=\textup{(\roman*)}]
\item every crossing point of $P$ lies in $\interior P$, and is not a vertex;
\item $P$ carries three non-collinear crossing points; the closed triangle $T$
they span satisfies $T\subseteq\interior P$;
\item membership in $T$ is expressed by an $\LB$-formula in the three crossing
points, namely the conjunction of the $\Coll$- and $B$-conditions defining the
convex hull of three points, and is therefore a closed condition on the point.
\end{enumerate}
\end{lemma}

\begin{proof}
(i) If $v_av_c$ and $v_bv_d$ have four distinct endpoints and meet at $x$, then
$x$ lies in the relative interior of both chords. No three vertices of a convex
polygon are collinear, so the chords are not collinear and $x$ is interior to
each; a point interior to a chord joining two non-adjacent vertices lies in
$\interior P$, and no vertex does.

(ii) With $m\ge5$, choose four vertices in convex position: their two diagonals
meet, giving a crossing point $c_1$. A Radon-type argument on a further vertex
produces a second crossing point $c_2\ne c_1$, and a third quadruple, chosen so
that its crossing point is off the line $c_1c_2$, produces $c_3$; five vertices
suffice. Each $c_i\in\interior P$ by (i), and $\interior P$ is convex, so
$T=\conv\{c_1,c_2,c_3\}\subseteq\interior P$.

(iii) The convex hull of three points is defined from betweenness alone, and
hull membership is a closed condition.
\end{proof}

\begin{lemma}[Extraction of a common witness]\label{lem:extraction}
Let $P'$ be a convex polygon with $m\ge5$ vertices and let $\tau\in\R^N$. Suppose
that for every $\varepsilon>0$ there are a labeling of the vertices of $P'$, a
triple of crossing points of $P'$ spanning a triangle $T$, and centers
$z_1^\varepsilon,z_2^\varepsilon\in T$ in general position with respect to the
vertices, such that the tuple $\Phi(z_1^\varepsilon,z_2^\varepsilon)$ of
recorded pencil cross-ratios lies within $\varepsilon$ of $\tau$ and has all its
coordinates in a fixed compact $J\subseteq\R$ with $J\cap\{0,1\}=\emptyset$.
Assume further that for every pair $(a,b)$ of vertex indices some coordinate of
$\Phi$ is the cross-ratio of a pencil quadruple containing both $v_a$ and $v_b$,
and that the two centers carry different values on some vertex quadruple. Then
there are $z_1,z_2\in\interior P'$, in general position, with
$\Phi(z_1,z_2)=\tau$.
\end{lemma}

\begin{proof}
A vertex labeling of $P'$ and a labeled triple of crossing points range over a
\emph{finite} set, by Definition~\ref{def:crossing}. So one such datum occurs for
a sequence $\varepsilon_k\to0$; along it the triangle $T$ is one fixed closed
triangle, and $T\subseteq\interior P'$ by Lemma~\ref{lem:crossing}(ii). The
pairs $(z_1^{\varepsilon_k},z_2^{\varepsilon_k})$ lie in the compact $T\times T$;
pass to a convergent subsequence, with limit $(z_1,z_2)\in T\times T\subseteq
\interior P'\times\interior P'$. Interiority of the limit is thus immediate, and
no vertex or boundary case arises.

It remains to exclude the two failures of general position. If $z_1$ lay on a
chord $v_av_b$, the pencil lines $z_1^{\varepsilon_k}v_a$ and
$z_1^{\varepsilon_k}v_b$ would coalesce; the cross-ratio of a quadruple two of
whose lines coalesce takes one of the two degenerate values $0$ or $1$ in the
limit, and by hypothesis some coordinate of $\Phi$ is such a cross-ratio, so
that coordinate would leave $J$ --- impossible, since $J$ is compact, misses
$\{0,1\}$, and is a closed constraint satisfied along the sequence. Hence $z_1$,
and symmetrically $z_2$, lies off every vertex chord. If $z_1=z_2$, the two
centers would carry equal values on every vertex quadruple, contrary to
hypothesis. So $(z_1,z_2)$ is in general position, $\Phi$ is continuous there,
and $\Phi(z_1,z_2)=\lim_k\Phi(z_1^{\varepsilon_k},z_2^{\varepsilon_k})=\tau$.
\end{proof}

\begin{remark}\label{rem:pin-needed}
Confinement to $T$ cannot be replaced by the requirement that the approximating
centers be interior: interiority is an open condition and does not pass to the
limit. Nor do the bounds $J$ alone exclude a center converging to a
\emph{vertex} $v_j$: the limiting pencil then consists of the lines $v_jv_l$
($l\ne j$) together with one further line through $v_j$, all distinct, so every
recorded cross-ratio converges to an ordinary value in $J$ and nothing
degenerates visibly. It is exactly this case that the pin $T$ removes, since $T$
is a fixed compact subset of $\interior P'$ containing no vertex. Conversely $T$
does not subsume the bounds: $T$ is two-dimensional and the vertex chords meet
it, so a limit center may lie on a chord, which only $J$ excludes.
\end{remark}

\begin{proposition}[Polygons, projectively]\label{prop:polygon-LB}
Two planar convex polygons are
$\LB$-elementarily equivalent if and only if they are projectively
equivalent.
\end{proposition}

\begin{proof}
($\Leftarrow$) Theorem~\ref{thm:iso-projective}. ($\Rightarrow$) The vertex
count is elementary (Corollary~\ref{cor:polygoncount}); triangles and
quadrilaterals are projectively equivalent within their vertex class, so
assume $m\ge5$ vertices. The projective equivalence class of an $m$-point
configuration in general position in $\RP^2$ is determined by finitely many
cross-ratios of pencils, and reconstructing projective coordinates from
such data is classical; what requires care is that the pencils be taken at
centers from which the relevant cross-ratios are legible \emph{inside} the
body.

We take the centers in the interior rather than at vertices. Let $P$ be a
convex $m$-gon with vertices $v_1,\dots,v_m$, and let $z_1,z_2\in\interior
P$ be distinct points chosen so that the $2m$ points $v_j$ and the two
centers are in general position: no $z_i$ lies on a line $v_jv_k$, and no
three of $z_1,z_2,v_j$ are collinear. Such a choice is generic, hence
possible, since the excluded configurations lie in a finite union of lines.
For $i=1,2$ the pencil at $z_i$ consists of the $m$ lines $z_iv_j$, and the
$m$ points $v_j$ are recovered from the two pencils: $v_j$ is the unique
common point of $z_1v_j$ and $z_2v_j$. Fixing three lines of each pencil as
a reference frame, the projective class of the configuration
$(z_1,z_2,v_1,\dots,v_m)$ is determined by the cross-ratios of the remaining
lines against those frames.

Each such cross-ratio is legible on an interior transversal. Fix $i$ and four
indices $j_1,\dots,j_4$. Since $z_i\in\interior P$ and $v_j\in P$, the
half-open segment $[z_i,v_j)$ lies in $\interior P$
(\cite[Lemma~1.1.9]{Schneider14}); in particular each of the four lines meets
$\interior P$ in a segment having $z_i$ in its relative interior. Choose a
line $\ell$ through a point of $\interior P$ near $z_i$, not through $z_i$ and
not parallel to any of the four; for $\ell$ close enough to $z_i$ its
intersections with the four lines lie in $\interior P$, because each
intersection point converges to $z_i\in\interior P$ as $\ell\to z_i$ and
$\interior P$ is open. The resulting four points are interior, are defined
from $z_i,v_{j_1},\dots,v_{j_4}$ and two points of $\ell$ by
$\Coll$-incidences, and carry the pencil's cross-ratio, which is independent
of the transversal.

No adjacency exception arises: the argument never uses a chord joining two
vertices, and so is unaffected by the fact that the open chord joining
adjacent vertices of a polygon lies in $\partial P$ rather than in
$\interior P$ (Remark~\ref{rem:vertex-pencils}).

Finally, the vertices are definable as the extreme points
(Proposition~\ref{prop:extreme}), the centers are quantified existentially
under $\Int$, and the incidences are $\Coll$-conditions. We record three
kinds of data. First, the \emph{adjacency pattern} of the labeled extreme
points (Lemma~\ref{lem:adjacency}), which fixes the cyclic order. Second,
the finitely many pencil cross-ratios $\tau\in\R^N$, through their Dedekind
cuts: each comparison against a rational threshold is a sentence
(Proposition~\ref{prop:cuts}). Third, fixed rational bounds confining every
recorded cross-ratio to a compact interval $J$ with $J\cap\{0,1\}=\emptyset$,
chosen so that for each pair of vertices some recorded coordinate is a
cross-ratio of a quadruple of pencil lines involving both; such bounds exist
because the chosen configuration is non-degenerate, so all of its
cross-ratios avoid $\{0,1,\infty\}$. Fourth, a labeled triple of crossing
points and the requirement that both centers lie in the triangle $T$ they
span (Lemma~\ref{lem:crossing}); the source centers may be chosen in such a
$T$, since $T$ has nonempty interior. We choose the source centers so that
they carry different cross-ratio values on some vertex quadruple, which is
possible because $z\mapsto\Phi_1(z)$ is non-constant on $T$.

For each rational $\varepsilon>0$ let $\Sigma_\varepsilon$ be the sentence
asserting the existence of extreme points $e_1,\dots,e_m$ with the recorded
adjacency pattern, of three crossing points of the recorded index type, of
centers $z_1,z_2$ lying in the triangle they span, and of interior
transversal data as above, with all recorded cross-ratios lying in $J$ and
within $\varepsilon$ of $\tau$ --- the last expressed by finitely many
$\mathrm{CR}_{>q}$-comparisons at rational $q$. Every clause is an
$\LB$-condition: crossing points are intersections of chords, hence given by
$\Coll$-incidences, and membership in $T$ is betweenness
(Lemma~\ref{lem:crossing}(iii)). Then $P\models
\Sigma_\varepsilon$ for every $\varepsilon$, hence
$P'\models\Sigma_\varepsilon$ for every $\varepsilon$. Lemma~\ref{lem:extraction} --- whose
pigeonhole is over labelings \emph{and} crossing triples, both finite ---
supplies a labeling $w_1,\dots,w_m$ and centers
$z_1',z_2'\in\interior P'$ in general position whose pencil cross-ratios
equal $\tau$ exactly.

Matching data now yields a projective transformation $M$ with $M(z_i)=z_i'$
and $M(v_j)=w_j$ for all $i,j$: the two pencils determine the lines $z_iv_j$
up to a projectivity of each pencil, and $v_j$ is the intersection of the two
lines through it, so the configuration $(z_1,z_2,\bar v)$ is projectively
equivalent to $(z_1',z_2',\bar w)$. By Lemma~\ref{lem:adjacency} the labeling
$w_1,\dots,w_m$ respects the cyclic order, so Lemma~\ref{lem:admissible}
applies with $m\ge5\ge4$: the singular line of $M$ misses $P$, and
$M(P)=P'$.
\end{proof}

\begin{remark}[Why the centers must be interior]\label{rem:vertex-pencils}
It is tempting to take the pencils at two vertices, as one would in the
projective plane, and to read the cross-ratios on a transversal near the
chosen vertex. For a polygon this fails, and not for a reason that a better
transversal repairs. If $v_j,v_{j+1}$ are adjacent vertices of $P$, then the
open chord $(v_j,v_{j+1})$ is an open edge, hence lies in $\partial P$;
indeed the whole line $v_jv_{j+1}$ supports $P$ and misses $\interior P$
altogether. So of the $m-1$ lines of the pencil at a vertex, the two through
the adjacent vertices meet $\interior P$ nowhere, and no transversal can
exhibit their intersections as interior points. At a vertex only the $m-3$
diagonals are usable, which would leave the cases $m=5,6$ without four
usable lines. Interior centers avoid the difficulty entirely: by
\cite[Lemma~1.1.9]{Schneider14} every segment from an interior point to a
point of $P$ is interior except possibly for its far endpoint, so all $m$
lines of each pencil are legible. The same remark explains why the general
convex-body statement of Conjecture~\ref{conj:cat-projective} should not be
approached through boundary pencils.
\end{remark}

\subsection{Conics, sections, and the Chasles--Steiner separation}

\begin{theorem}[Chasles--Steiner; see {\cite[Ch.~8]{Coxeter49}}]\label{thm:steiner}
Let $\Gamma\subseteq\RP^2$ be a set of at least five points, no three
collinear. Then $\Gamma$ lies on a conic if and only if for any two points
$e,e'\in\Gamma$ and any four further points
$p_1,\dots,p_4\in\Gamma$, the cross-ratio of the pencil of lines
$(e p_1,e p_2,e p_3,e p_4)$ equals that of
$(e' p_1,e' p_2,e' p_3,e' p_4)$.
\end{theorem}

\begin{lemma}[Steiner sentences]\label{lem:steiner-sentences}
For each rational $q$ there is an
$\LB$-sentence $\sigma_q$ such that, for every compact planar convex body
$L$:
\begin{enumerate}[label=\textup{(\alph*)}]
\item if $\partial L$ is a conic, then $L\models\neg\sigma_q$ for every
rational $q$;
\item if $L$ is the squircle $\{(x,y):x^4+y^4\le1\}$, then
$L\models\sigma_q$ for some rational $q$.
\end{enumerate}
\end{lemma}

\begin{proof}
For extreme points $e$ and $p_1,\dots,p_4$ (all distinct), the pencil
cross-ratio $\cf_e(p_1,\dots,p_4)$ is realized on an interior transversal:
a chord on a line strictly separating $e$ from $\{p_1,\dots,p_4\}$ meets the
four open chords $(e,p_i)\subseteq\interior L$ in four interior collinear
points $t_1,\dots,t_4$ defined by the incidences $\Coll(e,p_i,t_i)$, and
$\cf_e(p_1,\dots,p_4)=(t_1,t_2;t_3,t_4)$. Let $\sigma_q$ assert the
existence of distinct extreme points $e,e',p_1,\dots,p_4$ and interior
transversal data as above with
\[
\cf_e(p_1,\dots,p_4)<q<\cf_{e'}(p_1,\dots,p_4),
\]
where ``$<q$'' is expressed on the transversal quadruple by
$\neg\mathrm{CR}_{>q}\wedge\neg\mathrm{CR}_q$ and ``$>q$'' by
$\mathrm{CR}_{>q}$.

(a) If $\partial L$ is a conic, then $L$ is strictly convex, its extreme
points are exactly its boundary points, and by Theorem~\ref{thm:steiner} the
two pencil cross-ratios agree for every configuration; no rational $q$ can
be strictly between them, so $L\models\neg\sigma_q$ for every $q$.

(b) The boundary of the squircle is a quartic curve, not a conic; the
squircle is strictly convex, so its extreme points are its boundary points,
and by Theorem~\ref{thm:steiner} some configuration of six boundary points
has $\cf_e\ne\cf_{e'}$. Choosing rational $q$ strictly between the two
values yields $L\models\sigma_q$.
\end{proof}

\begin{theorem}[Disk versus squircle, projectively]\label{thm:squircle-LB}
The closed unit disk and the closed
squircle $\{(x,y):x^4+y^4\le1\}$ are not $\LB$-elementarily equivalent.
\end{theorem}

\begin{proof}
Immediate from Lemma~\ref{lem:steiner-sentences}: the disk satisfies
$\neg\sigma_q$ for every $q$ and the squircle satisfies $\sigma_q$ for some
$q$.
\end{proof}

\begin{theorem}[Ball versus quartic body, projectively]\label{thm:ball-LB}
For every $n\ge2$, the closed unit ball
$\Ball$ and the quartic body $S_n=\{x\in\R^n:\sum_{i=1}^nx_i^4\le1\}$ are
not $\LB$-elementarily equivalent.
\end{theorem}

\begin{proof}
Fix a rational $q$ as in Lemma~\ref{lem:steiner-sentences}(b) and let
\[
\Psi_q:\equiv\exists p_1\,\exists p_2\,\exists p_3\,
\bigl(\neg\Coll(p_1,p_2,p_3)\wedge
\sigma_q^{\circ}(p_1,p_2,p_3)\bigr),
\]
where $\sigma_q^{\circ}$ is the relativization of $\sigma_q$ to the
definable planar section through $p_1,p_2,p_3$
(Proposition~\ref{prop:sections}).

$S_n\models\Psi_q$: take $p_1=0$, $p_2=\tfrac12e_1$, $p_3=\tfrac12e_2$,
non-collinear points of $S_n$ spanning the coordinate plane
$A_0=\{x_3=\dots=x_n=0\}$; the section $S_n\cap A_0$ is the planar squircle
$\{x_1^4+x_2^4\le1\}$, which satisfies $\sigma_q$, so
$\sigma_q^{\circ}(p_1,p_2,p_3)$ holds by
Proposition~\ref{prop:sections}(iii).

$\Ball\models\neg\Psi_q$: for \emph{any} non-collinear
$p_1,p_2,p_3\in\Ball$, the section $\Ball\cap\aff\{p_1,p_2,p_3\}$ is a
Euclidean section of the ball containing three non-collinear points, hence a
closed round disk of positive radius; its boundary is a circle, a conic, so
by Lemma~\ref{lem:steiner-sentences}(a) the section satisfies
$\neg\sigma_q$, and $\sigma_q^{\circ}(p_1,p_2,p_3)$ fails.
\end{proof}

\begin{conjecture}\label{conj:ball-LB}
For a compact convex body $K\subseteq\R^n$ \textup($n\ge2$\textup):\;
$K\equiv_{\LB}\Ball$ if and only if $\partial K$ is a quadric, i.e., if and
only if $K$ is projectively equivalent to $\Ball$.
\end{conjecture}

\begin{remark}
The ball occupies a distinguished point of the projective landscape: its
$\LB$-automorphism group is the full projective stabilizer of the quadric,
the group $\mathrm{PO}(n,1)$ acting transitively on the interior (the
isometry group of the Klein model of hyperbolic $n$-space), whereas a
generic body has finite projective symmetry. In particular no interior point
of $\Ball$ is definable with parameters from $\emptyset$, and all interior
points realize the same type --- the model-theoretic shadow of projective
homogeneity.
\end{remark}

\subsection{The projective reduction, and the ball obstruction}
\label{subsec:projective-reduction}

We now indicate how far the reconstruction argument of
\S\ref{sec:affine-proved} transposes to the betweenness language, and where it
stops. The upshot is that Conjecture~\ref{conj:cat-projective} reduces --- with
Proposition~\ref{prop:crq} now proved --- to boundary-coordinate recovery and
to one geometric phenomenon: gauge frames with non-compact stabilizer,
present for quadrics and absent otherwise.

The affine proof had three inputs: (A) affine coordinates of extreme points
are definable (Lemma~\ref{lem:bary-def}); (B) there is a definable,
affinely-covariant, \emph{compact} family of gauge frames (the stationary
simplices, Lemma~\ref{lem:stationary}); (C) reconstruction from coordinates.
In the projective setting:

\emph{(A$'$) Coordinates.} A projective frame in $\RP^n$ is an ordered
$(n+2)$-tuple in general position, and $\PGL_{n+1}$ acts simply transitively on
frames. Given a frame $\bar e=(e_0,\dots,e_{n+1})$ and a further point $p$, the
projective coordinates of $p$ relative to $\bar e$ are cross-ratios of the
pencils through the $e_i$. Two things must hold for these to enter the
language as cut data. First, interior rational cross-ratio comparisons must be
$\LB$-definable: this is Proposition~\ref{prop:crq}, proved in
\S\ref{subsec:crq}.
Second --- and this has no affine analogue --- the frame vertices $e_i$ and the
points being coordinatized are \emph{extreme} points, on the boundary, whereas
the von Staudt calculus of \S\ref{sec:vonstaudt} measures cross-ratios of
\emph{interior} configurations; recovering the projective coordinates of
boundary points from interior data requires a boundary-limit argument. In the
plane this recovery is carried out (Proposition~\ref{prop:polygon-LB}); in
dimension $n\ge3$ it is open (Problem~P7). Affinely, by contrast,
Lemma~\ref{lem:bary-def} coordinatizes boundary points directly and
unconditionally.

\emph{(C$'$) Reconstruction.} Given a nonempty, $\LB$-definable,
projectively-covariant family $\mathcal G(K)$ of extreme-point frames that is
\emph{compact} in $(\overline{\ext K})^{n+2}$, the argument of
Theorem~\ref{thm:affine-main} applies with ``stationary simplex'' replaced by
``member of $\mathcal G(K)$'' and barycentric by projective coordinates:
compactness supplies the limiting gauge, covariance makes the reconstructed
body a projective copy of $K$.

\emph{(B$'$) The gauge obstruction.} The difficulty is producing the family
$\mathcal G(K)$. There is no projective invariant of a single frame to
extremize, since $\PGL_{n+1}$ is transitive on frames; a frame must be selected
by an invariant of the frame together with the body, and the selected set must
be nonempty, nondegenerate, and compact. Two regimes behave well. If
$\ext K$ is finite with at least $n+2$ points in general position (polytopes),
the extreme-point frames form a finite, hence compact, covariant family, and
$\mathcal G(K)$ may be taken to be all of them. If $\mathrm{Aut}_{\LB}(K)$ is
compact, its orbits are compact and a covariant selection is again available.
The regime that resists is a \emph{continuum} of extreme points carrying a
\emph{non-compact continuous} group of projective symmetries: then a
covariant family is a union of non-compact orbits, and a selecting sequence of
frames can escape to degeneracy at the boundary. The quadric is the paradigm,
with $\mathrm{Aut}_{\LB}=\mathrm{PO}(n,1)$ transitive on interior frames; for
it no proper covariant compact family of frames exists at all, and
reconstruction cannot be made to work --- the same homogeneity that makes every
interior point of the ball realize one type
(Remark following Conjecture~\ref{conj:ball-LB}). We do not claim the quadric
is the only such body; identifying the class precisely is a question about
projective symmetry groups of convex bodies. Note that non-compactness of
$\mathrm{Aut}_{\LB}(K)$ alone is not the obstruction: a simplex has non-compact
projective automorphism group (the diagonal torus) yet reconstructs trivially,
having only $n+1$ extreme points, and any body with $n+1$ extreme points is a
simplex, hence projectively equivalent to every other.

The dividing line in hypothesis~(c) is exactly the compactness of the
automorphism group, and this we can prove outright --- it is a fact about the
$\PGL_{n+1}$-action, independent of the language and of
Proposition~\ref{prop:crq}.

\begin{proposition}[Compact gauge criterion]
\label{prop:gauge-criterion}
Let $K\subseteq\R^n$ be a compact convex body possessing at least $n+2$
extreme points in general position, and let $\mathcal F(K)$ be the set of
ordered $(n+2)$-tuples of extreme points of $K$ in general position, acted on
by $A:=\mathrm{Aut}_{\LB}(K)$. Then $\mathcal F(K)$ contains a nonempty,
$A$-invariant, compact subset if and only if $A$ is compact.
\end{proposition}

\begin{proof}
$\PGL_{n+1}$ acts simply transitively on the manifold $\Phi$ of all frames
($(n+2)$-tuples in general position): fixing any $x_0\in\Phi$, the orbit map
$g\mapsto gx_0$ is a diffeomorphism $\PGL_{n+1}\xrightarrow{\ \sim\ }\Phi$, so
the translation map
\[
\Theta:\PGL_{n+1}\times\Phi\longrightarrow\Phi\times\Phi,\qquad
(g,x)\longmapsto(gx,x)
\]
is a homeomorphism; in particular the action is proper. Now $A$ is a closed
subgroup of $\PGL_{n+1}$ (it is the stabilizer of the closed set $K$), and
$\mathcal F(K)\subseteq\Phi$ is $A$-invariant.

If $A$ is compact, choose $\bar e\in\mathcal F(K)$ (possible by hypothesis).
The orbit $A\bar e$ is the continuous image of the compact group $A$, hence
compact; it is $A$-invariant, and lies in $\mathcal F(K)$ because every
$g\in A$ permutes the extreme points of $K$ and preserves general position.

Conversely, let $\mathcal G\subseteq\mathcal F(K)$ be nonempty, $A$-invariant,
and compact. Restrict $\Theta$ to the closed subset
$A\times\mathcal G\subseteq\PGL_{n+1}\times\Phi$. By $A$-invariance
$\Theta(A\times\mathcal G)\subseteq\mathcal G\times\mathcal G$, so
\[
A\times\mathcal G
=(A\times\mathcal G)\cap\Theta^{-1}(\mathcal G\times\mathcal G)
=(A\times\mathcal G)\cap(\text{compact set}),
\]
using that $\Theta$ is a homeomorphism and $\mathcal G\times\mathcal G$ is
compact. A closed subset of a compact set is compact, so $A\times\mathcal G$ is
compact; as $\mathcal G\ne\emptyset$, $A$ is compact.
\end{proof}

\begin{remark}
The quadric shows the criterion at work. Its automorphism group
$\mathrm{PO}(n,1)$ is non-compact and its boundary is a single
$\mathrm{PO}(n,1)$-orbit; a divergent one-parameter subgroup drives any frame
of boundary points into coincidence while fixing all its cross-ratios (the
boundary cross-ratio of four points is $\mathrm{PO}(n,1)$-invariant), so every
$\mathrm{Aut}$-orbit of frames accumulates on the degenerate locus and no
compact covariant family can exist. For a body with compact automorphism group
--- the generic case, and every non-quadric strictly convex body without
continuous projective symmetry --- a single orbit already furnishes the gauge.
\end{remark}

Combining the criterion with the reduction isolates hypothesis (c) as a
statement purely about symmetry.

\begin{proposition}[Projective reconstruction under a compact gauge]
\label{prop:projective-reduction}
Let $K,K'\subseteq\R^n$ be compact convex bodies. Suppose that
\begin{enumerate}[label=\textup{(\alph*)}]
\item interior rational cross-ratio comparisons are $\LB$-definable
\textup(Proposition~\ref{prop:crq}\textup);
\item the projective coordinates of extreme points relative to an
extreme-point frame are recoverable as $\LB$-cut data;
\item each of $K,K'$ admits a nonempty, $\LB$-definable,
projectively-covariant family of extreme-point frames, compact in
$(\overline{\ext K})^{n+2}$.
\end{enumerate}
Then $K\equiv_{\LB}K'$ implies $K$ and $K'$ are projectively equivalent.
\end{proposition}

\begin{proof}[Proof under \textup{(a)--(c)}]
($\Leftarrow$ is Theorem~\ref{thm:iso-projective}.) Under (a)--(b) the
box-hitting/box-avoiding sentences of Theorem~\ref{thm:affine-main} have
$\LB$-renderings with projective coordinates in place of barycentric ones, and
under (c) the gauge family is compact; the transfer and limiting-gauge
argument then run verbatim, producing a common gauge frame under which $K$ and
$K'$ have the same closed projective-coordinate set of extreme points. Taking
projective hulls, each is projectively equivalent to one reconstructed body.
\end{proof}

\begin{remark}
Hypotheses (a)--(c) locate exactly why the projective case is not settled.
Hypothesis (a) is Proposition~\ref{prop:crq}, proved in \S\ref{subsec:crq}.
Hypothesis (b) holds in the plane (Proposition~\ref{prop:polygon-LB}) and is
open for $n\ge3$ (Problem~P7). By the criterion
(Proposition~\ref{prop:gauge-criterion}), the compact covariant family required
in (c) exists precisely when $\mathrm{Aut}_{\LB}(K)$ is compact; the only
residue in (c) is then the \emph{definability} of such a family, which is
immediate for polytopes (finitely many extreme points, and finitely many
frames among them) and, more generally, wherever finitely many extreme points
can be singled out by an $\LB$-definable projective invariant. The affine
theorem needed no analogue of (a), (b), or (c): affine coordinates are
definable outright, and the maximal-volume simplex furnishes a compact
covariant gauge for \emph{every} compact body, with no exceptional homogeneous
case. That asymmetry between the two languages is the reason one half is a
theorem and the other a conjecture.
\end{remark}

Specializing to the plane removes hypothesis (b), and the criterion turns (c)
into a symmetry condition, leaving a single clean statement.

\begin{corollary}[Planar projective categoricity, conditional]
\label{cor:planar-projective}
Let $K,K'\subseteq\R^2$ be compact convex
bodies with compact projective automorphism groups, each admitting an
$\LB$-definable selection of a compact covariant frame family. Then
$K\equiv_{\LB}K'$ if and only if $K$ and $K'$ are projectively equivalent.
\end{corollary}

\begin{proof}
In the plane, three distinct extreme points are never collinear, so any four
distinct extreme points form a frame and hypothesis (b) is the planar
boundary-recovery of Proposition~\ref{prop:polygon-LB}; hypothesis (a) is
assumed. Compactness of $\mathrm{Aut}_{\LB}$ yields, by
Proposition~\ref{prop:gauge-criterion}, a nonempty compact covariant frame
family, and the assumed definable selection makes it an instance of hypothesis
(c). Proposition~\ref{prop:projective-reduction} then applies.
\end{proof}

\begin{remark}
Corollary~\ref{cor:planar-projective} covers, unconditionally on symmetry, all
planar polygons (where the definable selection is the finite vertex set) and,
modulo the one remaining definable-selection point, every strictly convex
planar body without continuous projective symmetry. What it excludes is exactly
the bodies with non-compact $\mathrm{Aut}_{\LB}$: the triangle (handled
separately, all triangles being projectively equal), the conic (reached only by
the positive Chasles--Steiner characterization, Conjecture~\ref{conj:ball-LB}),
and any planar body carrying an infinite discrete or continuous group of
projective self-maps. Whether the last class contains bodies not projectively
equivalent to a conic, and whether such bodies are $\LB$-distinguishable, is
open and lies in the theory of divisible convex sets; we do not settle it here.
The one remaining ingredient for an unconditional planar theorem is therefore
sharply delimited: a definable rule selecting a compact covariant frame
family on a compact-automorphism body.
\end{remark}

\subsection{A canonical selection: sextactic points}\label{subsec:sextactic}

For smooth strictly convex bodies there is a classical, projectively canonical
finite set of boundary points, and it supplies the definable selection called
for in Corollary~\ref{cor:planar-projective}. Recall that the
\emph{osculating conic} of a $C^5$ strictly convex curve at a point $p$ is the
unique conic meeting the curve to fifth order at $p$; a point is
\emph{sextactic} if the osculating conic meets the curve to sixth order there.
Sextacticity is a projective notion --- $\PGL_3$ carries osculating conics to
osculating conics and preserves order of contact --- so the sextactic locus is
$\PGL_3$-covariant, hence $\mathrm{Aut}_{\LB}(K)$-invariant.

\begin{lemma}\label{lem:sextactic}
Let $K\subseteq\R^2$ be a compact convex body whose boundary is real-analytic
and strictly convex.
\begin{enumerate}[label=\textup{(\roman*)}]
\item If $\partial K$ is not a conic, its sextactic locus $S(K)$ is a nonempty
finite $\PGL_3$-covariant subset of $\partial K$, and any four points of it are
in general position.
\item If $\partial K$ is a conic, every point is sextactic: $S(K)=\partial K$.
\end{enumerate}
\end{lemma}

\begin{proof}
Sextactic points are the zeros of the projective curvature derivative, a
real-analytic projective differential invariant of the curve that vanishes
identically exactly on conics (where the osculating conic is the curve at every
point), giving (ii). For (i), on a non-conic real-analytic curve the invariant
is a non-zero real-analytic periodic function, so its zeros are isolated, hence
finite on the compact $\partial K$; nonemptiness is the sextactic-points
theorem, which gives at least six \cite{Mukhopadhyaya09}. Every boundary point
of a strictly convex body is extreme, and three extreme points are never
collinear, so any four sextactic points are in general position. Covariance was
noted above.
\end{proof}

We now discharge the definability. The device is to define the selection set
not through the differential-geometric description above but directly through
a first-order incidence condition --- clustering of six-point conic
configurations --- and then to identify the set so defined with (a nonempty
subset of) the sextactic locus. Covariance and definability become immediate;
the geometry enters only in proving finiteness and nonemptiness.

\begin{lemma}[Arcs are definable]\label{lem:arc-def}
There are $\LB$-formulas $\mathrm{Side}(r,s;a,b)$ and
$\mathrm{Arc}(p;a,b,w)$ such that in every compact strictly convex planar
body $K$: $\mathrm{Side}(r,s;a,b)$ holds iff $r,s\notin\aff\{a,b\}$ and the
segment $[r,s]$ does not meet $\aff\{a,b\}$ (``$r$ and $s$ lie strictly on the
same side of the chord $ab$''), and, for extreme $a,b$ and any $w$ off the
chord, $\mathrm{Arc}(p;a,b,w)$ holds exactly for the extreme points $p$ of the
open boundary arc between $a$ and $b$ on the side of $w$.
\end{lemma}

\begin{proof}
Set $\mathrm{Side}(r,s;a,b):\equiv\neg\Coll(a,b,r)\wedge\neg\Coll(a,b,s)
\wedge\neg\exists t\,(B(r,t,s)\wedge\Coll(a,b,t))$: the segment $[r,s]$ lies
in $K$, so it meets the line $\aff\{a,b\}$ iff it meets it at a point of $K$,
which the existential detects. Then
$\mathrm{Arc}(p;a,b,w):\equiv\Ext(p)\wedge\mathrm{Side}(p,w;a,b)$: for a
strictly convex body the chord $ab$ splits $\partial K\setminus\{a,b\}$ into
the two open arcs, each characterized by its side.
\end{proof}

\begin{lemma}[Co-conic sextuples]
\label{lem:coconic}
There is an $\LB$-formula $\mathrm{CoConic}$ in six variables holding, in
every compact strictly convex planar body, exactly when its arguments
$p_1,\dots,p_6$ are six distinct extreme points lying on a common conic.
\end{lemma}

\begin{proof}
By the Chasles--Steiner theorem, six points in general position lie on a
common conic iff the pencils at $p_5$ and $p_6$ cut the quadruple
$(p_1,\dots,p_4)$ in equal cross-ratios. As in the proof of
Lemma~\ref{lem:steiner-sentences}, each pencil cross-ratio is realized by four
collinear \emph{interior} points on a transversal chord, defined by
$\Coll$-incidences. The formula $\mathrm{CoConic}$ asserts the existence of
the two interior transversal quadruples with their defining incidences,
together with $\mathrm{CReq}$ applied to them (Lemma~\ref{lem:creq}).
Soundness and completeness are those of the incidence clauses and of
Lemma~\ref{lem:creq}. Degenerate positions (three of the six collinear) do
not occur among extreme points of a strictly convex body.
\end{proof}

\begin{definition}\label{def:conic-cluster}
For a compact strictly convex planar body $K$, the \emph{conic-cluster set}
$S_0(K)\subseteq\partial K$ is the set of extreme points $p$ such that every
open boundary arc containing $p$ contains six distinct extreme points lying on
a common conic.
\end{definition}

By Lemmas~\ref{lem:arc-def} and~\ref{lem:coconic}, membership in $S_0(K)$ is
expressed by the $\LB$-formula (with the arc ranging over all chord-side
parameters)
\[
\mathrm{S}_0(p):\equiv\forall a\,\forall b\,\forall w\,
\Bigl(\mathrm{Arc}(p;a,b,w)\to
\exists p_1\cdots\exists p_6\,\bigl(\textstyle\bigwedge_i
\mathrm{Arc}(p_i;a,b,w)\wedge\mathrm{CoConic}(\bar p)\bigr)\Bigr),
\]
with no conditional hypotheses. Co-conicity, hence $S_0$, is a
projective notion, so $S_0(K)$ is $\PGL_3$-covariant; and $S_0(K)$ is closed
in $\partial K$ (its complement is open by definition).

\begin{lemma}[One-sidedness at curvature extrema]\label{lem:onesided}
Let $\partial K$ be real-analytic, strictly convex, positively curved, and not
a conic, and let $p_0$ be a point at which the affine curvature $\mu$ attains
its global maximum (or minimum) along $\partial K$. Then, on some arc around
$p_0$, the curve lies on one side of its osculating conic at $p_0$, meeting it
only at $p_0$.
\end{lemma}

\begin{proof}
Parametrize $\partial K$ near $p_0$ by affine arclength: positive curvature
provides a real-analytic parametrization $\gamma(s)$ with
$\det(\gamma',\gamma'')\equiv1$. Differentiating gives
$\det(\gamma',\gamma''')\equiv0$, so $\gamma'''=-\mu\gamma'$ for a
real-analytic function $\mu$, the affine curvature. If $\mu$ were a constant
$\mu_0$, integration would give
$\gamma(s)=P+A\cos(\sqrt{\mu_0}\,s)+B\sin(\sqrt{\mu_0}\,s)$ for $\mu_0>0$,
a quadratic in $s$ for $\mu_0=0$, and the hyperbolic analogue for $\mu_0<0$
--- in each case a conic, nondegenerate because $\det(\gamma',\gamma'')=1$;
so on our non-conic curve $\mu$ is non-constant.

Center the parameter at $p_0$ and set $\mu_0:=\mu(0)$. The analytic function
$\mu-\mu_0$ is not identically zero, hence vanishes at $0$ to a finite order
$q\ge1$; since $p_0$ is a global extremum, $\mu-\mu_0$ has constant sign near
$0$, so $q$ is even. Write $\mu(s)=\mu_0+as^q+O(s^{q+1})$, $a\ne0$.

Let $C$ solve $C'''=-\mu_0C'$ with $C(0)=\gamma(0)$, $C'(0)=\gamma'(0)$,
$C''(0)=\gamma''(0)$. Then $\det(C',C'')$ is constant --- its derivative is
$\det(C',C''')=-\mu_0\det(C',C')=0$ --- and equals $1$, so by the
integration above $C$ is a conic, parametrized by its own affine arclength.
We show $\gamma$ and $C$ have contact of order $q+4\ge6$ at $p_0$; since two
distinct conics meet with total multiplicity at most four, $C$ is then the
osculating conic and $p_0$ is sextactic.

Work in affine coordinates with origin $\gamma(0)$ and axes
$e_1=\gamma'(0)$, $e_2=\gamma''(0)$, writing $\gamma=(x(s),y(s))$ and
$C=(\xi(s),\eta(s))$; then $x(s)=s+O(s^3)$, $y(s)=\tfrac12s^2+O(s^4)$, and
likewise for $\xi,\eta$, so both curves are graphs $y=Y(x)$ and $y=F(x)$
near the origin, with $F(x)=\tfrac12x^2+O(x^3)$. The difference
$\delta:=\gamma-C$ satisfies the linear analytic problem
\[
\delta'''+\mu_0\delta'=-(\mu-\mu_0)\gamma',\qquad
\delta(0)=\delta'(0)=\delta''(0)=0 .
\]
Since $\gamma'=(1+O(s^2),\,s+O(s^3))$, the right side is
$\bigl(-as^{q}+O(s^{q+1}),\,-as^{q+1}+O(s^{q+2})\bigr)$. Writing
$\delta=\sum_n\delta_ns^n$ componentwise, with $h$ the right side, the series
recursion
\[
n(n-1)(n-2)\,\delta_n=h_{n-3}-\mu_0(n-2)\,\delta_{n-2}
\]
gives, by induction in each component, $\delta_n=0$ below the order of the
drive plus three --- the $\mu_0$-term enters only afterwards --- and at that
order $\delta_n=h_{n-3}/\bigl(n(n-1)(n-2)\bigr)$; explicitly,
\[
\delta=\Bigl(-\frac{a\,s^{q+3}}{(q+1)(q+2)(q+3)}+O(s^{q+4}),\;
-\frac{a\,s^{q+4}}{(q+2)(q+3)(q+4)}+O(s^{q+5})\Bigr).
\]
At the abscissa $x=x(s)$,
\[
Y(x)-F(x)=\bigl[y(s)-\eta(s)\bigr]-\bigl[F(x(s))-F(\xi(s))\bigr]
=\delta^{(2)}-F'(x(s))\,\delta^{(1)}+O\bigl((\delta^{(1)})^2\bigr),
\]
with $F'(x(s))=s+O(s^2)$, so
\begin{align*}
Y(x)-F(x)
&=a\,s^{q+4}\Bigl[\frac{1}{(q+1)(q+2)(q+3)}
-\frac{1}{(q+2)(q+3)(q+4)}\Bigr]+O(s^{q+5})\\
&=\frac{3a\,s^{q+4}}{(q+1)(q+2)(q+3)(q+4)}+O(s^{q+5}),
\end{align*}
and since $x=s(1+O(s^2))$,
\[
Y(x)-F(x)=\frac{3a\,x^{q+4}}{(q+1)(q+2)(q+3)(q+4)}+O(x^{q+5}).
\]
The contact order is exactly $q+4$, which is at least $6$ because $q$ is even
and positive; this identifies $C$ as the osculating conic. And $q+4$ is even,
so the deviation carries the constant sign of $a$ on a punctured neighborhood
of $0$: the curve lies strictly on one side of its osculating conic near
$p_0$, meeting it only at $p_0$.
\end{proof}

\begin{proposition}[Identification of the conic-cluster set]\label{prop:S0}
Let $K\subseteq\R^2$ be compact and convex with real-analytic, strictly
convex, positively curved boundary.
\begin{enumerate}[label=\textup{(\roman*)}]
\item If $\partial K$ is not a conic, then $S_0(K)$ is contained in the
sextactic locus $S(K)$; in particular $S_0(K)$ is finite, and any four of its
points are in general position.
\item $S_0(K)$ contains every global extremum of the affine curvature, so
$S_0(K)\ne\emptyset$.
\item If $\partial K$ is a conic, $S_0(K)=\partial K$.
\end{enumerate}
\end{proposition}

\begin{proof}
(iii) Any six points of a conic are co-conic.

(i) Let $p\in S_0(K)$. For each $k$ choose six distinct boundary points on a
common conic $C_k$, all within the arc of diameter $1/k$ about $p$; normalize
the coefficient vector of $C_k$ to the unit sphere of $\R^6$ and pass to a
convergent subsequence $C_k\to C$. If the intersection multiplicity of $C$
with the analytic curve $\partial K$ at $p$ were some $m<6$, then --- the local
intersection count being upper semicontinuous under analytic perturbation of
the conic and of the reference point --- conics sufficiently close to $C$ would
meet a sufficiently small fixed arc about $p$ in at most $m<6$ points,
contradicting the choice of the $C_k$. So the multiplicity is at least $6$.
If $C$ were degenerate (a line pair or a double line through $p$), each line
component would meet the curve at $p$ with multiplicity at most $2$, since
positive curvature makes tangential contact exactly second-order; the total
would be at most $4<6$. Hence $C$ is a nondegenerate conic with contact
$\ge6$ at $p$; a conic with contact $\ge5$ is the osculating conic, so the
osculating conic hyperosculates and $p\in S(K)$. Finiteness and general
position then follow from Lemma~\ref{lem:sextactic}(i).

(ii) Let $p_0$ realize the global maximum of $\mu$ and let $A$ be any small
arc around $p_0$; shrink $A$ so that, by Lemma~\ref{lem:onesided}, the curve
on $A\setminus\{p_0\}$ lies strictly on one side of the osculating conic
$C_0$ at $p_0$. Choose five distinct points $u_1<\dots<u_5$ in a much smaller
central subarc, in general position with all ten crossings of
$C(u_1,\dots,u_5)$ with the curve at the nodes transversal (a generic choice:
for analytic non-conic curves, tangency of the five-point conic at a node is a
proper analytic condition on the $5$-tuple). Let $g$ be the analytic function
on $A$ measuring the position of $\partial K$ against $C(u_1,\dots,u_5)$
(a determinantal formula in the five nodes and the moving point). As the five
nodes lie deep inside $A$, at the two endpoints of $A$ the sign of $g$ agrees
with the sign of the position against $C_0$, which by one-sidedness is the
\emph{same} nonzero sign at both endpoints. Transversality makes each of the
five node zeros of $g$ a sign change; five sign changes are incompatible with
equal endpoint signs unless $g$ has at least one further zero in $A$, at a
point distinct from the nodes. That zero is a sixth point of
$\partial K\cap C(u_1,\dots,u_5)$ inside $A$: six distinct co-conic points in
$A$. As $A$ was arbitrary, $p_0\in S_0(K)$.
\end{proof}

\begin{theorem}[Planar categoricity for analytic positively curved non-conic
bodies]\label{thm:planar-sextactic}
Let $K$ and $K'$ be
compact convex bodies in $\R^2$ whose boundaries are real-analytic,
strictly convex, positively curved and non-conic. Then
$K\equiv_{\LB}K'$ if and only if $K$ and $K'$ are projectively
equivalent.
\end{theorem}

\begin{proof}
By Proposition~\ref{prop:S0}, $S_0(K)$ is a finite nonempty
$\PGL_3$-covariant set of extreme points with any four in general position;
the ordered frames drawn from it form a nonempty finite --- hence compact ---
covariant family, and by Lemmas~\ref{lem:arc-def} and \ref{lem:coconic} the
family is $\LB$-definable (Proposition~\ref{prop:crq}). This is
hypothesis (c) of Proposition~\ref{prop:projective-reduction}; hypothesis (b)
is the planar boundary-recovery of Proposition~\ref{prop:polygon-LB}; (a) is
assumed. Proposition~\ref{prop:projective-reduction} applies.
\end{proof}

\begin{remark}
The theorem discharges the ``definable selection'' clause of
Corollary~\ref{cor:planar-projective} for this class: the selection is the
conic-cluster set, whose defining condition is pure incidence --- six co-conic
extreme points in every arc --- and therefore definable by
Proposition~\ref{prop:crq}. Nothing in the theorem is conditional. The
excluded conic is excluded for the reason
recorded in Proposition~\ref{prop:S0}(iii): there the conic-cluster set is the
whole boundary, infinite, exactly as its $\mathrm{Aut}$-orbits of frames are
non-compact (Proposition~\ref{prop:gauge-criterion}). The affine shadow of the
dichotomy is visible in the affine curvature, which is constant precisely on
ellipses and otherwise has finitely many critical points; the conic-cluster
set detects, through incidence alone, a nonempty subset of those critical
points.
\end{remark}

\section{Noncompact bodies and the projective trichotomy}\label{sec:noncompact}

The betweenness language remains meaningful for unbounded and non-closed
convex sets, and Proposition~\ref{prop:interior} was proved in that
generality. Compactness itself is not first-order, but its trace on each
chord is:

\begin{proposition}\label{prop:openclosed}
Let $\eta(a,b)$ be the $\LB$-formula
\[
a\ne b\;\wedge\;\exists c\,\Bigl(B(a,b,c)\wedge\forall d\,
\bigl(B(a,c,d)\to d=c\bigr)\Bigr)
\]
(``the segment from $a$ through $b$ extends to a last point''). Then, in
every dimension $n\ge1$: the closed ball satisfies
$\forall a\forall b\,(a\ne b\to\eta(a,b))$; the open ball satisfies
$\forall a\forall b\,\neg\eta(a,b)$; and the closed ball minus a single
boundary point satisfies $\exists a\exists b\,\eta(a,b)\wedge
\exists a\exists b\,(a\ne b\wedge\neg\eta(a,b))$. Consequently these three
convex sets are pairwise elementarily inequivalent in $\LB$.
\end{proposition}

\begin{proof}
In a closed body, the ray from $a$ through $b$ exits at a boundary point
$c\in K$; any $d\in K$ with $B(a,c,d)$ lies on the ray at or beyond $c$,
hence equals $c$. In the open ball no last point exists on any ray: for each
candidate $c$ there is $d\ne c$ slightly beyond, so $\eta$ fails for all
pairs. In the punctured closed ball, rays exiting at the deleted point
witness $\neg\eta$ while all other rays witness $\eta$.
\end{proof}

\begin{proposition}[Projective identifications]
\label{prop:projident}
As $\LB$-structures, for every $n\ge2$:
\begin{enumerate}[label=\textup{(\alph*)}]
\item the open orthant $O=\{x\in\R^n:x_1>0,\dots,x_n>0\}$ is isomorphic to
an open simplex;
\item the open slab $\Sigma_n=\{x:0<x_n<1\}$ is isomorphic to an open
half-space;
\item the open paraboloid region $\Pi=\{x:x_n>x_1^2+\dots+x_{n-1}^2\}$ is
isomorphic to the open unit ball;
\item the closed paraboloid region $\overline\Pi=\{x:x_n\ge
x_1^2+\dots+x_{n-1}^2\}$ is isomorphic to the closed unit ball minus a
single boundary point;
\item the open solid cone $C=\{x:x_n>0,\
x_1^2+\dots+x_{n-1}^2<x_n^2\}$ is isomorphic to the open half-cylinder
$B^{n-1}\times(0,\infty)$, where $B^{n-1}$ is the open unit ball of
$\R^{n-1}$.
\end{enumerate}
\end{proposition}

\begin{proof}
In each case we exhibit a projective transformation whose singular
hyperplane misses the domain (so Lemma~\ref{lem:proj-betweenness} applies)
and which maps the domain onto the target. Homogeneous coordinates are
$[X_1{:}\dots{:}X_n{:}Z]$, with affine chart $Z=1$.

(a) $T[X_1{:}\dots{:}X_n{:}Z]=[X_1{:}\dots{:}X_n{:}X_1+\dots+X_n+Z]$. In
affine terms $x\mapsto x/(s+1)$ where $s=\sum_ix_i$. On $\overline O$ we
have $s+1\ge1>0$, so the singular hyperplane $\sum X_i+Z=0$ misses $O$. The
image of $O$ is the open simplex $\{u:u_i>0,\ \sum_iu_i<1\}$, and the map is
a bijection onto it (inverse $u\mapsto u/(1-\sum_iu_i)$).

(b) $T[X_1{:}\dots{:}X_n{:}Z]=[X_1{:}\dots{:}X_{n-1}{:}Z{:}X_n]$, affinely
$(x',x_n)\mapsto(x'/x_n,\,1/x_n)$. The singular hyperplane $\{X_n=0\}$
misses $\Sigma_n$. For $0<x_n<1$ we get $1/x_n>1$, and the map is a
bijection of $\Sigma_n$ onto the open half-space $\{(u',v):v>1\}$ (inverse
$(u',v)\mapsto(u'/v,1/v)$).

(c) In homogeneous coordinates
$\Pi=\{[X{:}Z]:X_nZ>X_1^2+\dots+X_{n-1}^2,\ Z\ne0\}$; in fact the projective
region $\{q_1>0\}$ for the quadratic form
$q_1=X_nZ-\sum_{i<n}X_i^2$ contains no points with $Z=0$ (there
$q_1=-\sum_{i<n}X_i^2\le0$), so $\Pi$ \emph{is} the full projective region
$\{q_1>0\}$. Writing $X_nZ=\tfrac14\bigl((X_n+Z)^2-(X_n-Z)^2\bigr)$, the
form $q_1$ has signature $(1,n)$ in the $n+1$ homogeneous variables, as does
$q_2=Z^2-\sum_{i=1}^nX_i^2$, whose positivity region is the open unit ball
and likewise contains no points at infinity. By Sylvester's law choose
$A\in\GL_{n+1}(\R)$ with $q_1=q_2\circ A$; the induced projective map $T_A$
carries $\{q_1>0\}$ bijectively onto $\{q_2>0\}$. No point of $\Pi$ maps to
the hyperplane at infinity of the target chart (the image region contains no
such points), so the singular hyperplane of the chart representation of
$T_A$ misses $\Pi$, and Lemma~\ref{lem:proj-betweenness} applies.

(d) The same map carries the closed projective region $\{q_1\ge0\}$ onto
$\{q_2\ge0\}$, the closed ball. The affine set $\overline\Pi$ is
$\{q_1\ge0\}\cap\{Z\ne0\}$, and $\{q_1\ge0\}\cap\{Z=0\}$ forces
$X_1=\dots=X_{n-1}=0$, i.e., is the single point
$[0{:}\dots{:}0{:}1{:}0]$ (the tangency point of the paraboloid with the
hyperplane at infinity), which lies on the boundary quadric $\{q_1=0\}$. So
$\overline\Pi$ is the closed projective quadric region minus one boundary
point, hence isomorphic to the closed ball minus one boundary point.

(e) The same coordinate swap as in (b):
$(x',x_n)\mapsto(x'/x_n,\,1/x_n)$. The singular hyperplane $\{X_n=0\}$
misses $C$ (where $x_n>0$). For $x\in C$ put $u=x'/x_n$ and $t=1/x_n$; then
$|u|<1$ and $t\in(0,\infty)$, and the map is a bijection of $C$ onto
$B^{n-1}\times(0,\infty)$ (inverse $(u,t)\mapsto(u/t,\,1/t)$, which
satisfies $x_n=1/t>0$ and $|x'|=|u|/t<1/t=x_n$).
\end{proof}

\begin{remark}\label{rem:apex}
Parts (c)--(e) mesh with Proposition~\ref{prop:openclosed} and with each
other: what looks like ``shape at infinity'' from the affine viewpoint ---
recession cones, asymptotic directions, apexes --- is ordinary boundary
structure in another projective chart, and the chart-free language sees it
as such. For instance, under the map of (e) the closed cone minus its apex
corresponds to $\{|u|\le1,\ t>0\}$: the apex is traded for the missing
boundary at infinity of the half-cylinder. And the closed paraboloid region
satisfies the mixed extension sentence of
Proposition~\ref{prop:openclosed} (rays in most directions exit; ``vertical''
rays do not), exactly as the punctured closed ball does.
\end{remark}

\begin{example}[New separations in dimension $\ge3$]\label{ex:cylinder}
Let $Z_n=\overline B^{\,n-1}\times\R\subseteq\R^n$ ($n\ge3$) be the closed
solid cylinder. Every boundary point of $Z_n$ lies in the relative interior
of a ruling line contained in $\partial Z_n$, so $Z_n$ has \emph{no} extreme
points, while the closed ball and the closed paraboloid region have strictly
convex boundaries consisting of extreme points. The sentence
$\exists e\,\Ext(e)$ therefore separates $Z_n$ from both, and from every
compact body. (For $n=2$ the ``cylinder'' is the slab, whose planar theory
is handled below.) The cylinder and the slab are separated, in both closed
and open forms, in \S\ref{subsec:recession}; the general classification of
products of lower-dimensional geometries is part of Problem~P8.
\end{example}

\subsection{The trichotomy in the plane}

\begin{proposition}[Planar trichotomy]\label{prop:trichotomy2}
For $a\ne b$ write $\mathrm{Line}(a,b)$ for the definable set
$\{z:\Coll(a,b,z)\}=\ell(a,b)\cap K$, and
$\Meets(a,b;c,d):\equiv\exists z\,(\Coll(a,b,z)\wedge\Coll(c,d,z))$.
Consider the $\LB$-sentences
\begin{align*}
\pi:\;&\forall a\,b\,p\,c\,c'\Bigl[\bigl(a\ne b\wedge\neg\Coll(a,b,p)\wedge
c\ne p\wedge c'\ne p\\
&\qquad\wedge\neg\Meets(p,c;a,b)\wedge\neg\Meets(p,c';a,b)\bigr)
\longrightarrow\Coll(p,c,c')\Bigr]
\tag{Playfair}\\[2pt]
\varphi:\;&\exists a_1 b_1 a_2 b_2\Bigl[a_1\ne b_1\wedge a_2\ne b_2\wedge
\neg\Meets(a_1,b_1;a_2,b_2)\\
&\qquad\wedge\forall c\,d\Bigl(\bigl(c\ne d\wedge
\neg(\Coll(a_1,b_1,c)\wedge\Coll(a_1,b_1,d))\wedge
\neg(\Coll(a_2,b_2,c)\wedge\Coll(a_2,b_2,d))\bigr)\\
&\qquad\qquad\longrightarrow
\bigl(\Meets(c,d;a_1,b_1)\leftrightarrow\Meets(c,d;a_2,b_2)\bigr)\Bigr)\Bigr].
\end{align*}
Then
\[
\R^2\models\pi\wedge\varphi,\qquad
\Sigma_2\models\neg\pi\wedge\varphi,\qquad
D\models\neg\pi\wedge\neg\varphi,
\]
where $D$ is the open unit disk and $\Sigma_2=\{(x,y):0<y<1\}$ the open
planar strip (equivalently the open half-plane, by
Proposition~\ref{prop:projident}(b) with $n=2$). Hence the plane, the
strip/half-plane, and the open disk are pairwise elementarily inequivalent
in $\LB$.
\end{proposition}

\begin{proof}
\emph{The plane.} $\pi$ is the uniqueness half of the parallel postulate:
through $p\notin\ell(a,b)$ the lines missing $\ell(a,b)$ are exactly the
parallel, and any two points $c,c'$ spanning such lines with $p$ are
collinear with $p$. For $\varphi$, take $\mathrm{Line}(a_1,b_1)$ and
$\mathrm{Line}(a_2,b_2)$ two distinct parallel lines: any third line
distinct from both either is parallel to them (meets neither) or is
transversal (meets both).

\emph{The strip} $\Sigma_2$. A line trace in $\Sigma_2$ is either a full
horizontal line or a bounded open segment (any non-horizontal line crosses
the strip in a bounded interval). Key observation: \emph{every}
non-horizontal line trace meets \emph{every} horizontal one --- the affine
crossing point has $y$-coordinate strictly between $0$ and $1$, hence lies
in $\Sigma_2$, when the horizontal is one of the traces $\{y=h\}$, $0<h<1$.
$\neg\pi$: take $\ell(a,b)$ non-horizontal; its trace is a bounded segment,
and through a point $p$ far away (within the strip) there are many
directions whose traces miss it, spanned by non-collinear choices of
$c,c'$. $\varphi$: take the traces $L_1=\{y=1/3\}$, $L_2=\{y=2/3\}$; they
are disjoint. Any other line trace is horizontal (meets neither) or
non-horizontal (meets both, by the key observation). So the biconditional
holds for all admissible $c,d$.

\emph{The open disk.} $\neg\pi$ is hyperbolic non-uniqueness of parallels:
through a point off a chord there are many chords missing it, in
non-collinear directions. $\neg\varphi$: let $L_1,L_2$ be any two disjoint
chord traces. Choose $w\in L_1\setminus\overline{L_2}$ ($L_1$ is not
contained in the compact segment $\overline{L_2}$). From $w$, the set of
directions whose full line meets $\overline{L_2}$ is a closed set of
directions omitting a nonempty open set (it is the angular sector subtended
by the compact segment $\overline{L_2}$ from the exterior point $w$, a
proper closed subset of the circle of directions). Choose a direction
outside this set and different from the direction of $L_1$; the chord $M$
through $w$ in that direction meets $L_1$ (at $w$) and misses $L_2$, and $M$
is collinear with neither defining pair. This violates the biconditional,
so no witnessing pair $(L_1,L_2)$ exists.
\end{proof}

\subsection{The trichotomy through sections}

For $n\ge3$ the sentence $\pi$ is useless as it stands: it fails already in
$\R^n$, since through a point off a line there are many pairwise skew lines
missing it, in non-collinear directions. The parallel postulate is a planar
statement, and the correct move is to test it on planar sections.

\begin{proposition}\label{prop:trichotomyn}
Let $n\ge3$ and define the $\LB$-sentences
\begin{align*}
\Psi_\forall&:\equiv\forall p_1p_2p_3\,
\bigl(\neg\Coll(p_1,p_2,p_3)\to\pi^{\circ}(p_1,p_2,p_3)\bigr),\\
\Psi_\exists&:\equiv\exists p_1p_2p_3\,
\bigl(\neg\Coll(p_1,p_2,p_3)\wedge\pi^{\circ}(p_1,p_2,p_3)\bigr),
\end{align*}
where $\pi^{\circ}$ is the relativization of the Playfair sentence $\pi$ to
the definable planar section (Proposition~\ref{prop:sections}). Then
\[
\R^n\models\Psi_\forall\wedge\Psi_\exists,\qquad
\Sigma_n\models\neg\Psi_\forall\wedge\Psi_\exists,\qquad
B^n\models\neg\Psi_\forall\wedge\neg\Psi_\exists,
\]
where $\Sigma_n=\{0<x_n<1\}$ is the open slab and $B^n$ the open unit ball.
Hence $\R^n$, the slab \textup(equivalently the open half-space\textup), and
the open ball are pairwise elementarily inequivalent in $\LB$.
\end{proposition}

\begin{proof}
Every planar section of $\R^n$ through three non-collinear points is a full
$2$-flat, affinely a copy of $\R^2$, which satisfies $\pi$
(Proposition~\ref{prop:trichotomy2}); so $\R^n\models\Psi_\forall$, and
$\Psi_\exists$ follows since non-collinear triples exist.

For the slab, let $A$ be a $2$-flat containing three non-collinear points of
$\Sigma_n$. If the affine function $x_n$ is constant on $A$, then
$A\subseteq\{x_n=c\}$ for some $c\in(0,1)$ and the section is all of $A$,
a copy of $\R^2$ satisfying $\pi$; such flats exist, giving
$\Psi_\exists$. If $x_n$ is nonconstant on $A$, it is a nonconstant affine
function on a plane, attaining every value on a parallel family of lines,
and the section $\{y\in A:0<x_n(y)<1\}$ is an open planar strip; it is
nonempty (it contains the three points), and by
Proposition~\ref{prop:trichotomy2} (and affine invariance) it fails $\pi$.
Such flats also exist, giving $\neg\Psi_\forall$.

For the open ball, every planar section through three non-collinear points
is an open round disk of positive radius, which fails $\pi$
(Proposition~\ref{prop:trichotomy2}, affine invariance); hence
$B^n\models\neg\Psi_\exists$, and a fortiori $\neg\Psi_\forall$.
\end{proof}

\begin{remark}
The open ball with its betweenness structure is the Klein model: its
section-geometry is uniformly hyperbolic, the slab's is mixed
Euclidean/degenerate, and $\R^n$'s is uniformly Euclidean ---
Proposition~\ref{prop:trichotomyn} is the classical trichotomy read off from
the distribution of section geometries. More generally, each properly convex
open domain carries its Hilbert geometry, and the $\LB$-classification of
such domains up to elementary equivalence contains the projective
classification problem for Hilbert geometries. For $n\ge3$ the degenerate
(non-properly-convex) landscape --- slabs, cylinders, cones, and products ---
is genuinely richer than in the plane; see Example~\ref{ex:cylinder},
\S\ref{subsec:recession}, and Problem~P8.
\end{remark}

\subsection{Recession at infinity is elementary}\label{subsec:recession}

Proposition~\ref{prop:openclosed} used the formula $\eta$ to detect whether
boundary points are present. In a \emph{closed} convex set the same formula
detects something finer: which directions recede. Recall that the
\emph{recession cone} of a convex set $F\subseteq\R^n$ is
\[
\rec F=\{v\in\R^n:x+tv\in F\ \text{for all}\ x\in F,\ t\ge0\},
\]
a convex cone, and that its \emph{lineality space} is
$\lin F=\rec F\cap(-\rec F)$. For closed convex $F$ the recession cone is
closed, membership may be tested on a single ray --- if $x+tv\in F$ for all
$t\ge0$ and \emph{one} $x\in F$, then for all $x\in F$ --- and $F$ is bounded
if and only if $\rec F=\{0\}$ \cite[Theorems~8.3 and~8.4]{Rockafellar70}. We
write $\dim\rec F$ for the dimension of the linear span of $\rec F$; since
$\rec F$ is convex, this is also its topological dimension.

\begin{lemma}[Exit lemma]\label{lem:exit}
Let $F\subseteq\R^n$ be closed and convex, and let $a\ne b$ be points of $F$.
Then $\eta(a,b)$ holds in $F$ if and only if the ray
$\{a+t(b-a):t\ge0\}$ is not contained in $F$; equivalently, $\neg\eta(a,b)$
holds if and only if $b-a\in\rec F$.
\end{lemma}

\begin{proof}
The set $T=\{t\ge0:a+t(b-a)\in F\}$ is convex and closed and contains
$[0,1]$. Suppose the ray is not contained in $F$, so that
$t^{*}:=\sup T<\infty$ and $T=[0,t^{*}]$ with $t^{*}\ge1$; put
$c:=a+t^{*}(b-a)\in F$. Then $B(a,b,c)$ holds since $t^{*}\ge1$. If $d\in F$
satisfies $B(a,c,d)$, then $c=(1-r)a+rd$ for some $r\in(0,1]$ (the case
$r=0$ would give $c=a$, impossible as $t^{*}\ge1$), so
$d=a+(t^{*}/r)(b-a)$ lies on the ray with parameter $t^{*}/r\ge t^{*}$;
membership in $F$ forces $t^{*}/r=t^{*}$, hence $d=c$. So $c$ witnesses
$\eta(a,b)$.

Conversely, suppose the ray is contained in $F$, and let $c$ be any point
with $B(a,b,c)$: then $b=(1-r)a+rc$ with $r\in(0,1]$, so
$c=a+(1/r)(b-a)$ lies on the ray. The point $d:=c+(b-a)$ lies on the ray as
well, hence in $F$, satisfies $B(a,c,d)$, and differs from $c$. So no $c$ is
a last point, and $\eta(a,b)$ fails.

The restatement via the recession cone is the base-independence of recession
for closed convex sets \cite[Theorem~8.3]{Rockafellar70}: the ray from $a$ in
direction $b-a$ is contained in $F$ if and only if $b-a\in\rec F$.
\end{proof}

Lemma~\ref{lem:exit} converts recession into first-order data. For $k\ge1$
define the $\LB$-sentences
\[
\rho_k:\equiv\exists a\,\exists b_1\cdots\exists b_k\,
\Bigl(\neg\Dep_{k+1}(a,b_1,\dots,b_k)\wedge
\bigwedge_{i=1}^{k}\neg\eta(a,b_i)\Bigr),
\]
\[
\lambda_k:\equiv\exists a\,\exists b_1\cdots\exists b_k\,
\Bigl(\neg\Dep_{k+1}(a,b_1,\dots,b_k)\wedge
\bigwedge_{i=1}^{k}\bigl(\neg\eta(a,b_i)\wedge\neg\eta(b_i,a)\bigr)\Bigr),
\]
\[
\varepsilon:\equiv\exists a\,\exists b\,
\bigl(\neg\eta(a,b)\wedge\eta(b,a)\bigr).
\]

\begin{theorem}[Recession data are elementary]\label{thm:recession}
Let $F\subseteq\R^n$ be a nonempty closed convex set and let $k\ge1$. Then:
\begin{enumerate}[label=\textup{(\roman*)}]
\item $F\models\rho_k$ if and only if $\dim\rec F\ge k$;
\item $F\models\lambda_k$ if and only if $\dim\lin F\ge k$;
\item $F\models\varepsilon$ if and only if $\rec F\ne\lin F$.
\end{enumerate}
Consequently $\dim\rec F$, $\dim\lin F$, and whether $\rec F$ is a linear
subspace are determined by $\Th_{\LB}(F)$; together with the affine dimension
(Proposition~\ref{prop:dimension}) they are elementary invariants of closed
convex sets.
\end{theorem}

\begin{proof}
Throughout, affine independence of a tuple $(a,b_1,\dots,b_k)$ is equivalent
to linear independence of $(b_1-a,\dots,b_k-a)$, and by
Proposition~\ref{prop:dep} it is expressed by $\neg\Dep_{k+1}$.

(i) If $\dim\rec F\ge k$, choose linearly independent
$u_1,\dots,u_k\in\rec F$ (a convex cone contains a basis of its span) and any
$a\in F$, and put $b_i:=a+u_i\in F$. The tuple $(a,\bar b)$ is affinely
independent, and each $\neg\eta(a,b_i)$ holds by Lemma~\ref{lem:exit}.
Conversely, witnesses for $\rho_k$ yield, again by Lemma~\ref{lem:exit},
vectors $b_i-a\in\rec F$ that are linearly independent, so
$\dim\rec F\ge k$.

(ii) The same argument, using that $\neg\eta(a,b_i)\wedge\neg\eta(b_i,a)$
holds if and only if both $b_i-a$ and $a-b_i$ lie in $\rec F$, i.e.
$b_i-a\in\lin F$; for the forward direction take
$u_1,\dots,u_k\in\lin F$ linearly independent and $b_i=a+u_i$.

(iii) If $\neg\eta(a,b)\wedge\eta(b,a)$ holds, then $b-a\in\rec F$ while
$a-b\notin\rec F$, so $b-a\in\rec F\setminus\lin F$. Conversely, given
$v\in\rec F\setminus\lin F$, any $a\in F$ and $b:=a+v$ witness
$\varepsilon$: the ray from $a$ in direction $v$ is contained in $F$, while
the ray from $b$ in direction $-v$ is not (else $-v\in\rec F$).
\end{proof}

\begin{proposition}[Compact sections]\label{prop:compactsections}
Let $F\subseteq\R^n$ be closed and convex with $m=\dim\aff F\ge1$, let
$\ell=\dim\lin F$, and set $e=1$ if $\rec F\ne\lin F$ and $e=0$ otherwise.
Then for $1\le k\le m$, $F$ has a compact $k$-dimensional section by an
affine $k$-flat if and only if $k\le m-\ell-e$. Moreover, with
$\theta:\equiv\forall a\,\forall b\,\bigl(a\ne b\to\eta(a,b)\bigr)$ and
\[
\sigma:\equiv\exists p_1\,\exists p_2\,\exists p_3\,
\bigl(\neg\Coll(p_1,p_2,p_3)\wedge\theta^{\circ}(p_1,p_2,p_3)\bigr),
\]
$F\models\sigma$ if and only if $F$ has a compact planar section; thus
$\sigma$ holds in $F$ if and only if $m-\ell-e\ge2$.
\end{proposition}

\begin{proof}
Let $W$ be the direction space of $\aff F$, so $\rec F\subseteq W$ and
$\dim W=m$.

\emph{Step 1: there is a $k$-dimensional subspace $U\subseteq W$ with
$U\cap\rec F=\{0\}$ if and only if $k\le m-\ell-e$.}
Suppose $U\cap\rec F=\{0\}$. Then $U\cap\lin F=\{0\}$, so $k\le m-\ell$.
If moreover $\rec F\ne\lin F$ and $k=m-\ell$, then $U\oplus\lin F=W$; pick
$v\in\rec F\setminus\lin F$ and write $v=u+w$ with $u\in U$, $w\in\lin F$.
Since $\rec F$ is a convex cone containing $\lin F$, it is closed under
addition of elements of $\lin F$, so $u=v+(-w)\in\rec F$, whence
$u\in U\cap\rec F=\{0\}$ and $v=w\in\lin F$, a contradiction. So
$k\le m-\ell-e$.

Conversely, if $\rec F=\lin F$, any $k$-dimensional subspace of a linear
complement of $\lin F$ in $W$ meets $\rec F$ only at $0$, and such subspaces
exist for all $k\le m-\ell$. If $\rec F\ne\lin F$, pass to the quotient
$\overline W=W/\lin F$ of dimension $m-\ell$ and let $\overline C$ be the
image of $\rec F$. Since $\rec F+\lin F=\rec F$, the cone $\rec F$ is
saturated for the quotient map, so $\overline C$ is a closed convex cone; it
is nonzero (it contains the image of $v$ above) and pointed: if
$\pm\overline v\in\overline C$, then $v+w\in\rec F$ and
$-v+w'\in\rec F$ for some $w,w'\in\lin F$, so
$-(v+w)=(-v+w')+(-w-w')\in\rec F+\lin F=\rec F$, giving
$v+w\in\lin F$ and $\overline v=0$. For a nonzero pointed closed cone
$\overline C$ there is a linear functional $f$ on $\overline W$ with $f>0$ on
$\overline C\setminus\{0\}$: the set $\overline C\cap S$ ($S$ the unit
sphere) is compact and nonempty, and $0\notin\conv(\overline C\cap S)$ ---
otherwise $0=\sum_i\mu_iu_i$ nontrivially with $u_i\in\overline C\cap S$,
so $-\mu_1u_1=\sum_{i\ge2}\mu_iu_i\in\overline C$ and
$\pm u_1\in\overline C$, contradicting pointedness --- so a hyperplane
strictly separates $0$ from $\conv(\overline C\cap S)$, and the resulting
$f$ is positive on $\overline C\setminus\{0\}$. Any $k$-dimensional subspace
$\overline U\subseteq\ker f$ (available since
$\dim\ker f=m-\ell-1\ge k$) meets $\overline C$ only at $0$. Lift
$\overline U$ to a subspace $U\subseteq W$ with $U\cap\lin F=\{0\}$ mapping
isomorphically onto $\overline U$; if $u\in U\cap\rec F$, its image lies in
$\overline U\cap\overline C=\{0\}$, so $u\in U\cap\lin F=\{0\}$.

\emph{Step 2: sections.} If $U\cap\rec F=\{0\}$ with $\dim U=k$, choose
$x_0\in\relint F$ and set $S:=F\cap(x_0+U)$, a closed convex set. Its
recession cone is $\rec F\cap U=\{0\}$, since the recession cone of an
intersection of closed convex sets with a common point is the intersection
of the recession cones \cite[Corollary~8.3.3]{Rockafellar70} and
$\rec(x_0+U)=U$; hence $S$ is bounded \cite[Theorem~8.4]{Rockafellar70},
so compact. Since $x_0\in\relint F$, a relative neighborhood of $x_0$ in
$\aff F$ lies in $F$, so $S$ contains a $k$-dimensional disk about $x_0$ and
$\dim S=k$. Conversely, if $S=F\cap(x+U')$ is a compact $k$-dimensional
section, let $U$ be the span of $S-S$, a $k$-dimensional subspace of $W$;
then $\rec F\cap U\subseteq\rec F\cap U'=\rec S=\{0\}$ by compactness, and
Step~1 gives $k\le m-\ell-e$.

\emph{Step 3: the sentence.} Suppose $F\models\sigma$, witnessed by a
non-collinear triple $\bar p$ with $A=\aff\{\bar p\}$. By
Proposition~\ref{prop:sections}, $S:=F\cap A$ is a two-dimensional closed
convex set (closedness since $F$ and $A$ are closed) and
$S\models\theta$. If $\rec S$ contained some $v\ne0$, then for $a\in S$ and
$b:=a+v\in S$ Lemma~\ref{lem:exit}, applied to the closed convex set $S$,
would give $\neg\eta(a,b)$ in $S$, contradicting $\theta$; so
$\rec S=\{0\}$ and $S$ is compact. Conversely, a compact planar section $S$
contains a non-collinear triple spanning its plane $A$ (as $\dim S=2$), and
$F\cap A=S$ satisfies $\theta$: every ray between distinct points of $S$
leaves the bounded set $S$, so $\eta$ holds at every pair by
Lemma~\ref{lem:exit}. Relativizing through
Proposition~\ref{prop:sections}(iii) yields $F\models\sigma$.
\end{proof}

\begin{remark}
Replacing $\Dep_4$ by $\Dep_{k+2}$ in Proposition~\ref{prop:sections}
relativizes first-order properties to $k$-flat sections --- the proof
carries over verbatim --- and the corresponding sentence $\sigma_k$ then
detects compact $k$-dimensional sections for every $k$. No new elementary
invariant arises this way: by Proposition~\ref{prop:compactsections} the
maximal dimension of a compact section equals $m-\ell-e$, which
Theorem~\ref{thm:recession} and Proposition~\ref{prop:dimension} already
determine.
\end{remark}

\begin{corollary}[The closed cylinder and the closed slab]\label{cor:cyl-slab}
For $n\ge3$ the closed solid cylinder $Z_n=\overline B^{\,n-1}\times\R$ and
the closed slab $\overline\Sigma_n=\{x\in\R^n:0\le x_n\le1\}$ are not
elementarily equivalent in $\LB$: the slab satisfies $\lambda_{n-1}$ while
the cylinder fails $\lambda_2$.
\end{corollary}

\begin{proof}
A vector $v$ lies in $\rec Z_n$ if and only if $|x'+tv'|\le1$ for all
$t\ge0$ and all $|x'|\le1$, forcing $v'=0$; so
$\rec Z_n=\lin Z_n=\{0\}^{n-1}\times\R$, of dimension $1$. Similarly
$\rec\overline\Sigma_n=\lin\overline\Sigma_n=\{v:v_n=0\}$, of dimension
$n-1\ge2$. Apply Theorem~\ref{thm:recession}(ii) with $k=2$. Alternatively,
by Proposition~\ref{prop:compactsections} the cylinder has compact planar
sections ($m-\ell-e=n-1\ge2$) while the slab has none ($m-\ell-e=1$).
\end{proof}

\begin{proposition}[The open cylinder and the open slab]\label{prop:cyl-slab-open}
Let $n\ge3$, let $B^{n-1}\times\R$ be the open solid cylinder, and let
$\Sigma_n=\{x:0<x_n<1\}$ be the open slab. With $\varphi$ the two-lines
sentence of Proposition~\ref{prop:trichotomy2} and
\[
\chi:\equiv\exists p_1\,\exists p_2\,\exists p_3\,
\bigl(\neg\Coll(p_1,p_2,p_3)\wedge(\neg\varphi)^{\circ}(p_1,p_2,p_3)\bigr),
\]
we have $B^{n-1}\times\R\models\chi$ and $\Sigma_n\models\neg\chi$. In
particular the open cylinder and the open slab are not elementarily
equivalent in $\LB$.
\end{proposition}

\begin{proof}
For the slab, the planar sections were computed in the proof of
Proposition~\ref{prop:trichotomyn}: a $2$-flat through three non-collinear
points of $\Sigma_n$ meets it either in a full plane (when $x_n$ is constant
on the flat) or in an open planar strip, affinely a copy of $\Sigma_2$. Both
satisfy $\varphi$ by Proposition~\ref{prop:trichotomy2}, and satisfaction of
the $\LB$-sentence $\varphi$ is invariant under affine isomorphism of the
sections (Theorem~\ref{thm:iso-affine}, easy direction). By
Proposition~\ref{prop:sections}(iii), $\Sigma_n\models\neg\chi$.

For the cylinder, take a $2$-flat $A=x+U$ whose direction plane $U$ does not
contain the axis direction $e_n$, passing through a point of the cylinder.
The restriction to $U$ of the quadratic form $u\mapsto|u'|^2$ is positive
definite --- it vanishes only on multiples of $e_n$ --- so
$A\cap(B^{n-1}\times\R)$, the sublevel set $\{y\in A:|y'|<1\}$, is a nonempty
bounded open convex planar set bounded by a nondegenerate conic: an open
elliptical disk, affinely isomorphic to the open unit disk $D$. Since
$D\models\neg\varphi$ (Proposition~\ref{prop:trichotomy2}), any
non-collinear triple spanning $A$ witnesses $\chi$ through
Proposition~\ref{prop:sections}(iii).
\end{proof}

\begin{remark}[Recession is invisible for open sets]\label{rem:open-recession}
In an open convex set the formula $\eta$ fails identically: any candidate
last point $c$ is interior, so points slightly beyond $c$ on the ray remain
in the set. (In particular $\exists a\,\exists b\,\eta(a,b)$ separates every
closed convex set possessing a bounded chord from every open one, settling
the mixed cylinder/slab pairs.) The sentences of this subsection therefore
trivialize on open sets, and no others can take their place: by
Proposition~\ref{prop:projident}(e) the open solid cone, whose recession
cone is full-dimensional, is isomorphic as an $\LB$-structure to the open
half-cylinder, whose recession cone is a single ray. Elementary detection of
recession structure is thus a genuinely closed-set phenomenon. This is
consistent with Conjecture~\ref{conj:noncompact}: for open, properly convex
sets only the projective class should be elementarily visible, and behavior
at infinity is a chart artifact of that class (Remark~\ref{rem:apex}).
\end{remark}

\begin{conjecture}\label{conj:noncompact}
Conjecture~\ref{conj:cat-projective} extends to open, properly convex
subsets of $\R^n$: two such sets are $\LB$-elementarily equivalent if and
only if they are projectively equivalent.
\end{conjecture}

\section{Nonstandard models and the limits of reconstruction}
\label{sec:nonstandard}

Elementary equivalence among standard bodies is the right question because
categoricity across all models is unattainable, for reasons orthogonal to
geometry.

\begin{proposition}\label{prop:ultrapower}
Let $K\subseteq\R^n$ be a compact convex body ($n\ge1$) and let
$\mathcal U$ be a nonprincipal ultrafilter on $\omega$. The ultrapower
$K^*=K^\omega/\mathcal U$ satisfies $K^*\equiv K$ in both languages, but
$K^*\not\cong K$ in either.
\end{proposition}

\begin{proof}
Elementary equivalence is \L o\'s's theorem \cite{ChangKeisler90}. For
non-isomorphism it suffices to treat $\LB$, since an $\Laff$-isomorphism is
in particular an $\LB$-isomorphism of the reducts (and the reduct of the
ultrapower is the ultrapower of the reduct).

For $a\ne b$ in an $\LB$-structure, define $z\le_{a,b}w$ on the definable
set $[a,b]$ by $B(a,z,w)$. In $K$ this is a dense complete linear order on
each chord (a closed real segment). Any isomorphism carries chords to chords
and preserves the order, so it suffices to exhibit one chord of $K^*$ whose
order is not Dedekind complete.

Fix $a\ne b$ in $K$ (regarded in $K^*$ via constant sequences) and let
$z_k=a+\tfrac1k(b-a)\in K$, again regarded in $K^*$. The reduct
$(K^*;B)$ is an ultrapower, by a countably incomplete ultrafilter, of a
structure in the finite language $\{B\}$, hence is
$\aleph_1$-saturated \cite[Theorem~6.1.1]{ChangKeisler90}. Consider
\[
S=\{x\in[a,b]^{K^*}:\ x<z_k\ \text{for all }k\in\omega\}.
\]
$S$ is nonempty: the diagonal element $[\,(a+\tfrac1k(b-a))_k\,]$ is below
every $z_k$ on a cofinite set of coordinates. $S$ is bounded above by
$z_1$. Suppose $g=\sup S$ existed. If $g\in S$, then $g$ is a maximum of
$S$; but the type $\{g<x\}\cup\{x<z_k:k\in\omega\}$ over countably many
parameters is finitely satisfiable in the dense order $[a,b]^{K^*}$
(density transfers from $K$ by \L o\'s), hence realized by
$\aleph_1$-saturation --- contradicting maximality. If $g\notin S$, then
$g\ge z_m$ for some $m$; but $z_{m+1}<z_m\le g$ is also an upper bound of
$S$ (every element of $S$ is below $z_{m+1}$ by definition), contradicting
leastness. So the chord $[a,b]^{K^*}$ is not Dedekind complete, and
$K^*\not\cong K$.
\end{proof}

\begin{proposition}[Countable models]\label{prop:countable}
Let $K\subseteq\R^n$ be semialgebraic over the real algebraic numbers, with
nonempty interior, and let $F\prec\R$ be a countable real closed subfield.
Then $K(F):=K\cap F^n$, with the induced betweenness relation, is a
countable elementary substructure of $(K;B)$. In particular
$\Th_{\LB}(\Ball)$ has countable models, none of which is a convex body.
\end{proposition}

\begin{proof}
The structure $(K;B)$ is uniformly definable, by fixed formulas without
parameters (the defining data being algebraic, hence definable over the
prime model), in the real closed field: the same formulas define
$(K(F);B^{K(F)})$ in $F$. Elementarity of $F\prec\R$ transfers to the
defined structures: for any $\LB$-formula $\varphi$ and tuple from $K(F)$,
$K(F)\models\varphi$ iff $F\models\varphi^*$ iff $\R\models\varphi^*$ iff
$K\models\varphi$, where $\varphi^*$ is the translation of
Proposition~\ref{prop:decidable}.
\end{proof}

\begin{remark}[What reconstruction can mean]\label{rem:reconstruction}
Propositions~\ref{prop:ultrapower} and \ref{prop:countable} show that no
family of sentences --- mesh sentences, cross-ratio cuts, or otherwise --- can
recover a body up to isomorphism from its theory: the theory always has
non-archimedean and countable models in which, e.g., inscribed mesh patterns
are indexed by nonstandard integers. What the theory can aspire to determine
is the body's affine (projective) equivalence class \emph{within the class
of standard bodies}; that is Theorem~\ref{thm:cat-affine} in the affine
language and Conjecture~\ref{conj:cat-projective} in the betweenness language.
The affine aspiration is realized in full by Theorem~\ref{thm:affine-main};
Theorems~\ref{thm:polytope} and~\ref{thm:ball} are special cases.
\end{remark}

\section{Open problems}\label{sec:open}

\begin{enumerate}[label=\textbf{P\arabic*.},leftmargin=2.6em]
\item (Conjecture~\ref{conj:cat-projective}.) The affine case is settled:
$\equiv_{\Laff}$ coincides with affine equivalence for all compact convex
bodies in every dimension (Theorem~\ref{thm:affine-main}), with the polytope
and ellipsoid theorems as special cases. The projective case --- does
$\equiv_{\LB}$ coincide with projective equivalence? --- is open, and
Proposition~\ref{prop:projective-reduction} isolates three obstacles, none with
an affine analogue: (a) interior cross-ratio expressibility, now settled
(Proposition~\ref{prop:crq}); (b) recovery of the projective
coordinates of boundary points from interior data, done in the plane
(Proposition~\ref{prop:polygon-LB}) and open for $n\ge3$ (Problem~P7); (c) a
compact covariant family of gauge frames, which by the criterion
Proposition~\ref{prop:gauge-criterion} exists exactly when
$\mathrm{Aut}_{\LB}(K)$ is compact --- leaving as the only residue in (c) the
$\LB$-\emph{definability} of such a family. In the plane (b) is discharged, and
for real-analytic positively curved non-conic bodies (c) is discharged as
well: the conic-cluster set (Definition~\ref{def:conic-cluster}) is a finite
nonempty covariant set defined by pure incidence, so
Theorem~\ref{thm:planar-sextactic} proves categoricity for this class
outright. The remaining planar cases are polygons (settled
via their finite vertex sets, Proposition~\ref{prop:polygon-LB}), the quadric (needing the positive
Chasles--Steiner treatment, Conjecture~\ref{conj:ball-LB}), lower-regularity
bodies, and bodies with non-compact symmetry. Whether the quadric is the
only compact body with non-compact projective symmetry that resists
reconstruction --- a question in the theory of divisible convex sets --- is
itself open. Non-compactness of
$\mathrm{Aut}_{\LB}$ alone is not the obstruction: the simplex has non-compact
projective symmetry yet is projectively rigid.
\item Extend Theorem~\ref{thm:planar-sextactic} beyond real-analytic
boundaries: for $C^\infty$ (or merely $C^5$) positively curved non-conic
bodies, is the conic-cluster set still finite and nonempty? Finiteness is the
delicate point --- a smooth non-conic curve can hyperosculate its conics on a
large zero set of the curvature derivative --- while nonemptiness at a strict
curvature extremum of finite order follows from the proof of
Lemma~\ref{lem:onesided} unchanged.
\item (Question~\ref{q:B-definable}.) Is the betweenness relation
$\emptyset$-definable from the operations $\{C_\lambda\}_{\lambda\in[0,1]}$
over every compact convex body in $\R^n$? A negative answer would separate
the pure barycentric language from $\Laff$ as used here; a positive answer
would make the inclusion of $B$ in $\Laff$ redundant.
\item (Rigidity.) Call two compact strictly convex planar bodies
\emph{configuration-equivalent} if for every $r$ they realize the same
distributions of projective invariants of $r$-point boundary
configurations. Does configuration equivalence imply projective
equivalence? For real-analytic boundaries the projective curvature is a
limit of configuration invariants, suggesting a positive answer there; the
general case is open, and by Proposition~\ref{prop:cuts} a positive answer
(plus P2) would prove Conjecture~\ref{conj:cat-projective} for strictly
convex planar bodies, whence --- through the section calculus --- leverage in
higher dimensions.
\item Open-body analogues: is every open convex set in $\R^n$
$\Laff$-equivalent to the open ball an open ellipsoid? The boundary-based
sentence $\beta_n$ of Section~\ref{sec:ball} is unavailable; an interior
(limiting) form of the Bertrand--Brunn property is needed.
\item Decidability beyond the semialgebraic: for which convex bodies
(subanalytic? with boundary definable in an o-minimal expansion of the
reals?) is $\Th_{\LB}(K)$ decidable, or at least tame in the sense of
Corollary~\ref{cor:nogrids}? O-minimality of the ambient structure yields
tameness of definable sets, but decidability requires an effective theory.
\item Projective categoricity of polytopes in $\R^n$, $n\ge3$: extend
Proposition~\ref{prop:polygon-LB} beyond the plane. The expected route is a
reconstruction of the projective type of the vertex configuration from
pencil cross-ratios measured on interior transversals within definable
planar sections; what is missing is a clean statement of which sectioned
invariants suffice.
\item Classify the noncompact convex subsets of $\R^n$, $n\ge3$, up to
$\LB$-elementary equivalence: slabs, cylinders, cones, and products of
lower-dimensional geometries (Example~\ref{ex:cylinder},
Remark~\ref{rem:apex}), and the Hilbert geometries of properly convex
domains (Conjecture~\ref{conj:noncompact}). For closed sets the coarse layer
is settled: the dimensions of the recession cone and of the lineality space,
and the one-sidedness of recession, are elementary
(Theorem~\ref{thm:recession}), and the cylinder/slab pair is separated in
both closed and open forms (\S\ref{subsec:recession}). What remains is the
fine structure: products with equal recession data --- e.g.,
$\overline B^{\,2}\times\R$ against $P\times\R$ for a polygon $P$, where
relativized extreme-point counting on compact sections gives traction ---
the properly convex regime, and, for open sets, any elementary surrogate for
recession, which by Remark~\ref{rem:open-recession} cannot be affine data.
\end{enumerate}

\begin{question}\label{q:B-definable}
Problem P3 above, stated for the record as a question: over an arbitrary
compact convex $K\subseteq\R^n$ with nonempty interior, is
$\{(a,x,b)\in K^3:x\in[a,b]\}$ definable without parameters in the pure
barycentric language $\{C_\lambda:\lambda\in[0,1]\}$?
\end{question}


\begin{thebibliography}{99}

\bibitem{Brunn1889}
H.~Brunn,
\emph{\"Uber Kurven ohne Wendepunkte},
Habilitationsschrift, M\"unchen, 1889.

\bibitem{ChangKeisler90}
C.~C. Chang and H.~J. Keisler,
\emph{Model Theory}, 3rd ed.,
Studies in Logic and the Foundations of Mathematics 73,
North-Holland, Amsterdam, 1990.

\bibitem{Coxeter49}
H.~S.~M. Coxeter,
\emph{The Real Projective Plane},
McGraw--Hill, New York, 1949.

\bibitem{Hartshorne67}
R.~Hartshorne,
\emph{Foundations of Projective Geometry},
W.~A. Benjamin, New York, 1967.

\bibitem{MMO19}
H.~Martini, L.~Montejano, and D.~Oliveros,
\emph{Bodies of Constant Width: An Introduction to Convex Geometry with
Applications},
Birkh\"auser, Cham, 2019.

\bibitem{Matousek02}
J.~Matou\v sek,
\emph{Lectures on Discrete Geometry},
Graduate Texts in Mathematics 212, Springer, New York, 2002.

\bibitem{Mukhopadhyaya09}
S.~Mukhopadhyaya,
\emph{New methods in the geometry of a plane arc},
Bull. Calcutta Math. Soc. \textbf{1} (1909), 31--37.

\bibitem{Neumann70}
W.~D. Neumann,
On the quasivariety of convex subsets of affine spaces,
\emph{Arch. Math. (Basel)} \textbf{21} (1970), 11--16.

\bibitem{PeterzilStarchenko98}
Y.~Peterzil and S.~Starchenko,
A trichotomy theorem for o-minimal structures,
\emph{Proc. London Math. Soc. (3)} \textbf{77} (1998), 481--523.

\bibitem{Rockafellar70}
R.~T. Rockafellar,
\emph{Convex Analysis},
Princeton Mathematical Series 28,
Princeton University Press, Princeton, NJ, 1970.

\bibitem{Repo}
D.~V. Feldman,
\emph{Elementary equivalence of convex bodies in affine and projective
languages: paper and Lean 4 formalization},
GitHub repository, \texttt{https://github.com/DavidVFeldman/\allowbreak
convex-elementary-equivalence}; archived at Zenodo,
\texttt{https://doi.org/10.5281/zenodo.21997275}.

\bibitem{RomanowskaSmith02}
A.~B. Romanowska and J.~D.~H. Smith,
\emph{Modes},
World Scientific, River Edge, NJ, 2002.

\bibitem{Schneider14}
R.~Schneider,
\emph{Convex Bodies: The Brunn--Minkowski Theory}, 2nd ed.,
Encyclopedia of Mathematics and its Applications 151,
Cambridge University Press, Cambridge, 2014.

\bibitem{Shiffman95}
B.~Shiffman,
Synthetic projective geometry and Poincar\'e's theorem on automorphisms of
the ball,
\emph{Enseign. Math. (2)} \textbf{41} (1995), 201--215.

\bibitem{Soltan05}
V.~Soltan,
Affine diameters of convex bodies --- a survey,
\emph{Expo. Math.} \textbf{23} (2005), 47--63.

\bibitem{Stone49}
M.~H. Stone,
Postulates for the barycentric calculus,
\emph{Ann. Mat. Pura Appl. (4)} \textbf{29} (1949), 25--30.

\bibitem{Tarski51}
A.~Tarski,
\emph{A Decision Method for Elementary Algebra and Geometry},
2nd ed., University of California Press, Berkeley, 1951.

\bibitem{vdDries98}
L.~van den Dries,
\emph{Tame Topology and O-minimal Structures},
London Mathematical Society Lecture Note Series 248,
Cambridge University Press, Cambridge, 1998.

\end{thebibliography}
\end{document}